\documentclass[11pt,           
english                       
]{article}

\usepackage{stmaryrd}
\usepackage[english,german,strings]{babel}     
\usepackage{amsmath}           
\usepackage[utf8]{inputenc}    
\usepackage[T1]{fontenc}       
\usepackage{longtable}         
\usepackage{exscale}           
\usepackage[final]{graphicx}   
\usepackage[sort]{cite}        
\usepackage{array}             
\usepackage{fancyhdr}          
\usepackage[a4paper]{geometry} 
\usepackage[multiuser]{fixme}  
\usepackage{xspace}            
\usepackage{tikz}              
\usepackage{appendix}
\usepackage[expansion=false    
]{microtype}        
\usepackage[nottoc]{tocbibind}

\usepackage[pdftex]{hyperref}
\usepackage[numbers]{natbib}
\usepackage{tikz-cd}

\usepackage{nchairx} 
\newcommand{\blup}{\mathrm{Blup}}
\newcommand{\Partial}[1]{\frac{\partial}{\partial #1}}
\newcommand{\homology}{\mathrm{H}}
\newcommand{\Homology}[3]{\homology^{#2}(#1#3)}
\newcommand{\laover}{\Rightarrow}
\newcommand{\blowdown}[3]{#1 \colon \blup(#2,#3) \to #2 }
\newcommand{\foliation}[1]{\mathcal{#1}}
\newcommand{\filtration}[1]{\mathcal{#1}}
\newcommand{\vfield}{\mathfrak{X}}
\newcommand{\DNC}{\mathrm{DNC}}
\newcommand{\vanishing}{\mathcal{I}}
\makeatletter
\def\Ddots{\mathinner{\mkern1mu\raise\p@
\vbox{\kern7\p@\hbox{.}}\mkern2mu
\raise4\p@\hbox{.}\mkern2mu\raise7\p@\hbox{.}\mkern1mu}}
\makeatother

\begin{document}
	\selectlanguage{english}

\title{The blow-down map in Lie algebroid cohomology}
\date{}
\author{
 Andreas Sch\"{u}\ss ler
}

\newcommand{\Addresses}{{
  \bigskip
  \footnotesize

 A.~Sch\"{u}\ss ler, \textsc{KU Leuven, Department of Mathematics, Celestijnenlaan 200B box 2400, BE-3001 Leuven, Belgium,} \textit{email:} \texttt{a.schuessler@math.ru.nl}
}}

\maketitle
\abstract{
We study the blow-down map in cohomology in the context of real projective blowups of Lie algebroids. 
Using the blow-down map in cohomology we compute the Lie algebroid cohomology of the blowup of transversals of arbitrary codimension, generalising the Mazzeo-Melrose theorem on b-cohomology. 
To prove the result we develop a Gysin sequence for Lie algebroids. 
As another example we use the developed tools to compute the Lie algebroid cohomology of the action Lie algebroid $\liealg{so}(3)\ltimes \field R^3$, a result known in Poisson geometry literature. 
Moreover, we use similar techniques to compute the de Rham cohomology of real projective blowups. 
}
\vskip12pt

\setcounter{tocdepth}{2}
\tableofcontents

\section{Introduction}
	\renewcommand{\pr}{\mathrm{pr}}
	Blowup constructions are well-known in algebraic geometry (see e.g.\ \cite{hartshorne:1977a}). The idea is to replace a point or subvariety by all lines normal to it. A famous result by Hironaka \cite{hironaka:1964a,hironaka:1964b} states that in characteristic zero, one can always desingularise algebraic varieties, i.e.\ obtain a smooth variety, by a sequence of blowups. In the context of smooth manifolds, singularities often arise from additional geometric structures (a Poisson structure, a foliation, etc.). 
    
    There are various kinds of blowup constructions for smooth manifolds known in the literature. Spherical blowup, which replaces a submanifold by the sphere bundle of its normal bundle, has been used in Melrose's b-calculus (see \cite{melrose:1993a}, in particular Chapter 4). Furthermore, it has been used to resolve singularities in the context of Lie groupoids \cite{albin.melrose:2011a,nistor:2019a}.

    In this paper we focus on real projective blowups of smooth manifolds, in which the submanifold is replaced by the projectivisation  of its normal bundle. Of particular interest are blowups of Lie algebroids, for which the construction is given in \cite{debord.Skandalis:2017a} (\cite{gualtieri.li:2014a} for a base of codimension 1), see also \cite{obster:2021a}. Given a Lie algebroid $A\laover M$ and a Lie subalgebroid $B\laover N$ over a submanifold $N\subseteq M$, blowing up yields a new Lie algebroid $ \blup(A,B) $ over the base $ \blup(M,N) $ together with a Lie algebroid morphism back to $A$, the blow-down map,
    \begin{equation*}
	\begin{tikzpicture}[baseline=(current bounding box.center)]
	\node (1) at (0,1.5) {$ \blup(A,B) $};
	\node (q) at (0,0) {$ \blup(M,N) $};
	\node (2) at (2.3,1.5) {$ A $};
	\node (w) at (2.3,0) {$ M $};
	
	\path[->]
	(1) edge node[above]{$ p_A $} (2)
	(q) edge node[above]{$ p $} (w);

    \draw[-Implies,double equal sign distance]
	(1) -- (q);
	\draw[-Implies,double equal sign distance]
	(2) -- (w);
	
	\end{tikzpicture}
	\end{equation*}
    Note that there exists a corresponding construction for Lie groupoids \cite{debord.Skandalis:2017a}, see also \cite{obster:2021a}. The projective blowup construction has been used to desingularise proper groupoids \cite{posthuma.tang.wang:2017a, wang:2018a}, and in the context of Lie algebroids blowing up has been shown to recover interesting Lie algebroids (a construction called \textit{elementary modification} in \cite[Definition 2.11]{gualtieri.li:2014a}), like log- or scattering tangent bundles (see also \cite{klaasse:2017a, lanius:2021a}). 
    
    In this paper we seek to gain insights into Lie algebroid cohomologies using real projective blowups. The blow-down map induces a map between the respective cohomology groups
	\begin{equation}\label{eq:p_A_in_cohomology}
	p_A^\ast\colon \Homology{A}{\bullet}{}\to \Homology{\blup(A,B)}{\bullet}{}.
	\end{equation}
	It is this blow-down map in cohomology \eqref{eq:p_A_in_cohomology} we aim to study in this paper: Since it relates two cohomology groups, understanding of the map allows to gain information about one cohomology group via the other. This works in two ways: We either start with a known Lie algebroid and are interested in the cohomology of the blowup, or it is the cohomology of the original Lie algebroid we are interested in. In this case, we choose the Lie subalgebroid in such a way that the cohomology of the blowup is easier to compute than that of $A$, e.g.\ such that $\blup(A,B)$ is a regular Lie algebroid. Ultimately, we hope that the combination of desingularisation using blowups and understanding of the blow-down map in cohomology leads to a new way of computing Lie algebroid cohomologies.
	
	In general, we can fit the blow-down map into the short exact sequence
	\begin{equation*}
	\begin{tikzpicture}[baseline=(current bounding box.center)]
	\node (1) at (0,0) {$ 0 $};
	\node (2) at (1.6,0) {$ \Omega^\bullet(A) $};
	\node (3) at (4.6,0) {$ \Omega^\bullet(\blup(A,B)) $};
	\node (4) at (7.9,0) {$ \frac{\Omega^\bullet(\blup(A,B))}{p_A^\ast\Omega^\bullet(A)} $};
	\node (5) at (10,0) {$ 0 $};
	
	\path[->]
	(1) edge node[]{$  $} (2)
	(2) edge node[above]{$ p_A^\ast $} (3)
	(3) edge node[above]{$  $} (4)
	(4) edge node[]{$  $} (5);
	
	\end{tikzpicture}
	\end{equation*}
	of cochain complexes, resulting in a long exact sequence in cohomology \eqref{intro:eq:les_for_blup} with the following property, see Theorem~\ref{theorem:flatiso_and_consequences} for details.
	\begin{itheorem}[The blow-down map in cohomology]\label{intro:theorem:p_A_in_cohomology}
		In the long exact sequence
		\begin{equation}\label{intro:eq:les_for_blup}
		\begin{tikzpicture}[baseline=(current bounding box.center)]
		\node (1) at (0,0) {$ \dots $};
		\node (2) at (1.9,0) {$ \Homology{A}{\bullet}{} $};
		\node (3) at (4.8,0) {$ \Homology{\blup(A,B)}{\bullet}{} $};
		\node (4) at (8.6,0) {$ \homology^\bullet\left(\frac{\Omega^\bullet(\blup(A,B))}{p_A^\ast\Omega^\bullet(A)}\right)$};
		\node (5) at (11.6,0) {$ \Homology{A}{\bullet+1}{} $};
		\node (6) at (13.7,0) {$ \dots $};
		
		\path[->]
		(1) edge node[]{$  $} (2)
		(2) edge node[above]{$ p_A^\ast $} (3)
		(3) edge node[above]{$ f $} (4)
		(4) edge node[above]{$  $} (5)
		(5) edge node[]{$  $} (6);
		
		\end{tikzpicture}
		\end{equation}
		the map $ f $ only depends on local data around $ N $ and $ \field P=p^{-1}(N) $. 
	\end{itheorem}

The locality in Theorem \ref{intro:theorem:p_A_in_cohomology} can be understood in the following way: If $ \iota\colon U\hookrightarrow \blup(M,N) $ is an open neighbourhood of $ \field P $ then $ f = f_U\circ\homology(\iota^\ast) $. Here, $ \homology(\iota^\ast)\colon \homology^\bullet(\blup(A,B))\to \homology^\bullet(\blup(A,B)_U) $ denotes the map induced in cohomology by the inclusion and $ f_U $ is the map corresponding to $ f $ in the long exact sequence~\eqref{intro:eq:les_for_blup} for $ A_{p(U)} $, see again Theorem~\ref{theorem:flatiso_and_consequences} for more details. Another way to view the locality of the blow-down map in cohomology is by use of jets of forms. For a Lie algebroid $A\laover M$ and a submanifold $N\subseteq M$ we write 
\begin{equation}
    \mathscr{J}_N^\infty\Omega^\bullet(A)=\Omega^\bullet(A) / \bigcap_{k\in \field N}\vanishing_N^k\Omega^\bullet(A)
\end{equation}
for $\infty$-jets of forms of $A$ along $N$, where $\vanishing_N$ denotes the vanishing ideal of $N$. Then the locality in Theorem~\ref{intro:theorem:p_A_in_cohomology} originates from an isomorphism
\begin{equation}
    \homology^\bullet\left(\frac{\Omega^\bullet(\blup(A,B))}{p_A^\ast\Omega^\bullet(A)}\right) \simeq \homology^\bullet\left(\frac{ \mathscr{J}_{\field P}^\infty\Omega^\bullet(\blup(A,B))}{p_A^\ast\mathscr{J}_N^\infty\Omega^\bullet(A)}\right).
\end{equation}

As a first application of Theorem \ref{intro:theorem:p_A_in_cohomology} we express the cohomology of the blowup of a transversal in terms of $\Homology{A}{\bullet}{}$. The proof uses that transversals admit a simple normal form \cite{BursztynLimaMeinrenken2019}.

\begin{itheorem}[The blowup of transversals]\label{intro:theorem:transversals}
	Let $ \iota\colon N\hookrightarrow M $ be a closed transversal of $ A\laover M $ and denote the projection of the projective bundle $\field P\subseteq \blup(M,N)$ by $ \pi_{\field P}\colon \field P\to N $. Let $\iota^! A\laover N$ be the pullback of $A$ to $N$.
	\begin{enumerate}
		\item If $ \codim N $ is odd then we have an isomorphism
        \begin{equation}\label{intro:theorem_eq:odd_identif}
		\Homology{\blup(A,\iota^! A)}{\bullet}{}\simeq\Homology{A}{\bullet}{}\oplus \Homology{\iota^!A}{\bullet-1}{}
		\end{equation}
		and, under \eqref{intro:theorem_eq:odd_identif}, $p_A^\ast$ becomes the isomorphism $ p_A^\ast\colon \Homology{A}{\bullet}{}\xrightarrow{\simeq}\Homology{A}{\bullet}{}\oplus 0  $.
		\item\label{intro:theorem:transversals:item:even} If $ \codim N $ is even, there exists a tubular neighbourhood $ E\to N $ such that $ \homology^\bullet(\blup(A,\iota^! A)) $ fits into a long exact sequence
\begin{equation*}
    \dots\to \homology^{\bullet}(A)\xrightarrow{p_A^\ast} \homology^\bullet(\blup(A,\iota^! A)) \to \homology^{\bullet+1}_\mathrm{cv}(A_E)\oplus \homology^{\bullet-1}(\pi^!_{\field P} A) \xrightarrow{g} \homology^{\bullet+1}(A)\to \dots
\end{equation*}
  
		
		where $ g=i\circ \mathrm{pr}_{\homology^{\bullet+1}_\mathrm{cv}(A_E)}$. Here, by $ \homology^{\bullet}_\mathrm{cv}(A_E) $ we denote compact vertical cohomology and by $ i\colon \homology^{\bullet}_\mathrm{cv}(A_E)\to \Homology{A}{\bullet}{} $ the natural map. 
	\end{enumerate}
\end{itheorem}

Recall that by \cite[Section 2.4.1]{gualtieri.li:2014a} we can write the b-tangent bundle associated to a closed hypersurface $N\subseteq M$ by 
\begin{equation*}
    T_N^b M=\blup(TM,TN),
\end{equation*}
which is a blowup of a codimension $1$ transversal. Theorem \ref{intro:theorem:transversals} then reproduces the Mazzeo-Melrose decomposition for b-cohomology \cite{guillemin.miranda.pires:2014a, marcut.osornotorres:2014a}, see \cite{melrose:1993a} for the original result. In this sense, Theorem \ref{intro:theorem:transversals} can be seen as a generalisation of Mazzeo-Melrose as it allows for arbitrary Lie algebroids and transversals of arbitrary codimension. 

One of the ingredients needed for proving the second part of Theorem~\ref{intro:theorem:transversals} is the existence of a Gysin-like long exact sequence for the cohomology of the pullback of a Lie algebroid to a sphere bundle.

\begin{itheorem}[Gysin sequence for Lie algebroids]\label{intro:theorem:gysin}
    Let $ B\laover N $ be a Lie algebroid with anchor $ \rho $, $ \pi\colon \field S\to N $ a sphere bundle of rank $ k $, and $\pi^! B\laover \field S$ the pullback Lie algebroid. There exists a long exact sequence
		\begin{equation}
		\begin{tikzpicture}[baseline=(current bounding box.center)]
		\node (1) at (0,0) {$ \dots $};
		\node (2) at (1.7,0) {$ \homology^\bullet(B) $};
		\node (3) at (4,0) {$ \homology^\bullet( \pi^! B ) $};
		\node (4) at (7,0) {$ \homology^{\bullet-k}(B,o(\field S)) $};
		\node (5) at (10.1,0) {$ \homology^{\bullet+1}(B) $};
		\node (6) at (12,0) {$ \dots $};
		
		\path[->]
		(1) edge node[above]{$  $} (2)
		(2) edge node[above]{$ (\pi^!)^\ast $} (3)
		(3) edge node[above]{$ (\pi^!)_\ast $} (4)
		(4) edge node[above]{$ \wedge \rho^\ast e $} (5)
		(5) edge node[above]{$  $} (6)
		
		;
		\end{tikzpicture}
		\end{equation} 
		Here, $ (\pi^!)_\ast  $ denotes fibre integration and $ e\in \homology^{k+1}(N,o(\field S)) $ is the Euler class of the sphere bundle.
\end{itheorem}

A result we obtain while proving Theorem~\ref{intro:theorem:transversals} is on the de Rham cohomology of real projective blowups. For complex projective blowups there exist results on the de Rham cohomology, see e.g.\ \cite{griffith.harris:1978a}, but for real projective blowups we could not find the statement of Theorem \ref{intro:theorem:de_rham_of_blup} in the literature.

\begin{itheorem}[de Rham cohomology of real projective blowups]\label{intro:theorem:de_rham_of_blup}
		Let $ N\subseteq M $ be a closed submanifold.
	\begin{enumerate}
		\item If $ \codim N $ is odd, then we have an isomorphism
		\begin{equation}
		p^\ast\colon\Homology{M}{\bullet}{}\xrightarrow{\simeq}\Homology{\blup(M,N)}{\bullet}{}.
		\end{equation}
		\item If $ \codim N $ is even, let $ E\to N $ be a tubular neighbourhood of $ N $ in $ M $. Then $ \Homology{\blup(M,N)}{\bullet}{} $ fits into a long exact sequence
\begin{equation}
    \dots\to \homology^\bullet(M)\xrightarrow{p^\ast} \homology^\bullet(\blup(M,N))\xrightarrow{h}\homology^{\bullet+1}_\mathrm{cv}(E)\xrightarrow{i}\mathrm{H}^{\bullet+1}(M)\to \dots
\end{equation}
  
		
		
		Here, $ h $ first restricts a form to $ \field P $, then fibre-integrates and applies the Thom isomorphism.
	\end{enumerate}
\end{itheorem}

On the other hand, in some cases blowing up orbits of Lie algebroids with singular orbit foliation leads to regular Lie algebroids, whose cohomology is easier to compute. 
%

A particular example is the action Lie algebroid $\liealg{so}(3)\ltimes \field R^3$ with singular orbit given by the origin. We compute its cohomology by blowing up, reproducing the result on cohomology obtained by averaging in \cite{ginzburg.weinstein:1992a}.

\begin{itheorem}[The action Lie algebroid $\liealg{so}(3)\ltimes \field R^3$]\label{intro:theorem:so3}
	Let $ A=\liealg{so}(3)\ltimes \field R^3 $ and denote by $ A_{\{0\}}=\liealg{so}(3) $ its restriction to the origin.
	\begin{enumerate}
		\item The blowup $ \blup(A,A_{\{0\}}) $ is a regular Lie algebroid with cohomology given by
		\begin{equation*}
			\Homology{\blup(A,A_{\{0\}})}{k}{}\simeq\begin{cases}
			\{ f\in \Cinfty(\field R^3)\colon f \text{ only depends on the radius} \}&\text{ if }k=0,3\\
			0&\text{ otherwise}.
			\end{cases}
		\end{equation*}
		\item The blow-down map in cohomology
		\begin{equation*}
			p_A^\ast\colon \Homology{A}{\bullet}{}\xrightarrow{\simeq}\Homology{\blup(A,A_{\{0\}})}{\bullet}{}
		\end{equation*}
		is an isomorphism.
	\end{enumerate}
\end{itheorem}

\subsubsection*{Organization of the paper}

After fixing basic definitions and notations in Section~\ref{section:lie_algebroids_and_pull_backs} we briefly describe the known constructions of real projective blowups of submanifolds, (anchored) vector bundles, and Lie algebroids in Section~\ref{section:blow_up_of_LA}. In Section~\ref{section:blow_down_map_in_cohomology} we prove Theorem \ref{intro:theorem:p_A_in_cohomology} while studying the blow-down map in cohomology~\eqref{eq:p_A_in_cohomology}. 
Next, we use Theorem~\ref{intro:theorem:p_A_in_cohomology} to compute the cohomology of the blowup of transversals in Section~\ref{section:blow_up_of_transversals}, including Theorem \ref{intro:theorem:transversals} and \ref{intro:theorem:de_rham_of_blup}. 
In Section~\ref{section:invariant_submanifolds} we investigate two particular cases of blowups of invariant submanifolds. We prove
Theorem~\ref{intro:theorem:so3} and conclude the section by showing that, by repeatedly blowing up restrictions of an action Lie algebroid to orbits one does not always obtain a regular Lie algebroid using $ A=\liealg{sl}_2(\field R)\ltimes \field R^3 $ in Section \ref{sec:sl2r}.

Finally, in the Appendix we discuss fibre integration for Lie algebroids to be able to formulate and prove Theorem~\ref{intro:theorem:gysin}.
	
\subsubsection*{Acknowledgements}
I would like to thank Ioan M\u{a}rcu\textcommabelow{t} for guidance and many useful discussions throughout this project. I would like to thank Aldo Witte for suggesting to study b-tangent bundles from the blowup point of view. Further, I would like to thank Ioan M\u{a}rcu\textcommabelow{t} and Marco Zambon for their feedback on the paper.

The author was financially supported by Methusalem grant METH/21/03 – long term structural funding of the Flemish Government. 	
	
	\section{Lie algebroids and pullbacks}\label{section:lie_algebroids_and_pull_backs}
	In this section we fix notations and prove an identification for the Lie algebroid cohomology of the pullback of a Lie algebroid to a double cover of its base in Lemma~\ref{lemma:lapullback_cohomology}, which we use repeatedly throughout this paper. The material presented in this section is standard, see e.g.\ \cite{mackenzie:2005a} for the general theory of Lie algebroids.
	
	\subsection{Lie algebroid cohomology and representations}
	For a Lie algebroid $ A\laover M $ we denote the bracket on sections by $ [\argument,\argument  ]\colon \Gamma^\infty(A)\times \Gamma^\infty(A)\to \Gamma^\infty(A) $ and the anchor by $ \rho\colon A\to TM $. Occasionally, we add a subscript $ \argument_A $ to avoid confusion. For examples of Lie algebroids see e.g.~\cite[Chapter 8.1.1]{dufour.zung:2005a}.
	
	Since the bracket of a Lie algebroid is defined on sections of a vector bundle, it is not straightforward to define morphisms of Lie algebroids. However, Lie algebroids always come with a cochain complex, which allows for a simple definition of morphisms. To include coefficients in the definition of this cochain complex we need the notion of a representation of Lie algebroids on vector bundles \cite[Definition 2.6]{higgins.mackenzie:1990a}.
	
	\begin{definition}
		Let $ A\laover M $ be a Lie algebroid and $ E\to M $ a vector bundle. An $ \field R $-bilinear map $ \nabla\colon \Gamma^\infty(A)\times\Gamma^\infty(E)\to \Gamma^\infty(E) $ is called \textup{representation of $ A $ on $ E $} if the following holds:
		\begin{enumerate}
			\item $ \nabla $ is $ \Cinfty(M) $-linear in $ \Gamma^\infty(A) $ and satisfies the Leibniz rule
			\begin{equation}
			\nabla_a fs=f\nabla_a s+ \rho(a)(f)s
			\end{equation}
			for $ f\in \Cinfty(M) $, $ a\in \Gamma^\infty(A) $, and $ s\in \Gamma^\infty(E) $.
			\item $ \nabla $ is \textup{flat}, i.e.
			\begin{equation}
			\nabla_a\nabla_b-\nabla_b\nabla_a=\nabla_{[a,b]}
			\end{equation}
			for $ a,b\in \Gamma^\infty(A) $.
		\end{enumerate}
	\end{definition}
	Fix a Lie algebroid $ A\laover M $ and a representation $ \nabla $ on $ E\to M $. Then on $ \Omega^\bullet(A,E) $, the $ E $-valued forms on $ A $, one defines a differential $ \D^\nabla \colon \Omega^\bullet(A,E)\to \Omega^{\bullet+1}(A,E) $ by
	\begin{equation}
	\begin{aligned}
	\D^\nabla \omega (a_0,\dots,a_k)=&\sum_{i=0}^k (-1)^i \nabla_{a_i}\omega(a_0,\dots,\stackrel{i}{\wedge},\dots,a_k)\\
	&+\sum_{\mathclap{0\leq i<j\leq k}} (-1)^{i+j} \omega([a_i,a_j],a_0,\dots,\stackrel{i}{\wedge},\dots,\stackrel{j}{\wedge},\dots,a_k  ),
	\end{aligned}
	\end{equation}
	see e.g.~\cite[Definition 7.1.1]{mackenzie:2005a}. The corresponding \textit{Lie algebroid cohomology of $ A $ with coefficients in $ E $} is denoted by $ \Homology{A}{\bullet}{,E} $. If $ E=M\times \field R $ is the trivial line bundle with representation $ \nabla_a=\rho(a) $ we simply write $ (\Omega^\bullet(A),\D) $ and $ \Homology{A}{\bullet}{} $ for the cochain complex and the cohomology associated to $ A $, respectively.
	
	\begin{definition}
		Let $ A\laover M $ and $ B\laover N $ be Lie algebroids, and $ \Phi\colon A\to B $ a morphism of vector bundles. Then $ \Phi $ is a \textup{morphism of Lie algebroids} if the pullback $\Phi^\ast\colon \Omega^\bullet(B)\to \Omega^\bullet(A)$ defined by
		\begin{equation}
		(\Phi^\ast \omega)(a_1,\dots,a_k):= \big( M\ni p\mapsto \omega_{\phi(p)} (\Phi(a_1(p)),\dots,\Phi(a_k(p)))\big)
		\end{equation}
  for $\omega\in \Omega^k(B)$ and $a_1,\ldots,a_k\in \Gamma^\infty(A)$ is a chain map.
	\end{definition}
	
	If $ \nabla^B $ is a representation of $ B\laover N $ on $ E\to N $ and $ \Phi\colon A\to B $ a Lie algebroid morphism over its base map $ \phi\colon M\to N $, there is an induced representation of $ A $ on the pullback vector bundle $ \phi^\sharp E\to M $. For $s\in \Gamma^\infty(E)$ we denote by $\phi^\sharp s\in \Gamma^\infty(\phi^\sharp E)$ the corresponding pullback section and write $\phi^\sharp\colon \phi^\sharp E\to E$ for the canonical vector bundle morphism.
	
	\begin{proposition}
		Let $ \Phi\colon A\to B $ be Lie algebroid morphism over $ \phi\colon M\to N $ and $ (E\to N,\nabla^B) $ a representation of $ B $.
		\begin{enumerate}
			\item There is a representation of $ A $ on $ \phi^\sharp E $ on pullback sections given by
			\begin{equation}
			\big(\nabla^A_a\phi^\sharp s\big)_p=\nabla^B_{\Phi(a(p))} s\at{\phi(p)}\in \phi^\sharp E_p,
			\end{equation}
			where $ s\in \Gamma^\infty(E) $ and $ a\in \Gamma^\infty(A) $.
			\item The map $ \Phi^\ast\colon \Omega^\bullet(B,E)\to \Omega^\bullet(A,\phi^\sharp E) $ defined by
			\begin{equation}
			(\Phi^\ast \omega)(a_1,\dots,a_k):= \big( M\ni p\mapsto \omega_{\phi(p)} (\Phi(a_1(p)),\dots,\Phi(a_k(p))) \in \phi^\sharp E_p\big)
			\end{equation}
			for $ \omega\in \Omega^k(B,E) $ and $ a_1,\dots,a_k\in \Gamma^\infty(A) $ is a morphism of cochain complexes. 
		\end{enumerate}
	\end{proposition}
	
	To conclude we introduce some notation for later on.
	
	\begin{notation}
		\begin{itemize}
			\item A \textit{pair of manifolds $ (M,N) $} is a manifold $ M $ together with a closed, embedded submanifold $ N\subseteq M $ with $ \codim N\geq 1 $.
			\item Let $ B\to N $ be a vector subbundle of a vector bundle $A\to M$ with $ (M,N) $ a pair of manifolds. Then we call $ (A,B) $ a \textit{pair of vector bundles}.
			\item If $ (A,B) $ is a pair of vector bundles over $ (M,N) $ we write
			\begin{equation}
			\Gamma^\infty(A,B)=\{ a\in \Gamma^\infty(A)\colon a\at{N}\in \Gamma^\infty(B) \}
			\end{equation}
			for sections of $A$ that restrict to sections of $B$.
            \item Let $ B\laover N $ be a Lie subalgebroid of $ A\laover M $ such that $ (A,B) $ is a pair of vector bundles. Then we call $ (A,B) $ a \textit{pair of Lie algebroids}.
		\end{itemize}
	\end{notation}
	
	\subsection{Pullbacks of Lie algebroids}\label{sec:pullbacks_of_Lie_algebroids}
	
	Pullbacks of Lie algebroids, called inverse image Lie algebroids in~\cite[Section 1]{higgins.mackenzie:1990a}, are a useful tool to construct Lie (sub-)algebroids and trivialise representations.
	
	\begin{definition}
		Let $ A\laover M $ be a Lie algebroid, $ N $ a manifold, and $ \phi\colon N\to M $ a map transverse to the anchor. Then the\textup{ pullback Lie algebroid $ \phi^! A\laover N $} is given by the following:
		\begin{enumerate}
			\item For $ x\in N $, the fibres are given by
			\begin{equation}\label{eq:definition_fibre_of_pull_back}
			\phi^! A_x=\{ (a,v)\in A_{\phi(p)}\times T_xN\colon \rho(a)=T_x\phi v \}.
			\end{equation}
			\item The anchor of $ \phi^! A $ is given by the projection onto the second factor.
			\item The bracket is uniquely defined by
			\begin{equation}
			[(f (a\circ \phi), X ), (g(b\circ \phi)),Y ]=( fg([a,b]\circ \phi)+ X(g)(b\circ \phi)-gY(f)(a\circ \phi),[X,Y] )
			\end{equation}
			for $ f,g\in \Cinfty(N) $, $ a,b\in \Gamma^\infty(A) $, and $ X,Y\in \mathfrak{X}(N) $ such that~\eqref{eq:definition_fibre_of_pull_back} is satisfied.
		\end{enumerate}
	\end{definition}
	Note that in general there is no canonical way to assign meaning to the notion of a pullback section $ \phi^! a $, where $ a\in \Gamma^\infty(A) $. Regardless, the canonical map 
	\begin{equation}
	\begin{aligned}
	\phi^!\colon \phi^! A&\to A\\
	(a,v)&\mapsto a,
	\end{aligned}
	\end{equation}
 is a morphism of Lie algebroids. If $ \iota\colon N\hookrightarrow M $ is a submanifold of $ M $ such that the inclusion is transverse to the anchor, $ N $ is called a \textit{transversal} and we can consider $ \iota^! A $ as a Lie subalgebroid of $ A $.
	
The case where $ \phi\colon \tilde{M}\to M  $ is a double cover is of particular interest to us. 
Writing $ \{p^\pm\}=\phi^{-1}(p) $, the map $ T_{p^\pm}\phi $ is a bijection. 
Thus the $ TN $-part of a point in $ \phi^! A $ is uniquely determined by the point in $ A $, hence as a vector bundle $ \phi^! A=\phi^\sharp A $ is just the pullback vector bundle. 
Recall that for any vector bundle $ E\to M $ the pullback bundle $ \phi^\sharp E $ carries a canonical $ \field Z_2 $-action which sends $ v_{p^\pm} $ to $ v_{p^\mp} $, i.e.\ exchanges the base point without changing the fibre. 
The induced $ \field Z_2 $-action on $ \Omega^\bullet(\phi^! A,\phi^\sharp E) $ is given by
	\begin{equation}
	\big(\hat{1}. \tilde{\omega}\big)(\tilde{a}_1,\dots,\tilde{a}_k)=\hat{1}. \big( \tilde{\omega}( \hat{1}. \tilde{a}_1,\dots,\hat{1}. \tilde{a}_k ) \big),
	\end{equation}
	where $ \tilde{\omega}\in \Omega^k(\phi^! A,\phi^\sharp E) $, $ \tilde{a}_1,\dots,\tilde{a}_k\in \Gamma^\infty(\phi^! A) $, and $\hat{1}\in \field{Z}_2$ denotes the nontrivial element.
		\begin{lemma}\label{lemma:lapullback_cohomology}
		Let $ A\laover M $ be a Lie algebroid with a representation on a vector bundle $ E\to M $ and $ \phi\colon \tilde{M}\to M $ a double cover.
		\begin{enumerate}
			\item The action of $ \field Z_2 $  on $ \Omega^\bullet(\phi^! A,\phi^\sharp E) $ is by chain maps. The $ +1 $ eigenspace of $ \hat{1}\in \field Z_2 $ is given by
			\begin{equation}
			\Omega^\bullet(\phi^! A,\phi^\sharp E)_+=(\phi^!)^\ast\Omega^\bullet(A,E)\simeq \Omega^\bullet(A,E).
			\end{equation}
			\item Let $ L=\tilde{M}\times_{\field Z_2} \field R\to M $, where $\hat{1}.(p^\pm,\lambda)=(p^\mp, -\lambda)$. Consider the representation of $A$ on $L$ defined by flatness of locally constant sections. 
            Then the $ -1 $ eigenspace of $ \hat{1}\in \field Z_2 $ is
			\begin{equation}
			\Omega^\bullet(\phi^! A,\phi^\sharp E)_-=\Omega^\bullet(A,E\tensor L).
			\end{equation}
		\end{enumerate}
	\end{lemma}
	
Note that the condition on the representation of $A$ on $ L $ in the second part of Lemma~\ref{lemma:lapullback_cohomology} is nontrivial as it implies that the representation of $ A $ on $ L $ factors through the anchor, i.e.\ there is a (vector bundle) connection $ \nabla^{TM} $ on $ L $ such that 
	\begin{equation}
	\nabla_a=\nabla^{TM}_{\rho(a)}
	\end{equation}
	for all $ a\in \Gamma^\infty(A) $.
\begin{proof}[of Lemma \ref{lemma:lapullback_cohomology}]
			To see that $ \field Z_2 $ acts by chain maps is straightforward when evaluating a form on pullback sections. 
			To show the identifications of complexes, consider first the $ +1 $ eigenspace. Clearly, pullback forms are invariant. On the other hand, if $ \tilde{\omega}\in \Omega^k(\phi^! A,\phi^\sharp E)_+ $ we can define $ \omega\in \Omega^k(A,E) $ by
			\begin{equation*}
			\omega(a_1,\dots,a_k)\colon M\ni p\mapsto \tilde{\omega}_{\tilde{p}}(\phi^\sharp a_1 (\tilde{p}),\dots, \phi^\sharp a_k(\tilde{p} ))\in \phi^\sharp E_{\tilde{p}}=E_p,
			\end{equation*}
			where $ a_1,\dots,a_k\in \Gamma^\infty(A) $ and $ \tilde{p}\in \phi^{-1}(\{p\}) $. This is well-defined by $ \field Z_2 $-invariance of $ \tilde{\omega} $ and $ (\phi^!)^\ast\omega=\tilde{\omega} $, proving the first part.
			
			For the second part $ L $ be given as stated. The line bundle $ L $ is trivial if and only if $ \tilde{M}\to M $ is the trivial double cover and the representation of $ A $ on $ L $ is trivial, implying 
			\begin{equation*}
			\Omega^\bullet(\phi^! A,\phi^\sharp E)_-=\Omega^\bullet(\phi^! A,\phi^\sharp E)_+=\Omega^\bullet(A,E)=\Omega^\bullet(A,E\tensor L)
			\end{equation*}
			using the first part.
			If $ L $ is nontrivial, then the pullback bundle $\phi^\sharp L$ is trivial. In a trivialisation $ \phi^\sharp L=\tilde{M}\times \field R $ for every $ v_p\in L_p $ we have that
			\begin{equation*}\tag{$ \ast $}
			(\phi^\sharp)^{-1}(\{v_p\})=\{(p^+,r(v_p)),(p^-,-r(v_p))  \}\subseteq \tilde{M}\times \field R, 
			\end{equation*}
			where $ \{ p^+,p^- \}=\phi^{-1}(\{p\}) $, and the $ \field Z_2 $-action flips the two points. Let $ \omega=\eta\tensor \ell\in \Omega^k(A,E\tensor L) $ be given, where $ \eta\in \Omega^k(A,E) $ and $ \ell\in \Gamma^\infty(L) $. 
            Then we can define a form $ \tilde{\omega}\in \Omega^k(\phi^! A,\phi^\sharp E) $ by 
			\begin{equation*}
			\tilde{\omega}( \tilde{a}_1,\dots,\tilde{a}_k )\colon \tilde{M}\ni p^\pm\mapsto \pm r\big(\ell(\phi(p^\pm))\big) \eta\big( \phi^!(\tilde{a}_1(p^\pm)),\dots,\phi^!(\tilde{a}_k(p^\pm)) \big)\in E_{\phi(p^\pm)}=\phi^\sharp E_{p^\pm},
			\end{equation*}
            where $\tilde{a}_1,\dots,\tilde{a}_k\in \Gamma^\infty(\phi^! A)$.
			$ \Cinfty(M) $-linear extension as a module morphism along $ \phi^\ast $ gives a map $ \Omega^\bullet(A, E\tensor L)\to \Omega^\bullet(\phi^! A,\phi^\sharp E) $ which maps into the subcomplex of anti-invariant forms. 
            To show that it is a bijection let $ \tilde{\omega}\in \Omega^k(\phi^! A,\phi^\sharp E)=\Omega^k(\phi^! A,\phi^\sharp (E\tensor L)) $ be given. For $ a_1,\dots,a_k\in \Gamma^\infty(A) $ we set
			\begin{equation}
			\omega(a_1,\dots,a_k)\colon p\mapsto \phi^\sharp\big( \tilde{\omega}_{\tilde{p}}( \tilde{a}_1(\tilde{p}),\dots,\tilde{a}_k(\tilde{p})  )  \big), 
			\end{equation}
			where $ \tilde{a}_j(\tilde{p}) $ is chosen such that $ \phi^!(\tilde{a}_j(\tilde{p}))=a_j(p) $. This map is well-defined by $ \field Z_2 $-anti invariance of $ \tilde{\omega} $ and defines a smooth section since $ \phi $ is a local diffeomorphism. Clearly, the two constructions are inverses to each other. Finally, one needs to check that this construction gives a chain map, but this follows from the definition of the pullback representation.
		\end{proof}
 
	Finally, we note that the (anti-) invariant part of the cohomology is the cohomology of the (anti-) invariant subcomplex.
	
	\begin{lemma}
		Let $ A\laover M $, $ \phi\colon \tilde M\to M $, $ E\to M $ and $ L\to M $ be given as in Lemma~\ref{lemma:lapullback_cohomology}. Then
		\begin{equation}
		\Homology{\phi^!A}{\bullet}{,\phi^\sharp E}_\pm=\Homology{\Omega^\bullet(\phi^!A,\phi^\sharp E)_\pm}{\bullet}{}.
		\end{equation}
	\end{lemma}
 \begin{proof}
			This is obtained by averaging the $ \field Z_2 $-action: For example, to see that the natural map for the invariant eigenspaces is injective let $ [\omega]_+\in \Homology{\Omega^\bullet(\phi^!A,\phi^\sharp E)_+}{\bullet}{}  $ be given. If $ [\omega]=0 $ there exists $ \theta\in \Omega^{\bullet-1}(\phi^! A,\phi^\sharp E) $ with $ \D \theta=\omega $. But then 
			\begin{equation*}
			\tfrac{1}{2}(\theta+ (-1).\theta)\in \Omega^{\bullet-1}(\phi^!A,\phi^\sharp E)_+
			\end{equation*}
			is an invariant primitive for $ \omega $.
		\end{proof}

	\section{Blowup of Lie algebroids}\label{section:blow_up_of_LA}
	The idea of the blowup of a Lie algebroid is the following. Suppose that $ (A,B)\laover (M,N) $ is a pair of Lie algebroids. Then blowing up will lead to a new Lie algebroid $ \blup(A,B)\laover \blup(M,N) $ together with a morphism of Lie algebroids $ \blowdown{p_A}{A}{B} $. This means we get an induced map in cohomology $ p_A^\ast\colon \homology^\bullet(A)\to \homology^\bullet(\blup(A,B)) $. Now two scenarios are possible: Either the blown-up Lie algebroid is of interest and we would like to understand its cohomology in relation to $ \homology^\bullet(A) $, or we are interested in computing $ \homology^\bullet(A) $. In the second case, one chooses the Lie subalgebroid $B$ in such a way that $ \homology^\bullet(\blup(A,B)) $ is easier to compute and then draws conclusions regarding $ \homology^\bullet(A) $ via the blow-down map.
	
\subsection{Real projective blowups of smooth manifolds}\label{sec:RPblub_of_manifolds}

Throughout this section let $ (M,N) $ be a pair of manifolds. The blowup $ \blup(M,N) $ of $ N $ in $ M $ is as a set given by
\begin{equation}\label{eq:RPblub_of_manifolds}
\blup(M,N)=(M\setminus N)\cup \field P,
\end{equation}
i.e.\ by replacing $ N $ with $ \field P=\field P(\nu) $, the projectivisation of the normal bundle $$\nu=\nu(M,N)=TM|_N/TN$$ of $ N $ in $ M $. However, \eqref{eq:RPblub_of_manifolds} does as written does not carry an obvious smooth structure. There are several equivalent ways to define the blowup as a manifold.
Since they are all useful to keep in mind we list them here, but define the blowup using a universal property~\cite[Proposition 5.30]{obster:2021a}.
	
	\begin{definition}\label{def:RPblup_univ_property}
		The \textup{blowup} of $ N $ in $ M $ is given by a pair of manifolds $ (B,P) $ together with a map of pairs $ p\colon (B, P)\to (M,N) $ such that
		\begin{enumerate}
			\item $ P\subseteq B $ is a submanifold of codimension $ 1 $,
			\item $ p^{-1}(N)= P $ and the normal derivative $ \D^N p\colon \nu(B, P)\to \nu(M,N) $ is fiberwisely injective,
			\item it satisfies the following universal property: If $ (X,Y) $ is another pair of manifolds and $ q\colon (X,Y)\to (M,N) $ satisfies the first two conditions, there exists a unique map of pairs $ \Phi\colon (X,Y)\to (B, P) $ such that $ q=p\circ\Phi $.
		\end{enumerate}
		The map $ p\colon B\to P $ is called the \textup{blow-down map}.
	\end{definition}

In accordance to the description in \eqref{eq:RPblub_of_manifolds} we write $\blup(M,N)=B$ and $\field P=P$ for the codimension $1$ submanifold.
From Definition~\ref{def:RPblup_univ_property} we see that if $ \codim N=1 $ then $ \blup(M,N)=M $ are isomorphic via the blow-down map, while it hides that the blow-down map is proper (see~\cite[Lemma 2.2]{arone.kankaanrinta:2010a} and \cite[Proposition 5.34]{obster:2021a} for a proof), and that the restriction
	\begin{equation}
	p\at{\blup(M,N)\setminus \field P}\colon \blup(M,N)\setminus \field P\to M\setminus N
	\end{equation}
	is a diffeomorphism.
	One way to construct the blowup is via the \textit{deformation to the normal cone} \cite[Section 4]{debord.Skandalis:2017a}, which also gives insights into the functorial properties of the blowup. As a set the deformation to the normal cone is given by
	\begin{equation}
	\DNC(M,N)=M\times (\field R\setminus \{0\})\cup (\nu\times \{0\})
	\end{equation}
	and is endowed with a smooth structure that magnifies normal directions for $ \field R\ni t\to 0 $, which can be understood in terms of smoothness of a $ \field R^\times $-action given by
	\begin{equation}
	\lambda. (v,t)=\begin{cases}
	(v,\lambda^{-1} t) &\text{ if }t\neq 0\\
	(\lambda v,0)&\text{ if }t=0.
	\end{cases}
	\end{equation}
	This action is proper and free on $ \DNC(M,N)\setminus (N\times \field R) $. Here we consider $ N\times \{0\} $ to sit inside the normal bundle $ \nu $ as the zero section. The blowup of $ N $ in $ M $ is then given by the quotient
	\begin{equation}
	\blup(M,N)=\frac{\DNC(M,N)\setminus (N\times \field R)}{\field R^\times}=M\setminus N\cup \field P(\nu)
	\end{equation}
	and the blow-down map on the projectivisation of the normal bundle is just the fibre projection.
	
	\begin{remark}\label{remark:maps_between_blups_are_complicated}
		Even though the deformation to the normal cone constitutes a functor from the categories of pairs of manifolds to the category of manifolds, the same is not quite true for blowups, see~\cite[Definition 4.8]{debord.Skandalis:2017a}. If $ f\colon (M,N)\to (X,Y) $ is a map of pairs of manifolds, we can only define a blowup of $ f $ on 
		\begin{equation}
		\blup(f)\colon \blup_f(M,N):=\frac{\DNC(M,N)\setminus \DNC(f)^{-1}(Y\times \field R) }{\field R^\times}\to \blup(X,Y).
		\end{equation}
		In particular, $ \blup(f) $ is defined on all of $ \blup(M,N) $ iff $ f^{-1}(Y)=N $ and the normal derivative of $ f $ is fiberwisely injective.
	\end{remark}
	
	For local computations charts describing the smooth structure of the blowup are most useful \cite[Remark 5.29]{obster:2021a}.
	\begin{remark}[Charts for $ \blup(M,N) $]\label{remark:charts_for_blup}
		Let $ (M,N) $ be a pair of manifolds and let $ (U,(x,y)) $ be a submanifold chart, i.e.\ $ U\cap N=\{ x=0 \} $, where for simplicity we assume $ (x,y)(U)=\field R^{n} $. Then the collection $ \{ U_i \}_{i=1,\dots,k} $ defined by
		\begin{equation}
		U_i= p^{-1}(\{ x_i\neq 0 \})\cup \{ [v]_{q}\in \field P_{U\cap N} \colon \D^Nx_i(v)\neq 0  \}
		\end{equation}
		yields an open cover of $ p^{-1}(U) $ and for $ i\in\{ 1,\dots,k\} $ the maps
		\begin{equation*}
		\begin{aligned}
		\Phi_i\colon U_i&\to \field R^n\\
		\xi&\mapsto \begin{cases}
		(y(p(\xi)), \tfrac{x_1(p(\xi))}{x_i(p(\xi))},\dots, \tfrac{x_{i-1}(p(\xi))}{x_i(p(\xi))},x_i(p(\xi)),\tfrac{x_{i+1}(p(\xi))}{x_i(p(\xi))},\dots \tfrac{x_k(p(\xi))}{x_i(p(\xi))} ) &\text{ for } y_i(p(\xi))\neq 0\\
		
		(y(p(\xi)), \tfrac{\D^Nx_1(v)}{\D^Nx_i(v)},\dots,\tfrac{\D^Nx_{i-1}(v)}{\D^Nx_i(v)},0,\tfrac{\D^Nx_{i+1}(v)}{\D^Nx_i(v)},\dots \tfrac{\D^Nx_k(v)}{\D^Nx_i(v)} ) &\text{ for } \xi=[v]_{p(\xi)}
		\end{cases}
		\end{aligned}
		\end{equation*}
		yield charts for $ \blup( U,U\cap N )\subseteq \blup(M,N) $.
	\end{remark}
	
	The final point of view we want to take is locally around the submanifold $N$ in a tubular neighbourhood. Using gluing, this is enough to define the blowup even globally, see e.g.~\cite[Section 2]{mikhalkin:1997a}. Moreover, it implies that every tubular neighbourhood of $ N $ in $ M $ induces a tubular neighbourhood of $ \field P $ in $ \blup(M,N) $, see Corollary~\ref{corollary:lift_of_Evf_is_Evf}.
	
	\begin{remark}\label{remark:blup_of_zero_section}
		Let $ E\to N $ be a vector bundle and view $ N $ as a subset of $ E $ via the image of the zero section. Then
		\begin{equation}
		\blup(E,N)\simeq\field L(E)=\field L,
		\end{equation}
		where $ \field L(E) $ is the tautological line bundle over the projectivisation of $ E $, i.e.\ it is the line bundle over $ \field P $ with fibres given by
		\begin{equation}
		\field L_{[v]}=\{ v'\in E\colon v'=\lambda v \text{ for some }\lambda\in \field R \}
		\end{equation}
		for $ v\in E\setminus N $. In this case, the blow-down map is given by $ v'_{[v]}\mapsto v' $.
	\end{remark}
	
	Thus every tubular neighbourhood of $ N $ in $ M $ induces a tubular neighbourhood of $ \field P $ in $ \blup(M,N) $. We show that the respective Euler vector fields are related by the blow-down map. For that, we first note that vector fields on $ M $ that are tangent to $ N $ lift to the blowup in a unique way.

	\begin{lemma}\label{lemma:lifting_vf_to_the_blowup}
		Let $ N\subseteq M $ be a submanifold, $ X\in \mathfrak{X}(M) $ and write $ p\colon \blup(M,N)\to M $ for the blow-down map. Then there is a vector field $ \tilde{X}\in \mathfrak{X}(\blup(M,N)) $ with $ \tilde{X}\sim_p X $ if and only if $ X\in \Gamma^\infty(TM,TN) $, i.e.\ $ X $ is tangent to $ N $. In that case, $ \tilde{X} $ is unique and tangent to $ \field P\subseteq \blup(M,N) $.
		\begin{proof}
			First suppose that $ \tilde{X} $ exists. Note that for $ [v]\in \field P $ we have $ \image(T_{[v]} p)=T_{p([v])}N\oplus \field R v $, making use of a tubular neighbourhood. Then 
			\begin{equation*}
			\bigcap_{ [v]\in \field P_p }\image(T_{[v]}\rho) =T_pN,
			\end{equation*}
			thus $ X\in \Gamma^\infty(TM,TN) $ follows immediately. Conversely let $ X\in\Gamma^\infty(TM,TN) $ be given. Take an adapted chart $ (U,(y,x)) $ of $ N $ in $ M $ with $ N\cap U=\{ x=0 \} $. Then $ X $ is locally given by
			\begin{equation*}
			X\at{U}= f^j\Partial{y^j}+g^k\Partial{x^k}
			\end{equation*}
			for $ f^j,g^k\in \Cinfty(M) $, where all $ g^k $ vanish on $ N $. Fix $ i\in \{ 1,\dots,\codim N \} $ and consider the chart $ (U_i,\Phi_i=(\tilde{y},\tilde{x})) $ of the blowup adapted to $ \field P $ (given by $ \tilde{x}_i=0 $). Then on $ U_i\setminus \field P $ we have
			\begin{equation*}\tag{$ \ast $}
			p^{\ast}X=p^\ast f^j \Partial{\tilde{y}^j}+p^\ast g^i\Partial{\tilde{x}^i}+\sum_{k\neq i} \frac{1}{\tilde{x}^i}(p^\ast g^k-p^\ast g^i \tilde{x}^k   )\Partial{\tilde{x}^k}.
			\end{equation*}
			Since for all $ k\in \{1,\dots,\codim N\} $ we have $ g^k\at{N} =0$, $ p^\ast g^k\at{\field P}=0 $ follows. Thus there are functions $ \tilde{g}^k\in \Cinfty(U_i) $ such that $ p^\ast g^k=\tilde{x}^i \tilde{g}^k $. Inserting this in $ (\ast) $ shows that $ p^\ast X\at{U\setminus N} $ extends smoothly to a vector field on $ U_i $, hence $ p^\ast X\at{M\setminus N} $ extends smoothly to $ \field P $.
		\end{proof}
	\end{lemma}
	
	\begin{corollary}\label{corollary:lift_of_Evf_is_Evf}
		Let $ E\to N $ be a vector bundle with Euler vector field $ \xi\in \vfield(E) $. Then the lift of $ \xi $ is the Euler vector field of the line bundle $ \field L=\blup(E,N)\to \field P $.
		\begin{proof}
			Lemma~\ref{lemma:lifting_vf_to_the_blowup} implies that $ \xi $ lifts to the blowup. Since the blow-down map is a vector bundle morphism and a diffeomorphism on $ E\setminus N $, we can compare the flows of the vector fields there to see that the lift is indeed the Euler vector field.
		\end{proof}
	\end{corollary}

	\subsection{Blowups of Lie algebroids}

	Starting from a Lie algebroid $ A\laover M $ and a Lie subalgebroid $ B\laover N $, the blowup $ \blup(A,B) $ according to Definition~\ref{def:RPblup_univ_property} will in general not carry a natural structure of a Lie algebroid. In fact, if $ \rank(B)<\rank (A) $, by Remark~\ref{remark:maps_between_blups_are_complicated} already the vector bundle projection $ \pi\colon A\to M $ of $ A $ will not lift to the blowup. Instead, one has to consider $ \blup_\pi (A,B) $. 
	
	We start by blowing up vector subbundles, then anchored subbundles, and finally Lie subalgebroids. Even though each construction will bring its own universal property, we only state it for the blowup of Lie algebroids.

	\begin{lemma}\label{lemma:sections_of_blup}
		Let $\pi\colon (A,B)\to (M,N) $ be a pair of vector bundles. Then $ \blup(\pi)\colon \blup_\pi(A,B)\to \blup(M,N) $ is a vector bundle with sections given by the $ \Cinfty(\blup(M,N)) $-span of
		\begin{equation}
		\{ \blup(s)\colon \blup(M,N)\to \blup_\pi(A,B) \colon s\in \Gamma^\infty(A,B) \},
		\end{equation}
		considering $ s\in \Gamma^\infty(A,B) $ as a map of pairs $ (M,N)\to (A,B) $.
		\begin{proof}
			See~\cite[Section 5.4.2]{obster:2021a}. Note that $ \blup(s)\colon \blup_s(M,N)\to\blup(A,B) $ is actually defined on all of $ \blup(M,N) $ and maps into $ \blup_\pi(A,B) $ since $ s $ is a section.
		\end{proof}
	\end{lemma}
	
	From Lemma \ref{lemma:sections_of_blup} we can immediately write down local frames for the blowup.
	
	\begin{remark}[Local frames for $ \blup_\pi(A,B) $]\label{remark:local_frames_for_blups}
		Let $ B\to N $ be a vector subbundle of corank $ k $ of $ \pi\colon A\to M $. Let $ (U,(x,y)) $ be a submanifold chart of $ N $ in $ M $, i.e.\ $ U\cap N=\{ x=0 \} $. Moreover, let $ \{ e_1,\dots,e_k,f_1,\dots,f_{\rank B} \} $ be a local frame of $ A_U $ adapted to $ B $, meaning $ \{ f_1\at{N},\dots f_{\rank B}\at{N} \} $ is a local frame for $ B_{U\cap N} $. Then the collection 
		\begin{equation}
		\{ \blup(x_ie_1),\dots,\blup(x_ie_k),\blup(f_1),\dots,\blup(f_{\rank B}) \}
		\end{equation}
		yields a local frame for $ \blup_\pi(A,B)_{U_i} $.
	\end{remark}
	
	\begin{lemma}\label{lemma:vb_blup_in_two_stages}
		Let $ \pi\colon (A,B)\to (M,N) $ be a pair of vector bundles and let $ p\colon \blup(M,N)\to M $ denote the blow-down map of the base.
		\begin{enumerate}
			\item As vector bundles, $ \blup_\pi(A,A_N)=p^\sharp A $. Under this identification we have $ p_A=p^\sharp $.
			\item As vector bundles, $ \blup_{\pi_{p^\sharp A}}( p^\sharp A,p^\sharp B )=\blup(A,B) $.
		\end{enumerate}
		\begin{proof}
			Both statements follow from Remark~\ref{remark:local_frames_for_blups} using their respective local frames.
		\end{proof}
	\end{lemma}
	
	Given a map of pairs of vector bundles $ \Phi\colon (A,B)\to (E,F) $ over a base map $ \phi\colon (M,N)\to (X,Y) $ we obtain an induced map of vector bundles
	\begin{equation}
	\blup(\Phi)\colon \blup_{\pi_A}(A,B)_{\blup_\phi(M,N)}\to\blup_{\pi_E}(E,F).
	\end{equation}
	Next, we can include an anchor map into the construction, i.e.\ a vector bundle morphism $ \rho\colon A\to TM $ over the identity.
	
	\begin{lemma}\label{lemma:blup_of_anchored_vb}
		Let $ (\pi\colon A\to M,\rho) $ be an anchored vector bundle and $ B\to N $ a subbundle with $ \rho(B)\subseteq TN $. Then there is an induced anchor on $ \blup_\pi(A,B) $ which on blowup sections is given by
		\begin{equation}
		\rho_{\blup}(\blup(s))=\widetilde{\rho(s)},
		\end{equation}
		where $ s\in \Gamma^\infty(A,B) $ and $ \widetilde{\rho(s)} $ is the extension from Lemma~\ref{lemma:lifting_vf_to_the_blowup}. In particular, the blow-down map is a morphism of anchored vector bundles.
	\end{lemma}
	
	Alternatively, one can blowup the map of pairs $ \rho\colon (A,B)\to (TM,TN) $ and consider the concatenation \begin{equation}\label{eq:definition_of_anchor_of_blups}
	\rho_{\blup}\colon \blup_\pi(A,B)\to \blup_{\pi_{TM}}(TM,TN)\to T\blup(M,N),
	\end{equation}
	where the second map is given by the identity in fibres over $ M\setminus N $ and by the differential of the projection $ \nu(M,N)\setminus 0\to \field P(M,N) $ over the projective bundle ~\cite[Proposition 5.58]{obster:2021a}. Lastly, the universal property of the blowup of Lie algebroids is the following.
	
	\begin{definition}
		Let $ (A,B)\laover (M,N) $ be a pair of Lie algebroids. Then the \textup{blowup of $ B $ in $ A $} is a pair of Lie algebroids $ (\tilde{A},\tilde{B}) $ together with a morphism $ p_A\colon (\tilde{A},\tilde{B})\to (A,B) $ of pairs of Lie algebroids such that
		\begin{enumerate}
			\item $ \tilde{B} $ is a full rank Lie subalgebroid over a base of codimension $ 1 $,
			\item $ p_A^{-1}(B)=\tilde{B} $ and the normal derivative $ \D^Np_A\colon \nu(\tilde{A},\tilde{B})\to  \nu(A,B) $ is fiberwisely injective,
			\item it satisfies the following universal property: If $ (E,F) $ is another pair of Lie algebroids and $ \phi\colon (E,F)\to (A,B) $ a morphism of pairs of Lie algebroids such that the first two conditions are satisfied, there exists a unique morphism of pairs of Lie algebroids $ \Phi\colon (E,F)\to (\tilde{A},\tilde{B}) $ such that $ \phi =p_A\circ \Phi $.
		\end{enumerate}
	\end{definition}
	
	One can obtain an explicit model for the blowup of Lie algebroids by equipping $ \blup_\pi(A,B) $ with the anchor from Lemma~\ref{lemma:blup_of_anchored_vb} and bracket on blowup sections given by
	\begin{equation}\label{eq:def_bracket_on_blowup}
	[\blup(a),\blup(b)]=\blup([a,b])
	\end{equation}
	for $ a,b\in \Gamma^\infty(A,B) $. Together with the Leibniz rule, \eqref{eq:def_bracket_on_blowup} defines the bracket uniquely as these sections generate $ \Gamma^\infty(\blup_\pi(A,B)) $. Then the blow-down map $ p_A\colon \blup_\pi(A,B)\to A $ becomes a Lie algebroid morphism over $ \blowdown{p}{M}{N} $. Proving that this construction satisfies the universal property is analog to the case of blowups of manifolds, for which it is given in \cite[Proposition 5.30]{obster:2021a}.
	
	\begin{example}[Elementary modifications]
		Elementary modifications \cite{klaasse:2017a, gualtieri.li:2014a, lanius:2021a} can be expressed in terms of blowups. If $ N\subseteq M $ has $ \codim(N)=1 $, the blow-down map is a diffeomorphism. In particular, when blowing up the base manifold does not change, and in this case the sections of the Lie algebroid blowup are precisely given by 
		\begin{equation}
		\Gamma^\infty(\blup_\pi(A,B))=\Gamma^\infty(A,B)
		\end{equation}
		by Lemma~\ref{lemma:sections_of_blup}. Examples of these modifications are
		\begin{enumerate}
			\item the log-tangent bundle $ T^b_NM=\blup_{\pi_{TM}}(TM,TN) $,
			\item the scattering tangent bundle $ \blup_{\pi_{T^b_NM}}(T^b_NM,0_N) $.
		\end{enumerate}
	\end{example}

	\section{The blow-down map in cohomology}\label{section:blow_down_map_in_cohomology}

Since for a pair of Lie algebroids $(A,B)\laover (M,N)$ the only interesting blowup to consider is $ \blup_\pi(A,B) $, we drop the additional subscript $ \pi $ from now on. 

The blow-down map $ p_A\colon \blup_\pi(A,B)\to \blup(M,N) $ is a morphism of Lie algebroids, hence induces a map in Lie algebroid cohomology
	\begin{equation}
	p_A^\ast\colon \Homology{A}{\bullet}{}\to \Homology{\blup_\pi(A,B)}{\bullet}{}.
	\end{equation}

 In this section we study the blow-down map in cohomology by first considering the pullback by the blow-down map 
 \begin{equation}\label{eq:pullback_by_pA}
     p_A^\ast\colon \Omega^\bullet(A)\to\Omega^\bullet(\blup(A,B))
 \end{equation}
 on the level of forms. For a pair of Lie algebroids $ (A,B)\laover (M,N) $ the pullback \eqref{eq:pullback_by_pA} will not be surjective unless both $ N\subseteq M $ is of codimension $ 1 $ and $ B=A_N $. However, if we restrict to forms that are flat along $ N $ and $ \field P=p^{-1}(N) $, respectively, $ p_A^\ast $ becomes an isomorphism. Here we mean flatness in the following sense:
	
	\begin{definition}
		Let $ E\to M $ be a vector bundle and $ N\subseteq M $ a submanifold. Let  $ \vanishing_N $ denote the vanishing ideal of $ N $.
		\begin{enumerate}
			\item The space of \textup{flat sections along $ N $} is 
			\begin{equation}
			\Gamma^\infty_N(E)=\bigcap_{k\in \field N} \vanishing_N^k\Gamma^\infty(E).
			\end{equation}
			\item The \textup{$ \infty $-jets of sections along $ N $} are
			\begin{equation}
			\mathscr{J}_N^\infty\Gamma^\infty(E)=\frac{\Gamma^\infty(E)}{\Gamma^\infty_N(E)}.
			\end{equation}
		\end{enumerate}
	\end{definition}
	
	With the notation fixed we can formulate the important result of this section, which shows that the blow-down map in cohomology only depends on local data around the blown-up manifold.
	
	\begin{theorem}\label{theorem:flatiso_and_consequences}
		Let $ (A,B)\laover (M,N) $ be a pair of Lie algebroids.
		\begin{enumerate}
			\item\label{theorem:flatiso_and_consequences:item:flatiso} The induced map on flat forms
			\begin{equation}
			p_A^\ast\colon \Omega_N^\bullet(A)\to \Omega_{\field P}^\bullet(\blup(A,B))
			\end{equation}
			is an isomorphism of cochain complexes.
			\item\label{theorem:flatiso_and_consequences:item:consequences} By part \ref{theorem:flatiso_and_consequences:item:flatiso} we obtain an isomorphism of cochain complexes 
			\begin{equation}
			\frac{\Omega^\bullet(\blup(A,B))}{p_A^\ast\Omega^\bullet(A)}\cong \frac{\mathscr{J}^\infty_{\field P}\Omega^\bullet(\blup(A,B))}{p_A^\ast \mathscr{J}^\infty_N\Omega^\bullet(A)}.
			\end{equation}
			In particular, in the long exact sequence
			\begin{equation}\label{eq:les_for_blup}
			\begin{tikzpicture}[baseline=(current bounding box.center)]
			\node (1) at (0,0) {$ \dots $};
			\node (2) at (1.8,0) {$ \Homology{A}{\bullet}{} $};
			\node (3) at (4.6,0) {$ \Homology{\blup(A,B)}{\bullet}{} $};
			\node (4) at (8,0) {$ \homology^\bullet\left(\frac{\Omega^\bullet(\blup(A,B))}{p_A^\ast\Omega^\bullet(A)}\right)$};
			\node (5) at (10.9,0) {$ \Homology{A}{\bullet+1}{} $};
			\node (6) at (12.8,0) {$ \dots $};
			
			\path[->]
			(1) edge node[]{$  $} (2)
			(2) edge node[above]{$ p_A $} (3)
			(3) edge node[above]{$ f $} (4)
			(4) edge node[above]{$  $} (5)
			(5) edge node[]{$  $} (6);
			
			\end{tikzpicture}
			\end{equation}
			the map $ f $ only depends on local data around $ N $ and $ \field P $, e.g.\ if $ \iota\colon U\hookrightarrow \blup(M,N) $ is an open neighbourhood of $ \field P $ then $ f = f_U\circ\homology(\iota^\ast) $. Here, $ \homology(\iota^\ast)\colon \homology^\bullet(\blup(A,B))\to \homology^\bullet(\blup(A,B)_U) $ denotes the map induced in cohomology by the inclusion and $ f_U $ is the map corresponding to $ f $ in the long exact sequence~\eqref{eq:les_for_blup} for $ A_{p(U)} $. 
		\end{enumerate}
	\end{theorem}

	To show the first part of Theorem~\ref{theorem:flatiso_and_consequences} we use that a flat form on $ A $ along $ N $ can be written as a locally finite sum of $ \wedge $-products of flat functions and flat $ 1 $-forms. To show this, we need the following \cite[Theorem 1]{nagel:1973a}.
	
	\begin{lemma}\label{lemma:factorization_of_flat_functions}
		Let $ \Omega\subseteq \field R^n $ be open and $ F\subseteq \Omega $ be closed. Let $ \{f_i\}_{i\in \mathbb{N}}\subseteq \Cinfty_F(\Omega) $ be any countable collection of 
		functions which are flat along $ F $. Then there exist $ \{g_i\}_{i\in \field N}\subseteq \Cinfty_F(\Omega) $ and $ h\in \Cinfty_F(\Omega) $ such that
		\begin{enumerate}
			\item $ h(x)>0  $ if $ x\in \Omega\setminus F $,
			\item $ f_i=hg_i $ for all $ i\in \field N $.
		\end{enumerate}
	\end{lemma}
	
	\begin{lemma}\label{Cohomology_Blup.lemma:flat_generation}
		Let $ E\to M $ be a vector bundle and $ N\subseteq M $ a submanifold. Then any $ \omega\in \Gamma_N^\infty(\Lambda^\bullet E) $ can be written as a locally finite sum of $ \wedge $-products of elements in $ \Cinfty_N(M)\oplus\Gamma^\infty_N(E) $.
		\begin{proof}
			Let $ (U_\alpha,x_\alpha)_\alpha $ be an atlas of submanifold charts with local frames $ \{e_\alpha^1,\dots,e_\alpha^\ell  \} $ of $ E_{U_\alpha} $ and $ \omega\in \Gamma^\infty_N(\Lambda^k E) $ for some $ k>1 $. Let $ \{\chi_\alpha \}_\alpha $ be a partition of unity of order $ k+1 $ subordinate to the open cover $ \{ U_\alpha\}_\alpha $, i.e.\ $ \supp(\chi_\alpha)\subseteq U_\alpha $ for all $ \alpha $ and $ \sum_\alpha \chi_\alpha^{k+1}=1 $. Locally,
			\begin{equation*}
			\omega\at{U_\alpha}=f_{\alpha,j_1,\dots,j_k}e_\alpha^{j_1}\wedge \cdots\wedge e_\alpha^{j_k}
			\end{equation*}
			in the chart $ (U_\alpha,x_\alpha) $. Flatness of $ \omega $ on $ N $ means flatness of all coefficient functions $ f_{\alpha,j_1,\dots,j_k} $ on $ U_\alpha\cap N $. In particular, $ \chi_\alpha f_{\alpha,j_1,\dots,j_k} $ is flat on $ U_\alpha\cap N $ and by applying Lemma~\ref{lemma:factorization_of_flat_functions} $ k $ times we find $ g_{1,\alpha,j_1,\dots,j_k},\dots, g_{k,\alpha,j_1,\dots,j_k}\in \Cinfty_{U_\alpha\cap N}(U_\alpha) $ such that
			\begin{equation*}
			g_{1,\alpha,j_1,\dots,j_k}\cdots g_{k,\alpha,j_1,\dots,j_k}=\chi_\alpha f_{\alpha,j_1,\dots,j_k}.
			\end{equation*}
			Thus we can decompose
			\begin{equation*}
			\begin{aligned}
			\omega
            &=\sum_\alpha\chi_\alpha^{k+1} \omega=\sum_\alpha\chi_\alpha^k \chi_\alpha f_{\alpha,j_1,\dots,j_k}e_\alpha^{j_1}\wedge \cdots\wedge e_\alpha^{j_k}\\
			&=\sum_{\alpha,j_1,\dots,j_k}\chi_\alpha^k g_{1,\alpha,j_1,\dots,j_k}\cdots g_{k,\alpha,j_1,\dots,j_k}e_\alpha^{j_1}\wedge \cdots\wedge e_\alpha^{j_k}\\
			&=\sum_{\alpha,j_1,\dots,j_k} \big( \chi_\alpha g_{1,\alpha,j_1,\dots,j_k} e_\alpha^{j_1} \big)\wedge \cdots\wedge \big( \chi_\alpha g_{k,\alpha,j_1,\dots,j_k} e_\alpha^{j_k} \big),
			\end{aligned}
			\end{equation*}
			which proves the claim as $ \chi_\alpha g_{\ell,\alpha,j_1,\dots,j_k} e_\alpha^{j_\ell} $ is a well-defined $ 1 $-form on all of $ M $ which is flat on $ N $.
		\end{proof}
	\end{lemma}
	
	Using Lemma \ref{Cohomology_Blup.lemma:flat_generation} we can show Theorem~\ref{theorem:flatiso_and_consequences}. Since the first part is a statement about sections of vector bundles and does not involve the additional structures provided by a Lie algebroid, we formulate and prove it in a separate lemma.
	
	\begin{lemma}\label{lemma:past_is_iso_on_flats}
		Let $ (E,F)\to (M,N) $ be a pair of vector bundles. Then the pullback of the blow-down map $ p^\ast_E\colon \Gamma^\infty(E^\ast)\to \Gamma^\infty(\blup(E,F)^\ast) $ restricts to an isomorphism
		\begin{equation}
		p^\ast_E\colon \Gamma^\infty_N(E^\ast)\xrightarrow{\simeq} \Gamma^\infty_{\field P} (\blup(E,F)^\ast).
		\end{equation}
  \end{lemma}
\begin{proof}
The proof consists of two steps. First, we show that $ p^\ast\colon \Cinfty_N(M)\to \Cinfty_{\field P}(\blup(M,N)) $ is an isomorphism (which corresponds to $ E=M\times \field R $ and $ F=N\times\field R $). Then we can prove the statement for general vector bundles.
It is clear that $ p^\ast $ factors to flat functions (as $ p^\ast \vanishing_N\subseteq \vanishing_{\field P} $) and is injective. Thus, let $ f\in \Cinfty_{\field P}(\blup(M,N)) $ be given. 
Since $ f\at{\field P}=0 $ we find a function $ \tilde{f}\in \Cinfty(M) $ with $ p^\ast \tilde{f}=f $, which is automatically continuous. 
To show that it is smooth and flat along $ N $ we use the charts from Remark~\ref{remark:charts_for_blup}. 
First we note that differentiation along $ N $ is just differentiating along the base coordinates of $ \field P $, hence there is nothing to check. 
Thus we can assume that $ N $ is a point and, since differentiation is local, $ M=\field R^n $. 
As the computations in arbitrary dimensions are the same as in two with more bookkeeping, suppose $ n=2 $. 
			
			In the chart $ \Phi_1 $ the blow-down map maps
			\begin{equation*}
			p(\Phi_1^{-1}(x_1,x_2))=(x_1,x_1x_2).
			\end{equation*}
			We use $ (x_1,x_2) $ for the coordinates on $ U_1\subseteq \blup(\field R^2,\{0\}) $ and $ (x,y) $ on $ \field R^2 $. We show inductively that for any $ \alpha\in \field N $ and $ \beta\leq k $ the derivative $ \frac{\partial^\beta}{\partial x^\beta}\frac{\partial^{\alpha-\beta}}{\partial y^{\alpha-\beta}}\tilde{f} $ pulls back to a function in $ \Cinfty_{\field P}(\blup(M,N)) $ and thus extends continuously to zero. Regarding the differential of the blow-down map we find
			\begin{equation*}
			\begin{aligned}
			T_{(x_1,x_2)}p (\partial_{x_1})&=\partial_x+x_2\partial_y=\partial_x+\frac{y}{x}\partial_y\\
			T_{(x_1,x_2)}p(\partial_{x_2})&= x_1\partial_y=x\partial_y.
			\end{aligned}
			\end{equation*}
			In other words, away from the origin in $ p(U_1) $ we have for the pullback with $ p $ that
			\begin{equation}
			\begin{aligned}
			p^\ast \partial_x &= \partial_{x_1}-\frac{x_2}{x_1}\partial_{x_2} \\
			p^\ast \partial_y &= \frac{1}{x_1}\partial_{x_2}.
			\end{aligned}
			\end{equation}
			These are singular vector fields, but since $ f $ is flat along $ \field P $ we find $ g\in \Cinfty_{\field P\cap U_1}(U_1) $ such that $ f=x_1g $. Hence we can apply them to $ f $, yielding a function that is again flat along $ \field P $. The same computation for $ (U_2,\Phi_2) $ shows that the first derivatives of $ \tilde{f} $ indeed pull back to flat functions as claimed. But then the same argument also shows the step $ k\to k+1 $. 
			
			For the general vector bundle blowup, let $ (U,(x,y)) $ be a submanifold chart of $ N $ such that $ U\cap N=\{x=0\} $, and $ \{ e_1,\dots,e_{\mathrm{corank} F},f_1,\dots,f_{\rank F} \} $ be a local frame for $ E_U $ such that the collection $ \{f_\alpha\}_\alpha $ restricts to a frame of $ F $ over $ U\cap N $. Recall that for $ U_i\subseteq \blup(M,N) $, $ i=1,\dots,\codim N $, a local frame for $ \blup(E,F)_{U_i} $ is given by 
			\begin{equation*}\tag{$ \ast $}
			\{ \blup(x_ie_1),\dots,\blup(x_ie_{\mathrm{corank}F}),\blup(f_1),\dots,\blup(f_{\rank F}) \},
			\end{equation*} 
			see Remark~\ref{remark:local_frames_for_blups}. For the pullbacks of the dual frames of $ E_U $ to $ U_i $ one immediately finds
			\begin{equation*}
			\begin{aligned}
			p_E^\ast f^\alpha( \blup(f_\alpha) )&= 1\\
			p_E^\ast e^k(\blup(x_ie_k))&=x_i,
			\end{aligned}
			\end{equation*}
			while all other pairings are zero. Writing the dual frame of $ (\ast) $ as
			\begin{equation*}
			\{ \blup(x_ie_1)^\ast,\dots,\blup(x_ie_{\mathrm{corank}B})^\ast,\blup(f_1)^\ast,\dots,\blup(f_{\rank B})^\ast \},
			\end{equation*}
			any $ \hat{\omega}\in \Gamma^\infty_{\field P}(\blup(E,F)) $ can locally be written as  $ \hat{\omega}\at{U_i}= \hat{\omega}^k \blup(x_i e_k)^\ast+\hat{\omega}^\alpha \blup(f_\alpha)^\ast $. But since $ \hat{\omega} $ is flat along $ \field{P}\cap U_i $, the functions $ \tfrac{\hat{\omega}^k}{x_i} $ are still well-defined on $ U_i $ and flat along $ \field{P}\cap U_i $. On $p(U_i)\setminus N$, we can define the $1$-form
			\begin{equation*}
			\omega_i= \sum_\alpha (p_E)_\ast(\hat{\omega}^\alpha) f^\alpha+\sum_k (p_E)_\ast\big(\tfrac{\hat{\omega}^k}{x_i}\big)e^k.
			\end{equation*}
			Since $ p_E^\ast \omega_i=\hat{\omega}\at{U_i\setminus \field P} $, we have $\omega_i=\omega_j$ for $i,j\in \{1,\dots,\codim N\}$ whenever their domain of definition overlap. Thus, gluing gives $\omega\in \Omega^1(U\setminus N)$, which extends flatly to $ N\cap U $ and satisfies $p_E^\ast \omega=\hat{\omega}\at{p^{-1}(U)}$.
		\end{proof}
	
	\begin{proof}[of Theorem~\ref{theorem:flatiso_and_consequences}]
		The first part follows from Lemma~\ref{Cohomology_Blup.lemma:flat_generation} and Lemma~\ref{lemma:past_is_iso_on_flats}. For the second part consider the following commutative diagram.
		
		\begin{equation*}\label{Cohomology_Blup.eq:exact_diagram_flatntaylor}
		\begin{tikzpicture}[baseline=(current bounding box.center)]
		\node (21) at (2,0) {$ 0 $};
		\node (31) at (5.5,0) {$ 0 $};
		\node (41) at (9.5,0) {$ 0 $};
		\node (25) at (2,7) {$ 0 $};
		\node (35) at (5.5,7) {$ 0 $};
		\node (45) at (9.5,7) {$ 0 $};
		\node (12) at (0,1.5) {$ 0 $};
		\node (13) at (0,3.5) {$ 0 $};
		\node (14) at (0,5.5) {$ 0 $};
		\node (52) at (12,1.5) {$ 0 $};
		\node (53) at (12,3.5) {$ 0 $};
		\node (54) at (12,5.5) {$ 0 $};
		\node (44) at (9.5,5.5) {$ 0 $};
		\node (34) at (5.5,5.5) {$ \Omega_{\field P}^\bullet(\blup(A,B)) $};
		\node (24) at (2,5.5) {$ \Omega^\bullet_N(A) $};
		\node (43) at (9.5,3.5) {$ \frac{\Omega^\bullet(\blup(A,B))}{p_A^\ast\Omega^\bullet(A)} $};
		\node (33) at (5.5,3.5) {$ \Omega^\bullet(\blup(A,B)) $};
		\node (23) at (2,3.5) {$ \Omega^\bullet(A) $};
		\node (42) at (9.5,1.5) {$\frac{\mathscr{J}^\infty_{\field P}\Omega^\bullet(\blup(A,B))}{p_A^\ast \mathscr{J}^\infty_N\Omega^\bullet(A)} $};
		\node (32) at (5.5,1.5) {$ \mathscr{J}^\infty_{\field P}\Omega^\bullet(\blup(A,B)) $};
		\node (22) at (2,1.5) {$ \mathscr{J}^\infty_{N}\Omega^\bullet(A) $};
		\path[->] 
		(14) edge node[]{} (24)
		(24) edge node[above]{$ p_A^\ast $} (34)
		(34) edge node[]{} (44)
		(44) edge node[]{} (54)
		(13) edge node[]{} (23)
		(23) edge node[above]{$ p_A^\ast $} (33)
		(33) edge node[]{} (43)
		(43) edge node[]{} (53)
		(12) edge node[]{} (22)
		(22) edge node[above]{$ p_A^\ast $} (32)
		(32) edge node[]{} (42)
		(42) edge node[]{} (52)
		(25) edge node[]{} (24)
		(24) edge node[left]{$$} (23)
		(23) edge node[]{} (22)
		(22) edge node[]{} (21)
		(35) edge node[]{} (34)
		(34) edge node[left]{$  $} (33)
		(33) edge node[]{} (32)
		(32) edge node[]{} (31)
		(45) edge node[]{} (44)
		(44) edge node[left]{$  $} (43)
		(43) edge node[]{} (42)
		(42) edge node[]{} (41);
		\end{tikzpicture}
		\end{equation*}
		Exactness of the first and second column is clear. Regarding exactness of the second and third row note that $ p_A $ is a diffeomorphism on $ M\setminus N $, implying injectivity of $ p_A^\ast $. From the first part we get exactness of the first row, and thus
		\begin{equation*}
		\frac{\Omega^\bullet(\blup(A,B))}{p_A^\ast\Omega^\bullet(A)}\cong \frac{\mathscr{J}^\infty_{\field P}\Omega^\bullet(\blup(A,B))}{p_A^\ast \mathscr{J}^\infty_N\Omega^\bullet(A)}
		\end{equation*}
		by the $ 3\times3 $-Lemma \cite[Chapter 2, Lemma 5.1]{maclane:1963b}. The remaining claim follows from the local nature of jet spaces.		
	\end{proof}

	

	\section{Blowups of transversals}\label{section:blow_up_of_transversals}
	
	One class of Lie subalgebroids of $ A\laover M $ is given by transverse submanifolds $ \iota\colon N\hookrightarrow M $, i.e.\ submanifolds such that the inclusion is transverse to the anchor. Then $ \rho^{-1}(TN)=\iota^! A\subseteq A $ is a Lie subalgebroid, see Section~\ref{sec:pullbacks_of_Lie_algebroids}. An example of such is e.g.\ $ TN\subseteq TM $ for any closed submanifold $ N\subseteq M $ as a Lie subalgebroid of the tangent Lie algebroid. 
We compute the cohomology of $\blup(A,\iota^! A)$ in Corollary \ref{corollary:blup_of_transversals} after characterising the blow-down map in cohomology in this setting in Theorem \ref{theorem:blup_transversals:past_in_cohomology}.
 
A property of transversals we will use is that they admit simple normal forms \cite[Section 4]{BursztynLimaMeinrenken2019}.
	
	\begin{theorem}\label{theorem:normal_form_transversals}
		Let $ A\laover M $ be a Lie algebroid and $ \iota\colon N\hookrightarrow M $ a transversal. Then there exists a tubular neighbourhood $ \pr\colon E\to N $ of $ N $ in $ M $ such that 
		\begin{equation}
		A_E\cong \pr^! \iota^! A
		\end{equation}
		are isomorphic as Lie algebroids.
	\end{theorem}
	
	Moreover, any such isomorphism $ A_E\cong \pr^! \iota^! A $ induces an isomorphism 
	\begin{equation}\label{eq:transversals_local_cohomology}
	\homology^\bullet(A_E)=\homology^\bullet(\iota^!A)
	\end{equation}
	by means of the restriction map, see~\cite[Theorem 2]{frejlich:2019a}. The proof in \cite{frejlich:2019a} utilises \cite[Theorem 2]{crainic:2000a}, a spectral sequence argument, see also \cite[Section 5.2]{marcut.schuessler:2024a}.
	
	According to Theorem~\ref{theorem:flatiso_and_consequences} it is enough to understand the blow-down map in cohomology in a neighbourhood of $ N $, thus we can utilise Theorem~\ref{theorem:normal_form_transversals} to calculate $ \Homology{\blup(A,\iota^! A)}{\bullet}{} $. To formulate the result we need two definitions.
	
	\begin{definition}
		Let $ A\laover M $ be a Lie algebroid and $ B\laover N $ a Lie subalgebroid of $ A $.
		\begin{enumerate}
			\item A vector field $ X\in \mathfrak{X}(M) $ is called \textup{Euler-like along $ N $} if it is complete and there is a tubular neighbourhood embedding $ \phi\colon E\to U $ of a vector bundle $ E\to N $ onto an open neighbourhood $ U $ of $ N $ in $ M $ such that $ \phi^\ast X $ is the Euler vector field of $ E $ \cite[Definition 2.6]{BursztynLimaMeinrenken2019}.
			\item A section $ a\in \Gamma^\infty(A,B) $ is called \textup{Euler-like along $ B $} if $ \rho(A) $ is Euler-like along $ N $ and it induces the trivial inner derivation
			\begin{equation}
			[a\at{N},\argument]_B=0\colon \Gamma^\infty(B)\to \Gamma^\infty(B).
			\end{equation}
		\end{enumerate}
	\end{definition}

    From~\cite[Lemma 3.9]{BursztynLimaMeinrenken2019} we know that, given a transversal $ N $, we can always find Euler-like sections $ a\in \Gamma^\infty(A) $ along $ B $ with $ a\at{N}=0 $.
	Note that this definition of an Euler-like section of a Lie algebroid differs from the one given in \cite{bischoff.bursztyn.lima.meinrenken:2020a} in the way that we do not require $ a\at{N} $ itself to vanish. The reason behind this is that, for an isomorphism like \eqref{eq:transversals_local_cohomology} to hold, the existence of an Euler-like section in our sense is enough, one does not necessarily need a trivialisation $A_E\cong \pr^! \iota^! A$ of Lie algebroids, see \cite[Theorem 3.24]{marcut.schuessler:2024a}.
 
    Moreover, to formulate Theorem \ref{theorem:blup_transversals:past_in_cohomology} we need the notion of compact vertical cohomology.
	
	\begin{definition}\label{def:compact_vertical_cohomology}
		Let $ \pi\colon E\to M $ be a fibre bundle and $ A\laover E $ a Lie algebroid. The subcomplex of forms of $ A $ compactly supported in vertical directions is defined as
		\begin{equation}\label{eq:compact_vertical_complex}
		\Omega^\bullet_{\mathrm{cv}}(A)=\{ \omega\in \Omega^\bullet(A)\colon\supp(\omega)\cap \pi^{-1}(K) \text{ is compact for all compact }K\subseteq M \}.
		\end{equation}
	\end{definition}
	
	Since the differential is a local operator, \eqref{eq:compact_vertical_complex} is indeed a subcomplex, with cohomology denoted by $ \homology^\bullet_{\mathrm{cv}}(A) $, called the \textit{compact vertical cohomology} (see e.g.\ \cite{stokesfordensities}). Of course, the same notion can be defined also for forms with coefficients in some representation. Regarding the cohomology of the blowup of transversals, we obtain the following results.
	
	\begin{theorem}\label{theorem:blup_transversals:past_in_cohomology}
		Let $ \iota\colon N\hookrightarrow M $ be a closed transversal of a Lie algebroid $ A\laover M $. Denote the blow-down map by $ \blowdown{p_A}{A}{\iota^! A} $, the blow-down map of the base by $ \blowdown{p}{M}{N} $ and the projection of the projective bundle by $ \pi_{\field P}\colon \field P\to N $.
		\begin{enumerate} \item\label{theorem_item:past_in_cohomology:iso_and_codim1} We have (canonical) isomorphisms
			\begin{align}
			\Homology{\blup(A,\iota^! A)}{\bullet}{}&\cong\Homology{\blup(p^! A,\pi_{\field P}^! A)}{\bullet}{}\\
			&\cong\Homology{p^! A}{\bullet}{}\oplus \Homology{\pi_{\field P}^! A}{\bullet-1}{}.\label{theorem_eq:general_identif}
			\end{align}
			Under the identification \eqref{theorem_eq:general_identif}, the blow-down map $ p_A^\ast $ in cohomology becomes $ (p^!)^\ast\colon \Homology{A}{\bullet}{}\to\Homology{p^! A}{\bullet}{} $.
			\item \label{theorem_item:past_in_cohomology:odd} If $ \codim N $ is odd then 
			\begin{equation}\label{eq:odd_codim_p^!_cohomology}
			(p^!)^\ast\colon\Homology{A}{\bullet}{}\xrightarrow{\simeq}\Homology{p^!A}{\bullet}{}.
			\end{equation}
			\item \label{theorem_item:past_in_cohomology:even} If $ \codim N $ is even then any section $a\in \Gamma^\infty(A)$ Euler-like along $ \iota^! A $ with $a\at{N}=0$ and corresponding tubular neighbourhood $ E\to N $ in $ M $ gives rise to a long exact sequence
			\begin{equation}\label{eq:even_codim_p^!_cohomology}
			\begin{tikzpicture}[baseline=(current bounding box.center)]
			\node (1) at (0,0) {$ \dots $};
			\node (2) at (2,0) {$ \homology^\bullet(A) $};
			\node (3) at (4.5,0) {$ \homology^\bullet(p^!A) $};
			\node (4) at (7,0) {$ \homology_\mathrm{cv}^{\bullet+1} (A_E) $};
			\node (5) at (9.5,0) {$ \homology^{\bullet+1}(A) $};
			\node (6) at (11.5,0) {$ \dots $};
			
			\path[->]
			(1) edge node[]{$  $} (2)
			(2) edge node[above]{$ (p^!)^\ast $} (3)
			(3) edge node[above]{$ (p^!)_\ast $} (4)
			(4) edge node[above]{$ i $} (5)
			(5) edge node[]{$  $} (6)
			
			;
			\end{tikzpicture}
			\end{equation}
			where $ (p^!)_\ast $ first restricts a form to $ \field P $, fibre integrates and applies the Thom isomorphism (Lemma \ref{lemma:Thom_for_Liealgebroids}), and $ i([\omega])=[\omega] $ for any $ [\omega]\in \homology_\mathrm{cv}^\bullet(A_E) $.
		\end{enumerate}
	\end{theorem}
	Note that the choice of Euler-like section in the third part only affects the tubular neighbourhood, i.e.\ two Euler-like sections inducing the same tubular neighbourhood lead to the same long exact sequence \eqref{eq:even_codim_p^!_cohomology}.
 
	Theorem \ref{theorem:blup_transversals:past_in_cohomology} characterises the blow-down map in cohomology completely. Note that the reason for the distinction between even and odd codimension lies in the de Rham cohomology of the projective spaces that constitute the fibres of $ \field P\to N $, see also Theorem~\ref{theorem:blup_transversals:projective_bundles}. In the case of odd codimension the cohomology is trivial in all but $ 0 $-th degree, leading to the simplified form. We prove Theorem~\ref{theorem:blup_transversals:past_in_cohomology} in two steps: First we prove the case of $ \codim N=1 $ in Section~\ref{sec:transversals_codim1} and then complete the proof in Section~\ref{sec:proof_of_transversals}.
	
	From Equation~\eqref{theorem_eq:general_identif} we see that the only missing ingredient to compute the cohomology of the blowup of a transversal is computing $ \Homology{\pi_{\field P}^! A}{\bullet}{} $.
	
	\begin{theorem}\label{theorem:blup_transversals:projective_bundles}
		Let $ B\laover N $ be a Lie algebroid. Let $ E\to N $ be a vector bundle of rank $ k $, $ o(E)\to N $ its orientation bundle, and $ e\in \Homology{N}{k}{,o(E)} $ its Euler class. Let $ \pi_{\field P}\colon \field P\to N $ denote the projectivisation of $ E $.
		\begin{enumerate}
			\item If $ k $ is odd then we have an isomorphism
			\begin{equation}
			(\pi_{\field P}^!)^\ast\colon \Homology{B}{\bullet}{}\xrightarrow{\simeq}\Homology{\pi_{\field P}^!B}{\bullet}{}.
			\end{equation}
			\item \label{theorem_item:projective_bundles:gysin}If $ k $ is even then there is a Gysin-like long exact sequence
			\begin{equation}
			\begin{tikzpicture}[baseline=(current bounding box.center)]
			\node (1) at (-0.3,0) {$ \dots $};
			\node (2) at (1.5,0) {$ \homology^\bullet(B) $};
			\node (3) at (3.8,0) {$ \homology^\bullet( \pi_{\field P}^! B ) $};
			\node (4) at (7,0) {$ \homology^{\bullet-(k-1)}(B,o(E)) $};
			\node (5) at (10.2,0) {$ \homology^{\bullet+1}(B) $};
			\node (6) at (12,0) {$ \dots $};
			
			\path[->]
			(1) edge node[above]{$  $} (2)
			(2) edge node[above]{$ (\pi_{\field P}^!)^\ast $} (3)
			(3) edge node[above]{$ (\pi_{\field P}^!)_\ast $} (4)
			(4) edge node[above]{$ \wedge \rho^\ast e $} (5)
			(5) edge node[above]{$  $} (6)

			;
			\end{tikzpicture}
			\end{equation} 
			Here, $ (\pi_{\field P}^!)_\ast $ denotes fibre integration.
		\end{enumerate}
		\begin{proof}
			Let $ N $ be of odd codimension. In this case, the fibres of the projective bundle $\pi_{\field P}\colon \field P\to N $ are projective spaces of even dimension, which have nontrivial de Rham cohomology only in degree $ 0 $. By the Serre-Leray spectral sequence for Lie algebroids \cite[Corollary 5.8]{marcut.schuessler:2024a}, we immediately obtain
			\begin{equation*}
			(\pi_{\field P}^!)^\ast\colon \Homology{B}{\bullet}{}\xrightarrow{\simeq}\Homology{\pi_{\field P}^!B}{\bullet}{}.
			\end{equation*}
			The second part follows from Theorem~\ref{theorem:gysin_sequence_for_LA} noting that all maps are compatible with the antipodal action.
		\end{proof}
	\end{theorem}
	
	Putting together Theorem~\ref{theorem:blup_transversals:past_in_cohomology} and Theorem~\ref{theorem:blup_transversals:projective_bundles} we can compute $ \Homology{\blup(A,\iota^! A)}{\bullet}{} $.
	
	\begin{corollary}\label{corollary:blup_of_transversals}
		Let $ \iota\colon N\hookrightarrow M $ be a closed transversal of a Lie algebroid $ A\laover M $. Denote the blow-down map by $ \blowdown{p_A}{A}{\iota^! A} $, the blow-down map of the base by $ \blowdown{p}{M}{N} $, and the projection of the projective bundle by $ \pi_{\field P}\colon \field P\to N $.
		\begin{enumerate}
			\item If $ \codim N $ is odd then \begin{equation}\label{theorem_eq:odd_identif}
			\Homology{\blup(A,\iota^! A)}{\bullet}{}=\Homology{A}{\bullet}{}\oplus \Homology{\iota^!A}{\bullet-1}{}
			\end{equation}
			and, under \eqref{theorem_eq:odd_identif}, $ p_A^\ast\colon \Homology{A}{\bullet}{}\xrightarrow{\simeq}\Homology{A}{\bullet}{}\oplus 0  $.
			\item If $ \codim N $ is even, $ \homology^\bullet(\blup(A,\iota^! A)) $ fits into a long exact sequence
   \begin{equation*}
    \dots\to \homology^{\bullet}(A)\xrightarrow{p_A^\ast} \homology^\bullet(\blup(A,\iota^! A)) \xrightarrow{f} \homology^{\bullet+1}_\mathrm{cv}(A_E)\oplus \homology^{\bullet-1}(\pi^!_{\field P} A) \xrightarrow{g} \homology^{\bullet+1}(A)\to \dots
\end{equation*}
where $ \image (f)= X\oplus\homology^{\bullet-1}(\pi_{\field P}^! A) $ for a subspace $X\subseteq \homology^{\bullet+1}_\mathrm{cv}(A_E)$, and $ g=i\circ \mathrm{pr}_{\homology^{\bullet+1}_\mathrm{cv}(A_E)}$. 
			
		\end{enumerate}
	\end{corollary}
	
	As a corollary, Theorem~\ref{theorem:blup_transversals:past_in_cohomology} also computes the de Rham cohomology of real projective blowups of manifolds.

	\begin{corollary}
		Let $ (M,N) $ be a pair of manifolds.
		\begin{enumerate}
			\item If $ \codim N $ is odd, then
			\begin{equation}
			p^\ast\colon\Homology{M}{\bullet}{}\xrightarrow{\simeq}\Homology{\blup(M,N)}{\bullet}{}.
			\end{equation}
			\item If $ \codim N $ is even, let $ E\to N $ be a tubular neighbourhood of $ N $ in $ M $. Then $ \Homology{\blup(M,N)}{\bullet}{} $ fits into a long exact sequence
			\begin{equation}
			\begin{tikzpicture}[baseline=(current bounding box.center)]
			\node (1) at (0,0) {$ \dots $};
			\node (2) at (2,0) {$ \homology^\bullet(M) $};
			\node (3) at (5.1,0) {$ \homology^\bullet(\blup(M,N)) $};
			\node (4) at (8.3,0) {$ \homology_\mathrm{cv}^{\bullet+1} (E) $};
			\node (5) at (11,0) {$ \homology^{\bullet+1}(M) $};
			\node (6) at (13.2,0) {$ \dots $};
			
			\path[->]
			(1) edge node[]{$  $} (2)
			(2) edge node[above]{$ p^\ast $} (3)
			(3) edge node[above]{$ p_\ast $} (4)
			(4) edge node[above]{$ i $} (5)
			(5) edge node[]{$  $} (6)
			
			;
			\end{tikzpicture}
			\end{equation}
			Here, $ p_\ast $ first restricts a form to $ \field P $, then fibre-integrates and applies the Thom isomorphism.
		\end{enumerate}
		\begin{proof}
			The statement follows from Equation~\eqref{eq:odd_codim_p^!_cohomology} for odd codimension and Equation~\eqref{eq:even_codim_p^!_cohomology} for even codimension of $ N $, applied to the Lie algebroid $ A=TM $.
		\end{proof}
	\end{corollary}

\subsection{The case of codimension 1}\label{sec:transversals_codim1}
	
In this section we prove Theorem~\ref{theorem:blup_transversals:past_in_cohomology}, \ref{theorem_item:past_in_cohomology:iso_and_codim1} for $ \codim (N)=1 $. In this case, $p\colon \blup(M,N) \xrightarrow{\simeq} M$ and $\field P\simeq N$, hence the statement reduces to 
 \begin{equation}\label{eq:cohomology_codim1_blup}
		\homology^\bullet(\blup(A,\iota^! A))=\homology^\bullet(A)\oplus \homology^{\bullet-1}(\iota^! A).
\end{equation}

We prove a slightly stronger statement, which can be seen as the Mazzeo-Melrose analogue for the blowup of a codimension $1$ transversal, by adapting the proof in \cite[Section 2.1]{marcut.osornotorres:2014a} to our situation.

\begin{theorem}[Mazzeo-Melrose for Lie algebroid transversals]\label{theorem:blup_of_codim1_transversal}
		Let $ \iota\colon N\hookrightarrow M $ be a codimension $ 1 $ transversal of a Lie algebroid $ A\laover M $. Then the sequence
		\begin{equation}\label{eq:mazzeo_melrose_ses}
		\begin{tikzpicture}[baseline=(current bounding box.center)]
		\node (1) at (0,0) {$ 0 $};
		\node (2) at (1.9,0) {$ \Omega^\bullet(A) $};
		\node (3) at (5,0) {$ \Omega^\bullet(\blup(A,\iota^! A)) $};
		\node (4) at (9,0) {$ \Omega^{\bullet-1}(\iota^! A) $};
		\node (5) at (11,0) {$ 0 $};
		
		\path[->]
		(1) edge node[]{$  $} (2)
		(2) edge node[above]{$ p^\ast_A $} (3)
		(3) edge node[above]{$  j_{\blup(a)}\at{N} $} (4)
		(4) edge node[]{$  $} (5)
		
		;
		\end{tikzpicture}
		\end{equation}
		where $ j $ denotes the right insertion of sections and $ a\in \Gamma^\infty(A) $ is Euler-like along $ N $ with $ a\at{N}=0 $, is a split exact sequence of cochain complexes and does not depend on the choice of $a$. In particular, the Lie algebroid cohomology of the blowup is given by
		\begin{equation}
		\homology^\bullet(\blup(A,\iota^! A))=\homology^\bullet(A)\oplus \homology^{\bullet-1}(\iota^! A).
		\end{equation}
	\end{theorem}

We prove Theorem \ref{theorem:blup_of_codim1_transversal} in the remainder of this section. First, we collect properties of $\blup(a)\in \Gamma^\infty(\blup(A,\iota^! A))$ for $a\in \Gamma^\infty(A)$ as in the Theorem.
	
	\begin{lemma}\label{lemma:existence_of_eulerlikes}
		Let $ \iota\colon N\hookrightarrow M $ be a Lie algebroid transversal of any codimension $ \codim N\geq 1 $. For an Euler-like section $ a\in \Gamma^\infty(A) $ with $ a\at{N}=0 $ (which exists by \cite[Lemma 3.9]{BursztynLimaMeinrenken2019}), the blowup section $ \blup(a)\in \Gamma^\infty(\blup(A,\iota^! A)) $ satisfies the following:
		\begin{enumerate}
			\item $ \rho_\blup(\blup(a)) $ is Euler-like along $ \field P $.
			\item $ \blup(a)\in \Gamma^\infty(\blup(A,\iota^!A)) $ is nowhere vanishing on $ N $ and satisfies $ p_A(\blup(a)\at{\field P})=0 $, where $ p_A\colon \blup(A,\iota^! A)\to A $ denotes the blow-down map.
			\item $ \blup(a) $ is Euler-like along the Lie subalgebroid $ \blup(A,\iota^!A)_{\field P} $.
            \item If $a'\in \Gamma^\infty(A)$ is another section Euler-like along $\iota^! A$ with $a'\at{N}=0$, we have
            \begin{equation}
                (\blup(a)-\blup(a'))\at{N}=0.
            \end{equation}
            In particular, \eqref{eq:mazzeo_melrose_ses} does not depend on the choice of $a$.
		\end{enumerate}
		\begin{proof}
			The first part follows from Corollary~\ref{corollary:lift_of_Evf_is_Evf}.For the second part, since $ p_A(\blup(a)\at{\field P})=0 $ is clear we, only have to show that $ \blup(a)\at{\field P} $ is nowhere vanishing. For this, consider $ \rho\colon (A,\iota^! A)\to (TM,TN) $ and $ \rho(a) \colon (M,N)\to (TM,TN) $ as map of pairs. Then
			\begin{equation}
			\blup(\rho)(\blup(a))=\blup(\rho(a))\in \Gamma^\infty(\blup(TM,TN)).
			\end{equation}
			But for the tubular neighbourhood for which $ \rho(a) $ is the Euler vector field it is easy to see that $ \blup(\rho(a)) $ is nowhere vanishing on $ \field P $ using Remark~\ref{remark:local_frames_for_blups}. Thus $ \blup(a) $, too, must be nowhere vanishing on $ \field P $.
   
			For the third part let $ f\in \Cinfty(\blup(M,N)) $ and $ \mu\in \Gamma^\infty(\blup(A,\iota^! A)) $ be given. Then
			\begin{equation*}
			[\blup(a),f\mu]_\blup\at{\field P}=\underbrace{\rho_\blup (\blup(a))}_{=0 \text{ on }\field P}(f)\at{\field P}\mu+f[ \blup(a),\mu ]_\blup\at{\field P},
			\end{equation*}
			so it is enough to check the statement on a set of generators, namely $ \blup( \Gamma^\infty(A,\iota^! A) ) $. Next, we use $ a\in \Gamma^\infty(A) $ to obtain an isomorphism of Lie algebroids
			\begin{equation*}
			A\at{E}\cong \pr^! \iota^! A
			\end{equation*}
			for the tubular neighbourhood $ \pr\colon E\to N $ corresponding to $ \rho(a) $. Since the statement is a local one around $ \field P $ it is enough to work in this neighbourhood. By \cite[Remark 3.19]{BursztynLimaMeinrenken2019} this isomorphism maps $ a\in \Gamma^\infty(A) $ to $ (0,\xi_E) $, where $ \xi_E $ is the Euler vector field of $ E $. Pick a linear connection on $ E $ with corresponding horizontal lift $ \argument^\hor $. Let $ \{ U_\alpha \}_\alpha$ be an open cover of $ N $ such that there are local frames $ \{ f_1,\dots,f_{\rank \iota^!A} \} $ of $ (\iota^!A)_U $ and $ \{ e_1,\dots,e_{\codim N} \} $ of $ E_U $. Then the collection $ \{ \tilde{f}_1,\dots,\tilde{f}_{\rank \iota^!A} , \tilde{e}_1,\dots,\tilde{e}_{\codim N} \} $ defined by
			\begin{equation*}
			\begin{aligned}
			\tilde{f}_i&=( f_i\circ \pr, (\rho_{\iota^! A}(f_i))^\hor )\\
			\tilde{e}_j&=( 0, e_j^\ver )
			\end{aligned}
			\end{equation*}
			yields a local frame of $ \pr^! \iota^! A_{E_U} $. The module $ \Gamma^\infty (\pr^! \iota^! A_{E_U},\iota^! A_{U}) $ is generated by $ \tilde{f}_i $'s and $ g\tilde{e}_j $'s, where $ g\in \mathcal{J}_N $ is in the vanishing ideal of $ N $. 
            By definition we have
			\begin{equation*}
			[\blup(a),\blup(b)]_\blup
			=\blup([a,b])
			\end{equation*}
            and want to show that
            \begin{equation*}
                \blup([a,b])\at{\field P}=0
            \end{equation*}
            for any such section $ b $.
			For this, note that it is enough to show that $ [a,b] $ vanishes to second order along $ N $. Clearly, we have
			\begin{equation*}
			[a,\tilde{f}_i]=[(0,\xi_E),(f_i\circ \pr,\rho_{\iota^! A}(f_i)^\hor)]=(0,[\xi_E,\rho_{\iota^! A}(f_i)^\hor])=0.
			\end{equation*}
			Thus let $ g\in \mathcal{J}_N $ be given. Then
			\begin{equation*}
			\begin{aligned}
			[a,g\tilde{e}_j]&=(0,\xi_E(g)e_j^\ver)+g( 0,[\xi_E, e_j^\ver] )\\
			&=(0,\xi_E(g)e_j^\ver- ge_j^\ver)\\
			&=0 \mod \mathcal{I}_N^2
			\end{aligned}
			\end{equation*}
			as $ \xi_E(g)=g \mod\mathcal{J}_N^2 $.

   For the last part, let $a'\in \Gamma^\infty(A)$ be another choice of Euler-like section. In the trivialisation induced by $a$ we can write
   \begin{equation*}
       a'=g^i \tilde{f}_i+h^j\tilde{e}_j
   \end{equation*}
   for $g^i,h^j\in \Cinfty(M)$. Since by assumption $a'\at{N}=0$, $g^i\in \vanishing_N$ follows. Since $\rho_A(a')$ is Euler-like along $N$,
   \begin{equation*}
       \xi_E-h^j\tilde{e}_j\in \vanishing_N^2 \vfield(M)
   \end{equation*}
   vanishes to second order. In conclusion, 
   $\blup(a-a')\at{\field P}=0$.
		\end{proof}
	\end{lemma}

Next, we show that the insertion of $ \blup(a)\at{N} $ is a cochain map. Clearly the restriction to $ N $ itself is a cochain map, so the crucial part happens inside the Lie subalgebroid $ \blup(A,\iota^! A)_{N} $. We first note that, in the case of $ \codim N=1 $, the Lie algebroid $ \blup(A,\iota^! A)_{N} $ is an abelian extension according to the definition in \cite{mackenzie:2005a}.
	
	\begin{lemma}\label{lemma:codim1_transversal_abeilan_extension}
		Let $ \iota \colon N\hookrightarrow M $ be a transversal of a Lie algebroid $ A\laover M $ of any codimension. Then $ \blup(A,\iota^! A)_{\field P} $ fits into a short exact sequence of Lie algebroids
		\begin{equation}\label{eq:diagram_bluptrans_ses_of_LA}
		\begin{tikzpicture}[baseline=(current bounding box.center)]
		\usetikzlibrary{arrows}
		\node (1) at (0,1.2) {$ 0 $};
		\node (2) at (1.5,1.2) {$ L $};
		\node (3) at (4,1.2) {$ \blup(A,\iota^! A)_{\field P} $};
		\node (4) at (6.5,1.2) {$ \iota^! A $};
		\node (5) at (8,1.2) {$ 0 $};
		
		\node (w) at (1.5,0) {$ \field P $};
		\node (e) at (4,0) {$ \field P $};
		\node (r) at (6.5,0) {$ N $};
		\draw[-Implies,double equal sign distance]
		(2) -- (w);
		\draw[-Implies,double equal sign distance]
		(3) -- (e);
		\draw[-Implies,double equal sign distance]
		(4) -- (r);

		\path[->]
		(1) edge node[]{$  $} (2)
		(2) edge node[above]{$ i $} (3)
		(3) edge node[above]{$ p_A $} (4)
		(4) edge node[]{$  $} (5)
		
		(w) edge node[above]{$ \id_{\field P} $} (e)
		(e) edge node[above]{$ p $} (r);
		
		\end{tikzpicture}
		\end{equation}
		where $ L:= \ker p_A\at{\field P}  $ and $ i\colon L\hookrightarrow \blup(A,\iota^! A)_{\field P} $ denotes the inclusion. 
  
  Moreover, if $ \codim N=1 $ the sequence \eqref{eq:diagram_bluptrans_ses_of_LA} is an abelian extension of Lie algebroids.
		\begin{proof}
			Since $ p_A\at{\field P}\colon \blup(A,\iota^! A)_{\field P}\to A_N $ is a vector bundle morphism with image given by $ \iota^! A $, it follows that $ L=\ker p_A\at{\field P} $ is a vector bundle of rank $ \codim N $ and \eqref{eq:diagram_bluptrans_ses_of_LA} is a short exact sequence of Lie algebroids. 
            To show that in case of $ \codim N=1 $ the Lie algebroid $L$ is abelian, note that the anchor of the blowup makes the diagram
			\begin{equation*}
			\begin{tikzpicture}[baseline=(current bounding box.center)]
			\node (blup) at (0,1.5) {$ \blup(A,\iota^!A)_N $};
			\node (A) at (2.5,1.5) {$ A_N $};
			\node (TNl) at (0,0) {$ TM_N $};
			\node (TNr) at (2.5,0) {$ TM_N $};
			
			\path[->]
			(blup) edge node[above]{$ p_A $} (A)
			(blup) edge node[left]{$ \rho_\blup $} (TNl)
			(A) edge node[right]{$ \rho $} (TNr)
			(TNl) edge node[above]{$ = $} (TNr)
			
			;
			\end{tikzpicture}
			\end{equation*}
			commute, from which $ L\subseteq \ker \rho_\blup $ follows. Thus the inherited anchor is zero, turning $L$ into an abelian Lie algebroid of rank $ 1 $.
		\end{proof}
	\end{lemma}

 \begin{remark}
     For an arbitrary codimension of $N$, the extension \eqref{eq:diagram_bluptrans_ses_of_LA} is not abelian. Consider $A=T\field R^2\laover \field R^2$ and $\iota\colon N=\{0\}\hookrightarrow \field R^2$. Then $L=\blup(A,\iota^! A)_{\field P}$, and
     \begin{equation*}
         [\blup(x\partial_x),\blup(x\partial_y)]_\blup=\blup(x\partial_x)
     \end{equation*}
     does not vanish on $\field P$.
 \end{remark}
	
	By Lemma \ref{lemma:codim1_transversal_abeilan_extension}, for $ \iota \colon N\hookrightarrow M $ a transversal of a Lie algebroid $ A\laover M $ of codimension $ 1 $ there is a representation of $ \iota^! A $ on $ L $ given by
	\begin{equation}
	\nabla_b \eta =[\tilde{b},\eta ]=  [\hat{b},\hat{\eta}]\at{N}
	\end{equation}
	for $ \tilde{b}\in \Gamma^\infty(\blup(A,\iota^! A )\at{N}) $ with $ p(\tilde{b})=b $ or, alternatively, $ \hat{b}, \hat{\eta} $ extensions of $ \tilde{b},\eta $ to a section of $ \blup(A,\iota^! A) $ \cite[Proposition 3.3.20]{mackenzie:2005a}. 
 
Key to showing that $ j_{\blup(a)}\at{N} $ in~\eqref{eq:mazzeo_melrose_ses} is a chain map is to see that it is enough for $ \blup(a)\at{N} $ to be constant with respect to the representation of $ \iota^! A $.
	
	\begin{lemma}\label{lemma:req_for_eulervf_to_be_chain}
		Let $ A\laover M $ be a Lie algebroid, $ \iota\colon N\hookrightarrow M $ a codimension $ 1 $ transversal and $ L=\ker p_A\at{N} $. Suppose $ \eta\in \Gamma^\infty(L) $ is constant with respect to the action of $ \iota^! A $, i.e.\ 
		\begin{equation}
		\nabla_b \eta =0
		\end{equation}
		for all $ b\in \Gamma^\infty(\iota^! A) $. Then the map 
		\begin{equation}
		\Omega^\bullet(\blup(A,\iota^! A))\ni \omega\mapsto j_\eta \omega\at{N}\in \Omega^{\bullet-1}(\iota^! A)
		\end{equation}
		is a chain map, where for $ \omega\in \Omega^k(\blup(A,\iota^! A)) $ and $ b_1,\dots,b_{k-1}\in \Gamma^\infty(\iota^! A) $ we define 
		\begin{equation}
		j_\eta \omega\at{N}(b_1,\dots,b_{k-1})=\omega\at{N}(\tilde{b}_1,\dots,\tilde{b}_{k-1},\eta).
		\end{equation}
		Here $ \tilde{b_i} $ is any section of $ \Gamma^\infty(\blup(A,\iota^! A)_N) $ such that $ p_A(\tilde{b}_i)=b_i $.
		\begin{proof}
			Let $ \omega\in \Omega^k(\blup(A,\iota^! A)) $ and $ b_0,\dots,b_{k-1}\in \Gamma^\infty(\iota^! A) $ be given with corresponding $ \tilde{b}_i $. Then we get
			\begin{equation*}
			\begin{aligned}
			j_\eta\D\omega\at{N}&(b_0,\dots,b_{k-1})= (\D\omega)(b_0,\dots,b_{k-1},\eta)\\
			=&\sum_{i=0}^{k-1}(-1)^i\rho_\blup(\tilde{b}_i)\omega\at{N}(\tilde{b}_0,\dots,\stackrel{i}{\wedge},\dots,\tilde{b}_{k-1},\eta)+(-1)^{k}\rho_\blup(\eta)\omega\at{N}(\tilde{b}_0,\dots,\tilde{b}_{k-1})\\
			&+\sum_{\mathclap{0\leq i<j\leq k-1}}(-1)^{i+j}\omega\at{N}([\tilde{b}_i,\tilde{b}_j],\tilde{b}_0,\dots,\stackrel{i}{\wedge},\dots,\stackrel{j}{\wedge},\dots,\tilde{b}_{k-1},\eta)\\
			&+ (-1)^{k}\sum_{i=0}^{k-1}(-1)^i\omega\at{N}([\tilde{b}_i,\eta],\tilde{b}_0,\dots,\stackrel{i}{\wedge},\dots,\tilde{b}_{k-1})\\
			=&\D (j_\eta\omega\at{N})(b_0,\dots,b_{k-1})+(-1)^{k}\rho_\blup(\eta)\omega\at{N}(\tilde{b}_0,\dots,\tilde{b}_{k-1})\\&+(-1)^{k}\sum_{i=0}^{k-1}(-1)^i\omega\at{N}([\tilde{b}_i,\eta],\tilde{b}_0,\dots,\stackrel{i}{\wedge},\dots,\tilde{b}_{k-1}).
			\end{aligned}
			\end{equation*}
			Thus it defines a chain map if and only if the expression
			\begin{equation*}
			\rho_\blup(\eta)\omega\at{N}(\tilde{b}_0,\dots,\tilde{b}_{k-1})+\sum_{i=0}^{k-1}(-1)^i\omega\at{N}([\tilde{b}_i,\eta],\tilde{b}_0,\dots,\stackrel{i}{\wedge},\dots,\tilde{b}_{k-1})
			\end{equation*}
			is zero. But under the assumptions on $ \eta $ every single summand vanishes.
		\end{proof}
	\end{lemma}
	
	\begin{corollary}\label{corollary:blup_of_euler_trivialises_action}
		Let $ \iota \colon N\hookrightarrow M $ be a transversal of a Lie algebroid $ A\laover M $ with $ \codim N=1 $ and $ a\in \Gamma^\infty(A) $ the Euler-like section from Lemma~\ref{lemma:existence_of_eulerlikes}. Then the map
		\begin{equation}
		j_{\blup(a)}\at{N}\colon \Omega^\bullet(\blup(A,\iota^! A))\to \Omega^{\bullet-1}(\iota^! A)
		\end{equation} 
		is a chain map.
	\end{corollary}

	As a last ingredient we need the concept of an adapted distance function, see~\cite[Section 2.1]{marcut.osornotorres:2014a}.
	\begin{definition}
		Let $ E\to N $ be a vector bundle equipped with a metric. A function $ \lambda\colon E\setminus N\to \field R_0^+ $ is called an \textup{adapted distance function} if the following hold:
		\begin{enumerate}
			\item $ \lambda\in \Cinfty(E\setminus N) $.
			\item For all $ x\in E\setminus N $ with $ |x|< \tfrac{1}{2} $, $ \lambda(x)=|x| $.
			\item $ \lambda(x)=1 $ for all $ x\in E $ with $ |x|>1 $.
		\end{enumerate}
	\end{definition}
	
	Note that such a function always exists, see e.g.~\cite[Section 8.5]{geudens:2017a}. With the notion of an adapted distance function at hand, we can proof Theorem~\ref{theorem:blup_of_codim1_transversal}.
	
	\begin{proof}[of Theorem~\ref{theorem:blup_of_codim1_transversal}]
		First note that $ j_{\blup(a)}\at{N}\circ p_A^\ast=0 $ is clear as $ a\at{N}=0 $. Suppose that $ \omega\in \Omega^k(\blup(A,\iota^! A)) $ is mapped to $ 0 $ under $ j_{\blup(a)}\at{N} $. Since we have the short exact sequence
		\begin{equation*}
		\begin{tikzpicture}[baseline=(current bounding box.center)]
		\node (0r) at (0,0) {$ 0 $};
		\node (L) at (1.5,0) {$ L $};
		\node (Blup) at (4,0) {$ \blup(A,\iota^! A)_N $};
		\node (B) at (6.5,0) {$ \iota^! A $};
		\node (0l) at (8,0) {$ 0,$};
		
		\path[->]
		(0r) edge node[]{$  $} (L)
		(L) edge node[]{$  $} (Blup)
		(Blup) edge node[above]{$ p_A $} (B)
		(B) edge node[]{$  $} (0l)
		
		;
		\end{tikzpicture}
		\end{equation*}
		this just means that $ \omega\at{N} $ actually is the pullback of a form on $ \iota^! A $. By computing in local coordinates one can then easily show that $ \omega\in p^\ast\Omega^k(A) $: If $ \{ b_1, \dots,b_k,e \} $ is a local frame an adapted chart $ U\subseteq M $, $ N\cap U=\{x=0\} $ such that the collection of $ b's $ yield a local frame for $ \iota^! A $ when restricted to $ N $, a form of the blowup is generated by forms that locally looks like
		\begin{equation*}
		\omega= f \frac{e^\ast}{x}\wedge (b^\ast)^{I}+ (b^\ast)^J,
		\end{equation*}
		for some $ f\in \Cinfty(N) $, denoting the dual frames with a $ \argument^\ast $ and $ I,J $ multi-indices, meaning that e.g.\ $ (b^\ast)^{\{ i_1,i_2 \}}=b_{i_1}^\ast\wedge b_{i_2}^\ast $. Then the condition $ j_{\blup(a)}\omega\at{N}=0 $ implies that $ f=xg $ since $\blup(a)\at{N}$ is nowhere vanishing. 
        Finally, by Corollary~\ref{corollary:blup_of_euler_trivialises_action} the sequence \eqref{eq:mazzeo_melrose_ses} is indeed a sequence of chain complexes.
		
		To define a splitting let $ \lambda(r) $ be an adapted distance function on $ E $. Then $ \D \log \lambda $ is a form on $ E\setminus N $ dual to the vertical Euler vector field is a neighbourhood of $N$, thus there exists $ \rho^\ast\D\log\lambda\in \Omega^1(A\at{E\setminus N}) $. The corresponding form $ \alpha=\rho_\blup^\ast\D\log\lambda $ extends smoothly to $ N $ (as can again be seen in local coordinates) and satisfies $ \alpha(\blup(a))=1 $ on $ N $ by continuity. Moreover, it is compactly supported in fiber direction of $ E $ (by definition of $ \lambda $) and closed as $ \rho_\blup^\ast $ is a chain map and $ \D \log\lambda $ is closed on the dense subset $ E\setminus N $. Thus we can define the map
		\begin{equation*}
		\Omega^{\bullet-1}(\iota^! A)\ni \tilde{\omega}\mapsto p_A^\ast(\pr^!)^\ast\tilde{\omega}\wedge \alpha \in \Omega^\bullet(\blup(A,\iota^! A)).
		\end{equation*}
        By closedness of $ \alpha $ this is a chain map, thus defining the desired splitting. In particular, $ j_{\blup(a)}\at{N} $ is surjective, hence the sequence is exact.
		But now it is clear that in the decomposition
		\begin{equation*}
		\homology^\bullet(\blup(A,\iota^! A))=\homology^\bullet(A)\oplus \homology^{\bullet-1}(\iota^! A),
		\end{equation*}
		the pullback by the blow-down map maps into the first factor.
	\end{proof}
	
\begin{remark}
    One way to interpret the closed $ 1 $-form $ \alpha=\rho_{\blup}^\ast \D \log \lambda $ from the proof of Theorem~\ref{theorem:blup_of_codim1_transversal} is as a Lie algebroid splitting for~\eqref{eq:diagram_bluptrans_ses_of_LA}: Restricted to $ N $ it is still closed, hence we can identify
	\begin{equation}
	\iota^! A\simeq \ker\alpha\at{N}\subseteq \blup(A,\iota^! A)_N
	\end{equation}
	as Lie algebroids. Thus, by Corollary~\ref{corollary:cohom_for_trivial_extension_class} we find
	\begin{equation}
	\Homology{\blup(A,\iota^! A)_N}{\bullet}{}=\Homology{\iota^! A}{\bullet}{}\oplus \Homology{\iota^! A}{\bullet-1}{}
	\end{equation}
	using that $ L $ is trivialisable such that the representation of $ \iota^! A $ becomes the trivial one. This observation alone is enough to prove Theorem~\ref{theorem:blup_of_codim1_transversal}. Indeed, let $ E\to N $ denote the tubular neighbourhood induced by the chosen Euler-like section $ a\in \Gamma^\infty(A) $ and $ \field L $ its tautological line bundle (see Remark \ref{remark:blup_of_zero_section}). Then there is an isomorphism
	\begin{equation}
	\Homology{\blup(A,\iota^! A)_{\field L}}{\bullet}{}\cong \Homology{\blup(A,\iota^! A)_{N}}{\bullet}{}
	\end{equation}
	induced by the restriction, see \cite[Theorem 3.24]{marcut.schuessler:2024a} in combination with Lemma \ref{lemma:existence_of_eulerlikes}. Together with $ \Homology{A_E}{\bullet}{}\cong \Homology{\iota^! A}{\bullet}{} $ we find that in cohomology we obtain a long exact sequence
	\begin{equation}
	\begin{tikzpicture}[baseline=(current bounding box.center)]
	\node (1) at (0,0) {$ \dots $};
	\node (2) at (2,0) {$ \Homology{A_E}{k}{} $};
	\node (3) at (5.3,0) {$ \Homology{\blup(A_E,\iota^! A)}{k}{} $};
	\node (4) at (8.6,0) {$ \Homology{\iota^! A}{k-1}{} $};
	\node (5) at (11.3,0) {$ \Homology{A_E}{k+1}{} $};
	\node (6) at (13.3,0) {$ \dots $};
	
	\path[->]
	(1) edge node[]{$  $} (2)
	(2) edge node[above]{$ p_A^\ast $} (3)
	(3) edge node[above]{$ f $} (4)
	(4) edge node[above]{$ g $} (5)
	(5) edge node[]{$  $} (6);
	
	\end{tikzpicture}
	\end{equation}
	in which $ f $ is surjective and $ g $ is zero. Moreover, this identifies 
	\begin{equation}
	\homology^\bullet\left(\frac{\Omega^\bullet(\blup(A_E,\iota^! A))}{p_A^\ast\Omega^\bullet(A_E)}\right)=\Homology{\iota^! A}{\bullet-1}{}.
	\end{equation}
	Then Theorem~\ref{theorem:blup_of_codim1_transversal} can be proven with the same local-to-global technique used to prove  Theorem~\ref{theorem:blup_transversals:past_in_cohomology}, \ref{theorem_item:past_in_cohomology:even} in Section~\ref{sec:proof_of_transversals}.
\end{remark}

	\subsection{From arbitrary codimension to codimension 1}
	
	The next step to proving Theorem \ref{theorem:blup_transversals:past_in_cohomology} is to show that for a transversal $ \iota\colon N \hookrightarrow M $ of any codimension the blowup is isomorphic to the blowup of a codimension $ 1 $ transversal in a different Lie algebroid, namely in $ p^! A $. Since $ N $ is a transversal and, for every $x\in  \field P $, $ T_{p(x)}N\subseteq T_x pT_x\blup(M,N) $, the blow-down map is transverse to the anchor, hence $p^! A$ is a well-defined Lie algebroid. Moreover, at the end of the subsection we show that $(p^!)^\ast$, like $p_A^\ast$, is an isomorphism when restricted to flat forms.
	
	\begin{proposition}\label{prop:transv_blup_is_blup_of_pullback}
		Let $ A\laover M $ be a Lie algebroid and $ \iota\colon N\hookrightarrow M $ a transversal. Denoting the blow-down map of the base by $ p\colon \blup(M,N)\to M $ and the inclusion of the projective bundle by $ \iota_\field{P}\colon \field P\hookrightarrow \blup(M,N) $, there is an isomorphism of Lie algebroids
		\begin{equation}
		\blup(A,\iota^! A)\cong \blup(p^! A, \iota_\field{P}^! p^! A)
		\end{equation}
		over the identity.
		The corresponding map on sections is given by
		\begin{equation}\label{prop:transv_blup_is_blup_of_pullback:eq:section}
		\Gamma^\infty(\blup(A,\iota^! A))\ni b \mapsto ( p_A(b), \rho_{\blup}(b))\in \Gamma^\infty(p^! A,\iota_\field{P}^!p^!A),
		\end{equation}
		identifying $ \Gamma^\infty(\blup(p^!A,\iota_\field{P}^! p^!A ))=\Gamma^\infty(p^!A,\iota_\field{P}^! p^!A ) $. Here, $ p_A\colon \blup(A,\iota^!A)\to A $ denotes the blow-down map. In particular, under this isomorphism $ p_A $ and $ p^!\circ p_{p^! A} $ coincide.
		\begin{proof}
			First note that $ \iota_\field{P}\colon \field P\hookrightarrow \blup(M,N) $ is a transversal in $ p^! A $ because $ \iota\colon N\hookrightarrow M $ is. For $ X\in \Gamma^\infty(TM,TN) $ we denote by $ \tilde{X}\in \mathfrak{X}(\blup(M,N)) $ the vector field with $ \tilde{X}\sim_p X $, which exists by Lemma~\ref{lemma:lifting_vf_to_the_blowup}. We first check that~\eqref{prop:transv_blup_is_blup_of_pullback:eq:section} is well-defined. Since both spaces are $ \Cinfty(\blup(M,N)) $-modules and \eqref{prop:transv_blup_is_blup_of_pullback:eq:section} is compatible with the module structure, it is enough to check well-definedness on a set of generators, namely $ \blup(\Gamma^\infty(A,\iota^! A)) $. Thus, let $ a\in \Gamma^\infty(A,\iota^! A) $ be given. Then
			\begin{equation*}
			\blup(a)\mapsto (p_A(\blup(a)), \rho_{\blup}(\blup(a))=(a\circ p, \widetilde{\rho(a)}).
			\end{equation*}
			Since $ \widetilde{\rho(a)}\sim_p \rho(a) $, $ (a\circ p, \widetilde{\rho(a)})\in \Gamma^\infty(p^! A) $ follows. Moreover, by Lemma~\ref{lemma:lifting_vf_to_the_blowup}, $ \widetilde{\rho(a)} $ is tangent to $ T\field P $, so $ \blup(a) $ maps to a section in $ \Gamma^\infty(p^! A, \iota_\field{P}^! p^! A) $ as claimed. 
			
			Now it is immediate to see that outside of $ \field P $, Equation~\eqref{prop:transv_blup_is_blup_of_pullback:eq:section} gives an isomorphism (of vector bundles). Thus it is enough to consider the neighbourhood around $ N $ provided by the normal form theorem~\ref{theorem:normal_form_transversals}. Suppose $ \pr\colon E=M\to N $ is a vector bundle and $ A=\pr^! \iota^! A $. As in the proof of Lemma \ref{lemma:existence_of_eulerlikes}, let $ \{ a_1,\dots,a_k \} $ be a local frame of $ \iota^!A $ and $ \{ s_1,\dots,s_{\codim N} \} $ a local frame of $ E $ over some common open $ U\subseteq N $ (we will use the dual frame on $ E $ as fibre coordinates on $ E_U $). According to Remark~\ref{remark:local_frames_for_blups}, for $ \beta\in\{1,\dots,\codim N\} $ we get a frame for $ \blup(A,\iota^! A)_{U_\beta} $ by
			\begin{equation*}
			\{ \blup(\tilde{a}_1),\dots,\blup(\tilde{a}_k),\blup(s^\beta \tilde{s}_1),\dots,\blup(s^\beta\tilde{s}_{\codim N}) \}, 
			\end{equation*}
			where
			\begin{equation*}
			\begin{aligned}
			\tilde{a}_i&=( a_i\circ \pr, (\rho_{\iota^! A}(a_i))^\hor )\\
			\tilde{s}_j&=( 0, s_j^\ver )
			\end{aligned}
			\end{equation*}
			for the horizontal lift of some linear connection on $ E\to N $.
			Then~\eqref{prop:transv_blup_is_blup_of_pullback:eq:section} maps
			\begin{equation*}
			\begin{aligned}
			\blup(\tilde{a}_j)&\mapsto (\tilde{a}_j\circ p, \widetilde{\rho(a_j)}),\\
			\blup(s^\beta\tilde{s}_\alpha)&\mapsto (0,\widetilde{s^\beta \Partial{s^\alpha}} ).
			\end{aligned}
			\end{equation*}
			The above collection of sections still span $ \Gamma^\infty(p^! A, \iota_\field{P}^! p^! A) $ locally over $ U_\beta $ as	
            \begin{equation*}
			\widetilde{s^\beta \Partial{s^\alpha}}=\begin{cases}
			\Partial{\tilde{s}^\alpha} &\text{ if }\alpha\neq \beta\\
			s^\beta\Partial{\tilde{s}^\beta}-\sum_{\gamma\neq \alpha }\tilde{s}^\gamma\Partial{\tilde{s}^\gamma} &\text{ if }\alpha=\beta.
			\end{cases}
			\end{equation*}
			Finally, we check that it is a morphism of Lie algebroids. Clearly, \eqref{prop:transv_blup_is_blup_of_pullback:eq:section} preserves anchors by its very definition. Thus it is enough to check compatibility with the Lie bracket using a set of generators. Let $ a,a'\in \Gamma^\infty(A,\iota^! A) $ be given, then
			\begin{equation*}
			\begin{aligned}
			[ \blup(a),\blup(a') ]\mapsto &(p_A(\blup([a,a'])), \widetilde{\rho([a,a'])} )\\
			=& ([a,a']\circ p, \widetilde{[\rho(a),\rho(a')]} )\\
			=& ([a,a']\circ p, [\widetilde{\rho(a)},\widetilde{\rho(a')}] )\\
			=& [ (a\circ p, \widetilde{\rho(a)} ), (a'\circ p,\widetilde{\rho(a')}) ],
			\end{aligned}
			\end{equation*}
			hence~\eqref{prop:transv_blup_is_blup_of_pullback:eq:section} constitutes an isomorphism of Lie algebroids.
		\end{proof}
	\end{proposition}
	
Using Theorem \ref{theorem:blup_of_codim1_transversal}, the identification $\blup(A,\iota^! A)\cong \blup(p^! A, \iota_\field{P}^! p^! A)$ from Proposition \ref{prop:transv_blup_is_blup_of_pullback} implies the following.
	\begin{corollary}\label{corollary:homology_of_trans_intermsof_pullback}
		Let $ \iota\colon N\hookrightarrow M $ be a transversal of a Lie algebroid $ A\laover M $. Then
		\begin{equation}
		\homology^\bullet(\blup(A,\iota^! A))=\homology^\bullet(p^! A)\oplus \homology^{\bullet-1}(\iota_{\field P}^! p^! A),
		\end{equation}
		where $ \blowdown{p}{M}{N} $ denotes the blow-down map of the base manifolds and $ \iota_{\field P}\colon \field P\hookrightarrow \blup(M,N) $ the inclusion of the projective bundle. Under this identification, $ p_A^\ast\colon \homology^\bullet(A)\to \homology^\bullet(\blup(A,\iota^! A)) $ maps into $ \homology^\bullet(p^! A) $ and is given by $ (p^!)^\ast $.
	\end{corollary}
	
Thus, to compute the cohomology of the blowup one needs to compute the cohomology of $ p^! A $. 
By the normal form theorem \ref{theorem:normal_form_transversals} and \eqref{eq:transversals_local_cohomology}, locally this comes down to comparing the cohomology of $ \iota^!A $ to that of $ \pi_{\field P}^!A $, the pullback to a projective bundle. And since in this case, $ (p^!)^\ast $ too constitutes an isomorphism when restricted to flat forms, the local picture will be enough.
	
	\begin{lemma}\label{lemma:p^!_is_iso_on_flat}
		Let $ N\subseteq M $ be a transversal of a Lie algebroid $ A\laover M $. Denoting the blow-down map of the base by $ \blowdown{p}{M}{N} $, the map
		\begin{equation}
		(p^!)^\ast\colon \Omega_N^\bullet(A)\to \Omega_{\field P}^\bullet(p^! A)
		\end{equation}
		is an isomorphism.
		\begin{proof}
			We know that the diagram
			\begin{equation*}
			\begin{tikzpicture}[baseline=(current bounding box.center)]
			\node (1) at (0.5,0) {$ \Omega^\bullet(A) $};
			\node (2) at (0,2) {$ \Omega^\bullet(\blup(A,\iota^! A)) $};
			\node (4) at (3.5,0) {$ \Omega^\bullet(p^! A) $};
			\node (3) at (4,2) {$ \Omega^\bullet(\blup(p^! A,p^! \iota^! A)) $};
			
			\path[->]
			(1) edge node[left]{$ p_A^\ast $} (2)
			(1) edge node[above]{$ (p^!)^\ast $}  (4)
			(2) edge node[above]{$ \Phi^\ast $} node[below]{$ \simeq $} (3)
			(4) edge node[right]{$ p_{p^! A}^\ast $} (3)
			;
			\end{tikzpicture}
			\end{equation*}
			commutes, see Proposition~\ref{prop:transv_blup_is_blup_of_pullback} for the upper isomorphism. We see that the maps factor to flat forms, and there $ p_A^\ast $ and $ p_{p^!A}^\ast $ are isomorphisms, thus so is $ (p^!)^\ast $.
		\end{proof}
	\end{lemma}
	
	\begin{lemma}\label{lemma:cohomology_p^!A_locally_is_nice}
		Let $ E\to N $ be the tubular neighbourhood corresponding to an Euler-like section $ a\in \Gamma^\infty(A) $. Then the inclusions $ \iota^! A\hookrightarrow A_E $ and $ \pi_{\field P}^!\iota^! A\hookrightarrow p^! A_E $ induce isomorphisms such that
		\begin{equation}\label{eq:cohomology_p^!A_locally_is_nice}
		\begin{tikzpicture}[baseline=(current bounding box.center)]
		\node (12) at (0,2) {$ \homology^\bullet (A_E) $};
		\node (22) at (2.2,2) {$ \homology^\bullet(\iota^! A) $};
		\node (11) at (0,0) {$ \homology^\bullet(p^! A_E) $};
		\node (21) at (2.2,0) {$ \homology^\bullet (\pi_{\field P}^! \iota^! A) $};
		
		\path[]
		(12) edge[white] node[black]{$ \cong $} (22)
		(11) edge[white] node[black]{$ \cong $} (21);
		
		\path[->]
		(12) edge node[left] {$ (p^!)^\ast $} (11)
		(22) edge node[right] {$ (\pi_{\field P}^!)^\ast $} (21)
		;
		\end{tikzpicture}
		\end{equation}
		commutes. Here, $ \pi_{\field P}\colon \field P\to N $ denotes the projection of the projective bundle.

        If $\codim N$ is odd, then all maps in \eqref{eq:cohomology_p^!A_locally_is_nice} are isomorphisms.
		\begin{proof}
			The inclusions yield a commutative diagram
			\begin{equation*}
			\begin{tikzpicture}[baseline=(current bounding box.center)]
			\node (12) at (0,1.5) {$  A_E $};
			\node (22) at (2,1.5) {$\iota^! A $};
			\node (11) at (0,0) {$ p^! A_E $};
			\node (21) at (2,0) {$ \pi_{\field P}^! \iota^! A $};
			
			\path[->]
			(11) edge node[left] {$ p^! $} (12)
			(21) edge node[right] {$ \pi_{\field P}^! $} (22)
			(21) edge node[]{$  $} (11)
			(22) edge node[]{$  $} (12)
			;
			\end{tikzpicture}
			\end{equation*}
			of Lie algebroids, leading to~\eqref{eq:cohomology_p^!A_locally_is_nice} in cohomology, where the inclusions become isomorphisms by~\eqref{eq:transversals_local_cohomology}. If $\codim N$ is odd, Theorem \ref{theorem:blup_transversals:projective_bundles} shows the remaining statement.
		\end{proof}
	\end{lemma}

	\subsection{Proof of Theorem~\ref{theorem:blup_transversals:past_in_cohomology}}\label{sec:proof_of_transversals}
	The third part of Theorem~\ref{theorem:blup_transversals:past_in_cohomology} is more technical, thus we prove the remaining statements of Theorem \ref{theorem:blup_transversals:past_in_cohomology} first.
	
	\begin{proof}[of Theorem~\ref{theorem:blup_transversals:past_in_cohomology}, \ref{theorem_item:past_in_cohomology:iso_and_codim1} and \ref{theorem_item:past_in_cohomology:odd}]
		For the first part, the isomorphism 
		\begin{equation*}
		\Homology{\blup(A,\iota^! A)}{\bullet}{}\cong\Homology{\blup(p^! A,\pi_{\field P}^! A)}{\bullet}{}
		\end{equation*} 
		is Proposition \ref{prop:transv_blup_is_blup_of_pullback}, while $ \Homology{\blup(A,\iota^! A)}{\bullet}{}\cong\Homology{p^! A}{\bullet}{}\oplus \Homology{\pi_{\field P}^! A}{\bullet-1}{} $ is Theorem~\ref{theorem:blup_of_codim1_transversal} since $ \field P$ is a transversal of codimension $ 1 $ inside $p^! A$. For the second part note that we find a neighbourhood $E$ of $ \field P $ in $ \blup(M,N) $ such that $ \Homology{p^! A_{E}}{\bullet}{}\cong \Homology{\pi_{\field P}^!A}{\bullet}{} $ by means of an Euler-like section. Thus we have
		\begin{equation*}
		\homology^\bullet\left(\frac{\Omega^\bullet(p^! A)}{(p^!)^\ast \Omega^\bullet(A)}\right)\cong\homology^\bullet\left( \frac{\Omega^\bullet(p^! A_E)}{(p^!)^\ast\Omega^\bullet(A_E)} \right)=0,
		\end{equation*}
		where the first isomorphism is a direct consequence of Lemma~\ref{lemma:p^!_is_iso_on_flat} (see also~Theorem~\ref{theorem:flatiso_and_consequences}), and the second identification follows from Lemma \ref{lemma:cohomology_p^!A_locally_is_nice}. From this, 
		\begin{equation*}
		(p^!)^\ast\colon \Homology{A}{\bullet}{}\xrightarrow{\simeq}\Homology{p^!A}{\bullet}{}
		\end{equation*}
		follows at once, implying $ \Homology{\blup(A,\iota^! A)}{\bullet}{}=\Homology{A}{\bullet}{}\oplus \Homology{\iota^!A}{\bullet-1}{} $ by the first part.
	\end{proof}
	
	To show Theorem~\ref{theorem:blup_transversals:past_in_cohomology}, \ref{theorem_item:past_in_cohomology:even} we first work in a tubular neighbourhood $ \pr\colon E\to N $ induced by an Euler-like section of $ A $ along $ N $. 
	\begin{lemma}
		Let $ \pr\colon E\to N $ be a vector bundle with zero section $ \iota\colon N\hookrightarrow E $ and $ B\laover N $ a Lie algebroid. Then there is a long exact sequence
		\begin{equation}
		\begin{tikzpicture}[baseline=(current bounding box.center)]
		
		\node (1) at (0,0) {$ \dots $};
		\node (2) at (2,0) {$ \homology^\bullet (\pr^! B) $};
		\node (3) at (4.5,0) {$ \homology^\bullet ( p^! \pr^! B ) $};
		\node (4) at (7.5,0) {$ \homology^{\bullet+1}_\mathrm{cv}(\pr^! B) $};
		\node (5) at (10,0) {$ \Homology{\pr^! B}{\bullet+1}{} $};
		\node (6) at (12,0) {$ \dots $};
		
		\path[->]
		(1) edge node[]{$  $} (2)
		(2) edge node[above]{$ (p^!)^\ast $} (3)
		(3) edge node[above]{$ (p^!)_\ast $} (4)
		(4) edge node[above]{$ i $} (5)
		(5) edge node[]{$  $} (6)
		
		;
		\end{tikzpicture}
		\end{equation}
		Here, $ \blowdown{p}{E}{N} $ is the blow-down map of the base, $ (p^!)_\ast $ denotes the composition of the Thom isomorphism after fibre integrating the restriction of a form to $ \field P $, and $ i\colon \homology^\bullet_\mathrm{cv}(\pr^! B)\to \homology^\bullet(\pr^! B) $ denotes the map that regards a form compactly supported in fibre direction as just a form on $ \pr^! B $.
		\begin{proof}
			This follows from Theorem~\ref{theorem:blup_transversals:projective_bundles}, \ref{theorem_item:projective_bundles:gysin} using the same reasoning as in Lemma~\ref{lemma:thom_euler_diagram}, as we again have both
			\begin{equation*}
			\begin{aligned}
			\Homology{\pr^! B}{\bullet}{}&\cong \Homology{ B}{\bullet}{} \text{ and } \\
			\Homology{p^! B}{\bullet}{} &\cong \Homology{\pi_\field{P}^! B}{\bullet}{}
			\end{aligned}
			\end{equation*}
			by means of the respective restrictions. 
		\end{proof}
	\end{lemma}
	
	This already gives Theorem~\ref{theorem:blup_transversals:past_in_cohomology}, \ref{theorem_item:past_in_cohomology:even} in the local setting. To prove the global statement, we use that by Lemma~\ref{lemma:p^!_is_iso_on_flat} for any tubular neighbourhood $ E\to N $,
	\begin{equation}
	Q^\bullet :=\homology^\bullet\left(\frac{\Omega^\bullet(p^! A)}{(p^!)^\ast \Omega^\bullet(A)}\right)=\homology^\bullet\left( \frac{\Omega^\bullet(p^! A_E)}{(p^!)^\ast\Omega^\bullet(A_E)} \right) =:Q^\bullet_E,
	\end{equation}
	where equality already holds on the level of chain complexes. The plan is to identify $ Q^\bullet $ and $ \homology^{\bullet+1}_\mathrm{cv}(A_E) $ to complete the proof of Theorem~\ref{theorem:blup_transversals:past_in_cohomology}. We can define a map $ Q^\bullet_E\to \homology_\mathrm{cv}^{\bullet+1}(A_E) $ in the following way: First, pick any splitting $ \sigma $ of the short exact sequence
	\begin{equation*}
	\begin{tikzpicture}[baseline=(current bounding box.center)]
	\node (1) at (0,0) {$ 0 $};
	\node (2) at (1.7,0) {$ \Omega^\bullet(A_E) $};
	\node (3) at (4.2,0) {$ \Omega^\bullet(p^! A_E) $};
	\node (4) at (7,0) {$ \frac{\Omega^\bullet(p^! A_E)}{(p^!)^\ast\Omega^\bullet(A_E)} $};
	\node (5) at (9,0) {$ 0 $};
	
	\path[->]
	(1) edge node[]{$  $} (2)
	(2) edge node[above]{$ (p^!)^\ast $} (3)
	(4) edge node[]{$  $} (5);
	
	\begin{scope}[transform canvas={yshift=.3em}]
	\draw [<-] (4) edge node[above]{$ \tau $} (3);
	\end{scope}
	\begin{scope}[transform canvas={yshift=-.3em}]
	\draw [<-, dashed] (3) edge node[below]{$ \sigma $} (4);
	\end{scope}

	;
	\end{tikzpicture}
	\end{equation*}
	and let $ \chi\in \Cinfty(E) $ be a smooth function with compact vertical support such that $ \chi=1 $ in a neighbourhood of $ N $. Then also $ \tilde{\sigma}= p^\ast\chi \sigma $ defines a splitting for $ \tau $, since by Lemma~\ref{lemma:p^!_is_iso_on_flat} $\tau$ depends only on local data around $ \field P $, and there the map did not change. We proceed now to define a map the same way one would construct the edge endomorphism in the corresponding long exact sequence in cohomology: For $ [\lambda]\in Q^\bullet_E $ we know that $ \tau \D \tilde{\sigma}\lambda=\D\lambda=0 $, so we can find a unique $ \eta\in \Omega^\bullet(A_E) $ such that $ (p^!)^\ast \eta= \D \tilde{\sigma}\lambda $. Consequently, $ \eta $ has compact vertical support, and since $ (p^!)^\ast $ is injective, $ \D\eta=0 $. We map $ [\lambda] $ to $ [\eta]\in \homology^{\bullet+1}_\mathrm{cv}(A_E) $. This is well-defined: If $ \lambda'\in [\lambda] $, then $ \tau( \tilde{\sigma}\lambda-\tilde{\sigma}\lambda' )=0 $, thus there is $ \xi\in \Omega^\bullet_\mathrm{cv}(A_E) $ with $ (p^!)^\ast \xi=\tilde{\sigma}\lambda-\tilde{\sigma}\lambda' $. But then $ (p^!)^\ast \D\xi=\D(\tilde{\sigma}\lambda-\tilde{\sigma}\lambda') $, showing that the cohomology class $ [\eta]\in\homology^{\bullet+1}_\mathrm{cv}(A_E)  $ does not depend on the chosen representative.
	
	\begin{lemma}\label{lemma:iso_between_Q_and_Hcf}
		The above constructed map $ \Psi\colon Q^\bullet_E\to \homology_\mathrm{cv}^{\bullet+1}(A_E) $ is an isomorphism.
		\begin{proof}
			We will show this by using the $ 5 $-lemma \cite[Chapter 1, Lemma 3.3]{maclane:1963b}. Note that $ Q^\bullet_E $ and $ \homology_\mathrm{cv}^{\bullet+1}(A_E) $ fit into a long exact sequence
			\begin{equation*}
			\begin{tikzpicture}[baseline=(current bounding box.center)]
			\node (1) at (0.3,2) {$ \dots $};
			\node (2) at (2,2) {$ \homology^\bullet(A_E) $};
			\node (3) at (4.4,2) {$ \homology^\bullet(p^!A_E) $};
			\node (4) at (7.1,2) {$ Q^\bullet_E $};
			\node (5) at (9.5,2) {$ \homology^{\bullet+1}(A_E)  $};
			\node (6) at (12.4,2) {$ \homology^{\bullet+1}(p^!A_E) $};
			\node (7) at (14.4,2) {$ \dots $};
			\node (q) at (0.3,0) {$ \dots $};
			\node (w) at (2,0) {$ \homology^\bullet(A_E) $};
			\node (e) at (4.4,0) {$ \homology^\bullet(p^!A_E) $};
			\node (r) at (7.1,0) {$ \homology_\mathrm{cv}^{\bullet+1}(A_E) $};
			\node (t) at (9.5,0) {$ \homology^{\bullet+1}(A_E) $};
			\node (z) at (12.4,0) {$ \homology^{\bullet+1}(p^!A_E) $};
			\node (u) at (14.4,0) {$ \dots $};
			
			\path[->]
			(1) edge node[]{$  $} (2)
			(2) edge node[above]{$ (p^!)^\ast $} (3)
			(3) edge node[above]{$ \tau $} (4)
			(4) edge node[above]{$ \delta $} (5)
			(5) edge node[above]{$ (p^!)^\ast $} (6)
			(6) edge node[]{$  $} (7)
			
			(q) edge node[]{$  $} (w)
			(w) edge node[above]{$ (p^!)^\ast $} (e)
			(e) edge node[above]{$ -2(p^!)_\ast $} (r)
			(r) edge node[above]{$ i $} (t)
			(t) edge node[above]{$ (p^!)^\ast $} (z)
			(z) edge node[]{$  $} (u)
			
			(2) edge node[]{$  $} (w)
			(3) edge node[]{$  $} (e)
			(4) edge node[left]{$ \Psi $} (r)
			(5) edge node[]{$  $} (t)
			(6) edge node[]{$  $} (z)
			
			;
			\end{tikzpicture}
			\end{equation*}
			where $ \delta $ denotes the edge homomorphism and unlabeled vertical arrows are the identity. By the 5-Lemma, if the diagram commutes, $ \Psi $ is an isomorphism. Thus we have to check if the squares left and right to $ \Psi $ are commuting. For the square on the right this is obvious, since $ \Psi $ is defined in the same way one constructs the edge homomorphism. 
   
   Consider the square on the left and let $ [\omega]\in \homology^\bullet(p^! A_E) $ be given. We first compute $ -2(p^!)_\ast[\omega] $ explicitly and then check that it coincides with $ \Psi\tau[\omega] $. By the choice of $ E $, we find $ \eta\in \Omega^\bullet(\pi_{\field P}^! \iota^! A) $ such that $ [\omega]=[(\pi_{\field L}^!)^\ast\eta] $ by \eqref{eq:transversals_local_cohomology}. Here, $ \pi_{\field L}\colon \field L\to\field P $ and $ \pi_{\field P}\colon \field P\to N $ denote the bundle projections of the tautological line bundle and the projective bundle, respectively. We can also consider the double cover of $ \field P $ given by the sphere bundle $ \mathfrak{p}\colon \field S\to \field P $ and the surjective submersion $ \mathfrak{q}\colon E\setminus 0\to \field S $ to pull $ \eta $ back to a form $ (\mathfrak{q}^!)^\ast (\mathfrak{p}^!)^\ast \eta $ on $ A_{E\setminus 0} $. Pick a fibre metric on $ E $ and let $ \chi\in \Cinfty(E) $ be a function such that $ \chi=1 $ in a neighbourhood of $ N $, $ \chi $ only depends on the radius and $ \chi(x)=0 $ for all $ x\in E $ with $ |x|>1 $. Then
			\begin{equation*}
			\hat{\eta}=\D\chi\wedge (\mathfrak{q}^!)^\ast (\mathfrak{p}^!)^\ast \eta
			\end{equation*}
			defines a closed form in $ \Omega^\bullet_\mathrm{cv}(A_E) $ that is send to $ -2(p^!)_\ast [\omega] $ by fibre integration, hence
			\begin{equation*}
			-2(p^!)_\ast [\omega] = [\hat{\eta}].
			\end{equation*}
			Next, we have
			\begin{equation*}
			\begin{aligned}
			(p^!)^\ast \hat{\eta}&= (p^!)^\ast (\D\chi\wedge (\mathfrak{q}^!)^\ast (\mathfrak{p}^!)^\ast \eta)\\
			&= \D (p^\ast \chi)\wedge ((\mathfrak{p}\circ \mathfrak{q}\circ p)^!)^\ast \eta\\
			&= \D (p^\ast \chi (\pi_{\field L}^!)^\ast \eta ).
			\end{aligned}
			\end{equation*}
			So choosing the cut-off function in the definition of $ \Psi $ to be $ p^\ast \chi $ and the splitting $ \sigma $ such that $ \sigma(\tau(\pi^!_\field L)^\ast\eta)=(\pi^!_\field L)^\ast\eta $, it follows that the left square commutes and thus $ \Psi $ constitutes an isomorphism.
		\end{proof}
	\end{lemma}
	
	We can now complete the proof of Theorem~\ref{theorem:blup_transversals:past_in_cohomology}. 
	
	\begin{proof}[of Theorem~\ref{theorem:blup_transversals:past_in_cohomology}, \ref{theorem_item:past_in_cohomology:even}]
		We want to show that for a transversal $ \iota\colon N\hookrightarrow M $ of a Lie algebroid $ A\laover M $ of even codimension any Euler-like section of $ A $ along $ N $ gives rise to a long exact sequence
		\begin{equation*}
		\begin{tikzpicture}[baseline=(current bounding box.center)]
		\node (1) at (0,0) {$ \dots $};
		\node (2) at (2,0) {$ \homology^\bullet(A) $};
		\node (3) at (4.5,0) {$ \homology^\bullet(p^!A) $};
		\node (4) at (7,0) {$ \homology_\mathrm{cv}^{\bullet+1} (A_E) $};
		\node (5) at (9.5,0) {$ \homology^{\bullet+1}(A) $};
		\node (6) at (11.5,0) {$ \dots $};
		
		\path[->]
		(1) edge node[]{$  $} (2)
		(2) edge node[above]{$ (p^!)^\ast $} (3)
		(3) edge node[above]{$ (p^!)_\ast $} (4)
		(4) edge node[above]{$ i $} (5)
		(5) edge node[]{$  $} (6)
		
		;
		\end{tikzpicture}
		\end{equation*}
		where $ \blowdown{p}{M}{N} $ denotes the blow-down map of the base. Since $ (p^!)^\ast\colon \Omega^\bullet(A)\to \Omega^\bullet(p^! A) $ is a chain map, we obtain a long exact sequence
		\begin{equation*}
		\begin{tikzpicture}[baseline=(current bounding box.center)]
		\node (1) at (0,0) {$ \dots $};
		\node (2) at (2,0) {$ \homology^\bullet(A) $};
		\node (3) at (4.5,0) {$ \homology^\bullet(p^!A) $};
		\node (4) at (6.5,0) {$ Q^\bullet $};
		\node (5) at (8.5,0) {$ \homology^{\bullet+1}(A) $};
		\node (6) at (10.5,0) {$ \dots $};
		
		\path[->]
		(1) edge node[]{$  $} (2)
		(2) edge node[above]{$ (p^!)^\ast $} (3)
		(3) edge node[above]{$ \tau $} (4)
		(4) edge node[above]{$ \delta $} (5)
		(5) edge node[]{$  $} (6)
		
		;
		\end{tikzpicture}
		\end{equation*}
		In this long exact sequence, by Lemma~\ref{lemma:p^!_is_iso_on_flat} the map $ \tau $ only depends on local data, i.e.\ writing $ \iota_E\colon A_E\to A $ we have $  \tau_E\circ\homology(\iota_E^\ast)=\tau $. Moreover, the connecting homomorphism $ \delta $ is defined using any kind of splitting, so in particular we can choose one with support localised inside $ E $. Then the result follows from (the proof of) Lemma~\ref{lemma:iso_between_Q_and_Hcf}. But then the rest of the statement follows from the identification
		\begin{equation*}
		\homology^\bullet(\blup(A,\iota^! A))=\homology^\bullet(p^! A)\oplus \homology^{\bullet-1}(\pi_{\field P}^! A)
		\end{equation*}
  from Corollary \ref{corollary:homology_of_trans_intermsof_pullback}.
	\end{proof}
	
	\section{Invariant submanifolds}\label{section:invariant_submanifolds}
Suppose the orbit foliation of $A\laover M$ is singular with $N\subset M$ a closed and embedded singular leaf. Then one can hope to increase the dimensions of the orbit by blowing up $ A_N $, possibly turning $\blup(A, A_N)$ into a regular Lie algebroid with more easily computable cohomology. In this section we discuss two examples: The action Lie algebroids $\liealg{so}(3)\ltimes \field R^3$ and $\liealg{sl}_2(\field R)\ltimes \field R^3$ with singular leaves given by the origin, where in both cases we consider the adjoint action. For $\liealg{so}(3)\ltimes \field R^3$ we find that blowing up the origin resolves the singularity, allowing to compute the cohomology of both the blowup and $\liealg{so}(3)\ltimes \field R^3$, while for $\liealg{sl}_2(\field R)\ltimes \field R^3$ we find that even iterated blowups will not result in a regular Lie algebroid. 

To compute the cohomology of $ \blup(A,A_N) $ there exists no auxiliary statement analog to Theorem \ref{theorem:normal_form_transversals}. Instead, one can work directly on formal cohomology, for which we can utilise the Serre spectral sequence developed in \cite{marcut.schuessler:2024a}. If $i\colon L \hookrightarrow A$ is any Lie subalgebroid, the differential ideal
 \begin{equation*}
     \mathcal{I}(L):=\ker (i^\ast\colon \Omega^\bullet(A)\to \Omega^\bullet(L))
 \end{equation*}
 induces a filtration on $\Omega^\bullet(A)$ by
 \begin{equation*}
    \mathcal{F}^p\Omega^\bullet(A):= (\mathcal{I}(L)^{\wedge p}\Omega(A))^\bullet.
 \end{equation*}
 This filtration gives rise to a spectral sequence, called the \textit{Serre spectral sequence associated to $L$} in \cite{marcut.schuessler:2024a}, which we will denote by $\{(E_L)^{\bullet,\bullet}_r\}_{r\geq 0}$.
 Recall the following results on formal cohomology along (invariant) submanifolds \cite[Theorem 3.17, Theorem 3.18]{marcut.schuessler:2024a}.

    \begin{theorem}\label{theorem:formal_cohomology_ss}
        Let $N\subset M$ be a closed and embedded  submanifold of a Lie algebroid $A\laover M$. 
        \begin{enumerate}
        \item The Serre spectral sequence associated to any Lie subalgebroid $L\laover N$ converges to formal Lie algebroid cohomology along $N$.
            \item \label{theorem:formal_cohomology_ss:item:E1} If $N$ is an invariant submanifold the first page of the Serre spectral sequence associated to $A_N$ is given by 
            \begin{equation}
                (E_{A_N})_1^{p,q}\simeq\mathrm{H}^{p+q}(A_N, S^p\nu^\ast)\implies \mathrm{H}^{p+q}(\mathscr{J}_N^\infty\Omega^\bullet(A)).
            \end{equation}
            \item If $A$ is linearisable around $N$ then 
            \begin{equation}
                \mathrm{H}^\bullet(\mathscr{J}_N^\infty\Omega^\bullet(A))=\prod_{p=0}^\infty \mathrm{H}^{\bullet}(A_N, S^p\nu^\ast).
            \end{equation}
        \end{enumerate}
    \end{theorem}

The representation of $A_N$ on the normal bundle $\nu$ of $N$ in $M$ is given by (see e.g.\ \cite[Section 6.1]{fernandes.marcut:2022a})
\begin{equation*}
     \nabla_a (X|_N \text{ mod }TN)=[\rho( \tilde{a}),X]|_N \text{ mod }TN,
\end{equation*}
where $X\in \vfield(M) $ and $\tilde{a}\in \Gamma^\infty(A)$ is some extension of $a\in \Gamma^\infty(A_N)$. The Lie algebroid $A\laover M$ is called linearisable around $N$ if, around $N$, it is diffeomorphic to the action Lie algebroid $A\ltimes \nu$, see \cite[Theorem 2.4]{higgins.mackenzie:1990a} for the general construction of action Lie algebroids.

When blowing up the restriction of a Lie algebroid to an invariant submanifold, the blow-down map naturally induces a map between spectral sequences.
\begin{lemma}\label{lemma:p_A_in_spectral_sequence}
    Let $A\laover M$ be a Lie algebroid and $N\subseteq M$ a closed and embedded invariant submanifold. 
    \begin{enumerate}
        \item The blow-down map $p_A\colon \blup(A, A_N)\to A$ induces a map between Serre spectral sequences
        \begin{equation}
            p_A^\ast\colon \{(E_{A_N})^{\bullet,\bullet}_r\}_{r\geq 0}\to \{(\tilde{E}_{\blup(A,A_N)_{\field P}})^{\bullet,\bullet}_r\}_{r\geq 0}.
        \end{equation}
        \item If the induced map
        \begin{equation}
            p_A^\ast\colon \mathrm{H}^{\bullet}(A_N, S^p\nu(M,N)^\ast)\to \mathrm{H}^{\bullet}(\blup(A, A_N)_{\field P}, S^p\nu(\blup(M,N),\field P)^\ast)
        \end{equation}
        is an isomorphism for all $p\in \field N$ then 
        \begin{equation}
            p_A^\ast\colon \mathrm{H}^{\bullet}(\mathscr{J}_N^\infty\Omega^\bullet(A))\xrightarrow{\simeq}\mathrm{H}^{\bullet}(\mathscr{J}_{\field P}^\infty\Omega^\bullet(\blup(A,A_N)))
        \end{equation}
        is an isomorphism.
    \end{enumerate}
    \begin{proof}
        It is easy to check that $p_A^\ast$ respects the corresponding filtrations and thus induces a map between spectral sequences. Then the second part follows from Theorem \ref{theorem:formal_cohomology_ss}, \ref{theorem:formal_cohomology_ss:item:E1}.\ and the Mapping Lemma for spectral sequences \cite[Lemma 5.2.4]{weibel:1994a}.
    \end{proof}
\end{lemma}

To compute cohomologies in this section we also need the spectral sequence of an abelian extension introduced in \cite{mackenzie:2005a}.

\begin{definition}
A Lie algebroid $A\laover M$ is called \textup{abelian extension} if it fits into a short exact sequence of Lie algebroids
    \begin{equation}\label{eq:ses_abelian_extension}
	\begin{tikzpicture}[baseline=(current bounding box.center)]
	\node (1) at (0,0) {$ 0 $};
	\node (2) at (1.5,0) {$ L $};
	\node (3) at (3,0) {$ A $};
	\node (4) at (4.5,0) {$ B $};
	\node (5) at (6,0) {$ 0 $};
	
	\path[->]
	(1) edge node[]{$  $} (2)
	(2) edge node[above]{$ i $} (3)
	(3) edge node[above]{$ \Phi $} (4)
	(4) edge node[]{$  $} (5);
	\end{tikzpicture}
	\end{equation}
 over $M$ such that $L\laover M$ is abelian. 
\end{definition}

There is a representation of $ B $ on $ L $ given by $ \nabla^L_b \ell=[ a,i(\ell) ] $, where $ a\in \Gamma^\infty(A) $ satisfies $ \Phi(a)=b\in \Gamma^\infty(B) $ and $ \ell\in \Gamma^\infty(L) $, see \cite[Proposition 3.3.20]{mackenzie:2005a}. The \textit{extension class} $[\gamma]\in \mathrm{H}^2(B,L)$ of an abelian extension is defined via a splitting $\sigma\colon B\to A$ of \eqref{eq:ses_abelian_extension} by the class of the induced curvature tensor $\gamma\in \Omega^2(B,L)$, where
\begin{equation}\label{eq:extension_class}
    \gamma(b_1,b_2)=[ \sigma (b_1),\sigma (b_2) ]_A-\sigma([b_1,b_2]_B) \in \Gamma^\infty(L)
\end{equation}
for $b_1,b_2\in \Gamma^\infty(B)$. Note that $[\gamma]$ is independent of the chosen splitting. Then one obtains the following (see \cite[Corollary 4.2]{marcut.zeiser:2022a} for $\Phi$ the anchor of regular Lie algebroids, and \cite[Theorem 7.1]{marcut.schuessler:2024a} for a Lie algebroid $B\laover Q$ over an arbitrary manifold).

\begin{theorem}\label{theorem:abelian_ss_second_page}
    Given an abelian extension \eqref{eq:ses_abelian_extension}, there is a spectral sequence converging to the cohomology of $A$ with second page given by
    \begin{equation}
        \big(d_2: E_2^{p,q}\to E_2^{p+2,q-1}\big) \simeq \big( (-1)^p\mathrm{i}_{[\gamma]}\colon  \mathrm{H}^{p}(B,\Lambda^qL^\ast)\to \mathrm{H}^{p+2}(B,\Lambda^{q-1}L^\ast)\big).
    \end{equation}
\end{theorem}

\begin{corollary}\label{corollary:cohom_for_trivial_extension_class}
    If the abelian extension \eqref{eq:ses_abelian_extension} splits as Lie algebroids, i.e.\ the extension class is zero, and $L=M\times \field R$ is trivial with trivial representation of $B$, the cohomology of $A$ is given by
    \begin{equation}
        \homology^\bullet(A)=\homology^\bullet(B)\oplus\homology^{\bullet-1}(B).
    \end{equation}
\end{corollary}

	\subsection{The action Lie algebroid $ \liealg{so}(3)\ltimes \field R^3 $}\label{sec:so3}
	
	In this section we consider $ A=\liealg{so}(3)\ltimes \field R^3 $, the action Lie algebroid corresponding to infinitesimal rotations in $ \field R^3 $, which is linear around the origin. More abstractly, $A\laover \field R^3$ is a trivial vector bundle with global frame $\{e_1,e_2,e_3\}$, and bracket and anchor given by
     \begin{align}
         [e_i,e_j]&=\sum_{k=1}^3\varepsilon_{ijk} e_k, \label{eq:so3_bracket}\\
         \rho(e_i)&=\sum_{j,k=1}^3\varepsilon_{ijk} x_j \partial_k.
     \end{align}

 Since $ \liealg{so}(3) $ is compact, averaging~\cite{ginzburg.weinstein:1992a} shows that
	\begin{equation}\label{so3:eq:cohomology_averaging}
	\Homology{A}{\bullet}{}\simeq\Homology{\liealg{so}(3)}{\bullet}{}\tensor \Cinfty(\field R^3)^{\liealg{so}(3)}=\algebra A\oplus 0\oplus 0\oplus\algebra A,
	\end{equation}
	identifying $ \Cinfty(\field R^3)^{\liealg{so}(3)}=\algebra A=\{ f\in \Cinfty(\field R)\colon f(-x)=f(x)\text{ for all }x\in \field R \} $ with the even functions on $ \field R $, as $ \Cinfty(\field R^3)^{\liealg{so}(3)}=\{ f\in \Cinfty(\field R^3)\colon f\text{ only depends on the radius} \} $. Blowing up the Lie subalgebroid $ A_{\{0\}} $ results in yet another action Lie algebroid,
	\begin{equation}
	\blup(A,A_{\{0\}})=\liealg{so}(3)\ltimes \blup(\field R^3,\{0\}),
	\end{equation}
	see~\cite[Corollary 5.93]{obster:2021a}, and we write 
    \begin{equation}
        \tilde{e}_i:=\blup(e_i)=p^\sharp e_i \in \Gamma^\infty(\blup(A,A_{\{0\}})).
    \end{equation}
We compute the cohomology of the blowup and can reproduce \eqref{so3:eq:cohomology_averaging}.
	\begin{proposition}
		Let $ A=\liealg{so}(3)\ltimes \field R^3 $. Then
		\begin{equation}
		\Homology{A}{\bullet}{}=\Homology{\liealg{so}(3)\ltimes \blup(\field R^3,\{0\})}{\bullet}{}=\algebra A\oplus 0\oplus 0\oplus\algebra A.
		\end{equation}
	\end{proposition}
	
	We first show the second equality by computing the cohomology of the blowup using the spectral sequence for abelian extensions and then show that the cohomology of the blowup is isomorphic to the cohomology of $ A $ by proving that the formal cohomologies are isomorphic. First we see that the blowup indeed is an abelian extension.
	
	\begin{lemma}
		The blowup $ \blup(A,A_{\{0\}}) $ is an abelian extension, i.e.\ it fits into an exact sequence of Lie algebroids
		\begin{equation}
		\begin{tikzpicture}[baseline=(current bounding box.center)]
		\node (1) at (0,0) {$ 0 $};
		\node (2) at (2,0) {$ L=\ker \rho_{\blup} $};
		\node (3) at (5,0) {$ \blup(A,A_{\{0\}}) $};
		\node (4) at (7.5,0) {$ T\foliation{F} $};
		\node (5) at (8.8,0) {$ 0 $};
		
		\draw[->]
		(1) edge node[]{$  $} (2)
		(2) edge node[above]{$  $} (3)
		(3) edge node[above]{$ \rho_{\blup} $} (4)
		(4) edge node[]{$  $} (5);
		
		\end{tikzpicture}
		\end{equation}
		where $ \foliation{F} $ denotes the orbit foliation of $ \blup(A,A_{\{0\}}) $ given by
		\begin{equation}
		\foliation{F}=\{ p^{-1}( \{x_1^2+x_2^2+x_3^3=r^2 \} )\}_{r\geq 0}.
		\end{equation}
		In particular, $ \blup(A,A_{\{0\}}) $ is a regular Lie algebroid.
		\begin{proof}
Recall that the blow-down map $ \blowdown{p_A}{A}{A_{\{0\}}} $ restricted to $ \blup(\field R^3,\{0\})\setminus \field P $ is a diffeomorphism, thus on this set the orbit foliation is the same as on $ \field R^3\setminus \{0\} $ under $ p $. 
To see that $ \field P $ is a single orbit, since the situation is highly symmetric it is enough to consider one of the charts of $ \blup(\field R^3,\{0\}) $ from Remark~\ref{remark:charts_for_blup}, say $ U_1 $.
Using $ (\tilde{x}_1,\tilde{x}_2,\tilde{x}_3) $ to denote the coordinates on $ U_1 $ we find 
			\begin{equation}\label{eq:so3_rho_in_U1_coords}
			\begin{aligned}
			\rho_{\blup}(\tilde{e}_1)\at{U_1}&= \tilde{x}_3\tilde{\partial}_2-\tilde{x}_2\tilde{\partial}_3\\
			\rho_{\blup}(\tilde{e}_2)\at{U_1}&=-\tilde{x}_1\tilde{x}_3\tilde{\partial}_1+\tilde{x}_2\tilde{x}_3\tilde{\partial}_2+(1+\tilde{x}_3^2)\tilde{\partial}_3\\
			\rho_{\blup}(\tilde{e}_3)\at{U_1}&= \tilde{x}_1\tilde{x}_2\tilde{\partial}_1-(1+\tilde{x}_2^2)\tilde{\partial}_2-\tilde{x}_2\tilde{x}_3\tilde{\partial}_3.
			\end{aligned}
			\end{equation}
			Thus we see that on $ U_1\cap\field P=\{ \tilde{x}_1=0 \} $, the image of $ \rho_{\blup} $ still spans a two-dimensional distribution. Since $ \field P $ is connected, this implies that $ \field P $ is a single orbit.
		\end{proof}
	\end{lemma}
	
Thus we are in the framework of Theorem~\ref{theorem:abelian_ss_second_page}. Note that in this particular case the spectral sequence will stabilise after the second page as $ \rank L=1 $. 
As a vector bundle, $ L $ is given by the pullback of the tautological line bundle back to itself.
	
\begin{lemma}
Denoting the projection of the tautological line bundle $ \field L $ of $ \field P $ by $ \pi_{\field L}\colon \field L\to\field P $, the kernel of $ \rho_{\blup} $ is isomorphic to 
\begin{equation}
L=\ker \rho_{\blup}=\pi_{\field L}^\sharp \field L,
\end{equation}
by mapping 
\begin{equation}\label{eq:so3_identify_L_with_tautological_bundle}
\pi_{\field L}^\sharp \field L\ni((v^1,v^2,v^3)_{[v']})_{w_{[v']}}\mapsto (v^1\tilde{e}_1+v^2\tilde{e}_2+v^3\tilde{e}_3)\at{w_{[v']}}\in L.
\end{equation}
\begin{proof}
    The defined map identifies $ \pi_{\field L}^\sharp \field L $ with a rank $ 1 $ subbundle of $ \blup(A,A_{\{0\}}) $. Thus, we only need to show that its image lies in the kernel of $ \rho_{\blup} $. In the view of the chart for $ U_1 $ of the blowup, recall that the canonical chart for $ U_1 $ already is a vector bundle chart for $ \field L_{U_1\cap \field P} $ with $ \tilde{x}_1 $ as the fibre coordinate. Here, the map~\eqref{eq:so3_identify_L_with_tautological_bundle} becomes
    \begin{equation*}
    ((\tilde{x}_1)_{[1:\tilde{x}_2:\tilde{x}_3]})_{(\tilde{x}_1',\tilde{x}_2,\tilde{x}_3)}\mapsto \tilde{x}_1(\tilde{e}_1+\tilde{x}_2\tilde{e}_2+\tilde{x}_3\tilde{e}_3).
    \end{equation*}
    But given~\eqref{eq:so3_rho_in_U1_coords} it is easy to check that points of the form $ \tilde{e}_1+\tilde{x}_2\tilde{e}_2+\tilde{x}_3\tilde{e}_3 $ lie in the kernel of $ \rho_{\blup} $.
\end{proof}
\end{lemma}

In particular, $ L $ is not a trivial vector bundle. Pulled back to a double cover, however, we can even trivialise the action of the orbit foliation.

\begin{lemma}\label{lemma:so3_properties_of_globalpullback}
Consider the double cover $ \pr\colon \field S^2\times \field R\to\blup(\field R^3,\{0\}) $ defined by
\begin{equation}
\pr\colon (x,t)\mapsto \begin{cases}
p^{-1}(tx) &\text{ if }t\neq 0\\
[x] &\text{ if }t=0
\end{cases}
\end{equation}
where we view $ \field S^2 $ as the unit sphere in $ \field R^3 $. Then the induced abelian extension
\begin{equation}\label{eq:so3_abelian_extension_of_pullback}
\begin{tikzpicture}[baseline=(current bounding box.center)]
\node (1) at (0,0) {$ 0 $};
\node (2) at (1.5,0) {$\pr^! L $};
\node (3) at (4.2,0) {$ \pr^!\blup(A,A_{\{0\}}) $};
\node (4) at (7.2,0) {$ \pr^!T\foliation{F} $};
\node (5) at (8.8,0) {$ 0, $};

\draw[->]
(1) edge node[]{$  $} (2)
(2) edge node[above]{$  $} (3)
(3) edge node[above]{$ \rho_{\pr^!} $} (4)
(4) edge node[]{$  $} (5);

\end{tikzpicture}
\end{equation}
where $ \rho_{\pr^!} $ denotes the anchor of $ \pr^!\blup(A,A_{\{0\}}) $, has the following properties:
\begin{enumerate}
\item The Lie algebroid $\pr^! \blup(A,A_{\{0\}})$  is a product Lie algebroid
    \begin{equation}
        \pr^! \blup(A,A_{\{0\}})\simeq \pr^! \blup(A,A_{\{0\}})_{\field S^2\times \{1\}}\times (0\laover \field R).
    \end{equation}
    \item The Lie algebroid $ \pr^! T\foliation{F} $ is the foliation Lie algebroid of the foliation
    \begin{equation}
    \{ \field S^2\times\{ t \} \}_{t\in\field R}.
    \end{equation}
    \item There exists a non-vanishing section $ g\in \Gamma^\infty(\pr^! L) $, anti-invariant under the action of $ \field Z_2 $, which is constant under the action of $ \pr^! \foliation{F} $.
\end{enumerate}
\begin{proof}
    For the first part, recall that 
    \begin{equation*}
        \blup(A, A_{\{0\}})=\liealg{so}(3)\ltimes \blup(\field R^3,\{0\})
    \end{equation*}
    is an action Lie algebroid. Since $\pr\colon \field S^2\times \field R\to \field L $ is a double cover, we obtain 
    \begin{equation*}
        \pr^!\blup(A, A_{\{0\}})=\liealg{so}(3)\ltimes (\field S^2\times \field R).
    \end{equation*}
    We show that the action of $\liealg{so}(3)$ is independent of the $\field R$ variable, from which the statement follows.
    
    Note that out of charts for the blowup we get charts on $ U_i^\pm=\pr^{-1}(U_i) $ in the natural way. Consider $(U_1^\pm,\tilde{x}\circ\pr)$. The coordinate transformation
    \begin{equation*}
        \begin{aligned}
            y_1&=\tilde{x}_1\sqrt{1+\tilde{x}_2^2+\tilde{x}_3^2},\\
            y_2&=\tilde{x}_2\\
            y_3&=\tilde{x}_3
        \end{aligned}
    \end{equation*}
    gives a product chart for $U_1^\pm\times \field R$, where $y_1$ is the coordinate on $\field R$. In this chart, the anchor is given by
    \begin{equation*}
        \begin{aligned}
            \rho(\pr^!\tilde{e}_1)&=y_3\Partial{y_2}-y_2\Partial{y_3}\\
            \rho(\pr^!\tilde{e}_2)&=-y_2y_3\Partial{y_2}-(1+y_3^2)\Partial{y_3}\\
            \rho(\pr^!\tilde{e}_3)&=(1+y_2^2)\Partial{y_2}+y_2y_3\Partial{y_3}.
        \end{aligned}
    \end{equation*}
    Thus the action of $\liealg{so}(3)$ does not depend on the $\field R$ variable $y_1$.
    
    Using the first part, the second part is clear. For the last part, define a section over $ U_i^\pm $ of $ \pr^! L $ by $ (g_i^\pm,0) $, where
    \begin{equation}\label{eq:so3_section_g}
    g_i^\pm=\frac{\pm 1}{\sqrt{1+\sum_{j\neq i} y_j^2}} ( \pr^!\tilde{e}_i+\sum_{j\neq i}y_j\pr^!\tilde{e}_j ).
    \end{equation}
    These definitions agree on overlaps and thus define a global trivialising section, which is constant under the action of $ \pr^! T\foliation{F} $ as a straightforward computation shows.
\end{proof}
\end{lemma}

This is sufficient to compute the spaces on the second page of the spectral sequence associated to the abelian extension \eqref{eq:so3_abelian_extension_of_pullback}  (Theorem \ref{theorem:abelian_ss_second_page}). For the differential on the second page we show that the extension class does not vanish in a way that $ \D_2\colon E_2^{0,2}\to E_2^{1,0} $ is an isomorphism.

\begin{proposition}\label{so3_prop:cohomology_of_blowup}
The cohomology of $ \blup(A,A_{\{0\}}) $ is given by
\begin{equation}
\Homology{\blup(A,A_{\{0\}})}{\bullet}{}=\algebra A\oplus 0\oplus 0\oplus \algebra A,
\end{equation}
where $ \algebra A=\{ f\in \Cinfty(\field R)\colon f(-x)=f(x)\text{ for all }x\in \field R \} $.
\begin{proof}
    By Theorem~\ref{theorem:abelian_ss_second_page} and Lemma~\ref{lemma:so3_properties_of_globalpullback} the nontrivial entries on the second page of the spectral sequence associated to \eqref{eq:so3_abelian_extension_of_pullback} are given by
    \begin{equation*}
    \begin{tikzpicture}[baseline=(current bounding box.center)]
    \node (1) at (0,1) {$ \Cinfty(\field R) $};
    \node (2) at (2,1) {$ 0 $};
    \node (3) at (4,1) {$ \Cinfty(\field R) $};
    \node (q) at (0,0) {$ \Cinfty(\field R) $};
    \node (w) at (2,0) {$ 0 $};
    \node (e) at (4,0) {$ \Cinfty(\field R) $};
    
    \path[->, dashed]
    (1) edge node[fill=white]{$ \D_2 $} (e)
    
    ;
    \end{tikzpicture}
    \end{equation*}
    as $ \Homology{\pr^!T\foliation{F}}{\bullet}{}=\Cinfty(\field R)\oplus0\oplus\Cinfty(\field R) $ by \cite[Lemma 5.4]{marcut.schuessler:2024a}, see also \cite[Lemma 3]{crainic.marcut:2013a}. Note that in the third column we integrated along the spheres of the foliation to identify the cohomology with functions on $ \field R $, and that the $ \field Z_2 $-action reverses the orientation. 

    For the differential $ \D_2 $ we compute the extension class~\eqref{eq:extension_class}. To do so, note that by Lemma \ref{lemma:so3_properties_of_globalpullback} it is enough to consider the short exact sequence \eqref{eq:so3_abelian_extension_of_pullback} restricted to a single fibre $\field S^2\times \{1\}$. We show that the extension class is given by the volume form on $\field S^2$. Consider $\field S^2 \subseteq \field R^3$, identify the trivial vector bundles $\pr^! \blup(A,A_{\{0\}})_{\field S^2\times \{1\}}\simeq T\field R^3\at{\field S^2}$, and identify for $x\in \field S^2$ the tangent space
    \begin{equation*}
        T_x\field S^2\simeq \{v\in \field T_x\field R^3\colon \langle x,v\rangle=0\} \subseteq T_x\field R^3.
    \end{equation*}
    Note that this gives a splitting $\sigma$ of the short exact sequence. To determine the curvature, note that under these identifications, the section $g\in \Gamma^\infty(\pr^! L)$ defined in \eqref{eq:so3_section_g} is given by the outward pointing unit normal vector field of $\field S^2$. The Lie bracket on the constant sections of $T\field R^3\at{\field S^2}$ is the cross product. 
    
    For $V\in \vfield(\field S^2)$, we write $\sigma(V)=v^ie_i$ for some $v^i\in \Cinfty(\field S^2)$. Let $V,W\in \vfield (\field S^2)$ be given. Pairing the curvature with the trivialising normal vector field yields
    \begin{equation*}
        \begin{aligned}
            \langle x,\gamma(V,W) \rangle&=\langle x, [\sigma(V),\sigma(W)]\rangle\\
            &= \langle x, \rho(\sigma(V))w^\ell e_\ell-\rho(\sigma(W))v^\ell e_\ell\rangle+\langle x,\sigma(V)\times\sigma(W)\rangle\\
            &=3\langle x,\sigma(V)\times\sigma(W)\rangle \\
            &= 3\mathrm{i}_{\sigma(W)}\mathrm{i}_{\sigma(V)}\mathrm{i}_x\mathrm{vol}_{\field R^3},
        \end{aligned}
    \end{equation*}
    since
    \begin{equation*}
        \begin{aligned}
            \langle x, \rho(\sigma(V))w^\ell e_\ell&-\rho(\sigma(W))v^\ell e_\ell \rangle =\sum_{i,j,k,\ell=1}^3 \big(\varepsilon_{ijk} x_j x_\ell (v^i (\partial_k w^\ell)-w^i (\partial_k v^\ell))\big)\\
            &= \sum_{i,j,k,\ell=1}^3 \big(\varepsilon_{ijk} x_j (v^i (\partial_k \underbrace{ x_\ell w^\ell}_{=0})-v^i\delta_{k\ell} w^\ell -w^i (\partial_k \underbrace{x_\ell v^\ell}_{=0} )+w^i\delta_{kl}v^\ell)    \big)\\
            &=\sum_{i,j,k=1}^3\big(\varepsilon_{ijk} x_j (-v^iw^k+w^iv^k)    \big)\\
            &=\sum_{i,j,k=1}^3 2\varepsilon_{ijk} x_jw^i v^k\\
            &= 2 \langle x, \sigma(V)\times \sigma(W)\rangle.
        \end{aligned}
    \end{equation*}
    Thus $\gamma$ is a multiple of the volume form on $\field S^2$ and, in conclusion, wedging with the extension class of \eqref{eq:so3_abelian_extension_of_pullback} is an isomorphism. Thus, the cohomology of $ \pr^! \blup(A,A_{\{0\}}) $ is given by
    \begin{equation*}
    \Homology{\pr^!\blup(A,A_{\{0\}})}{\bullet}{}=\Cinfty(\field R)\oplus 0\oplus 0\oplus\Cinfty(\field R).
    \end{equation*}
    Using Lemma~\ref{lemma:lapullback_cohomology} the statement follows: In degree $ 0 $ the invariant functions give the cohomology of the blowup, whereas in degree $ 3 $, since we trivialised the representation, we need anti-invariant classes of forms. But after integrating anti-invariant forms correspond to invariant functions as the $ \field Z_2 $-action reverses the orientation.
\end{proof}
\end{proposition}

Next, we compute the cohomology associated to the representation of $ \blup(A,A_{\{0\}})_{\field P} $ on the normal bundle $ \nu(\blup(\field R^3,\{0\}),\field P)=\field L $. For the original Lie algebroid we know by \cite[Theorem 3.5]{ginzburg.weinstein:1992a} that
\begin{equation}\label{so3_eq:formal_cohomology_so3}
\Homology{A_{\{0\}}}{k}{,S^\ell (\field R^3)^\ast}=\begin{cases}
\field R &\text{ if }k=0,3\text{ and }\ell\text{ even}\\
0 &\text{ otherwise}.
\end{cases}
\end{equation}
We show that \eqref{so3_eq:formal_cohomology_so3} is isomorphic to $ \Homology{\blup(A,A_{\{0\}})_{\field P}}{\bullet}{,S^\ell \field L^\ast} $, which by Lemma \ref{lemma:p_A_in_spectral_sequence}
implies that 
\begin{equation*}
\mathrm{H}^\bullet(\mathscr{J}^\infty_{\field P}\Omega^\bullet(\blup(A,A_{\{0\}}))) \simeq \mathrm{H}^\bullet(\mathscr{J}^\infty_{\{0\}}\Omega^\bullet(A)).
\end{equation*}
Then the cohomology of the quotient complex
\begin{equation}
\frac{\mathscr{J}^\infty_{\field P}\Omega^\bullet(\blup(A,A_{\{0\}}))}{p_A^\ast\mathscr{J}^\infty_{\{0\}}\Omega^\bullet(A)}
\end{equation}
vanishes.

\begin{proposition}
The cohomology of $ \Omega^\bullet(\blup(A,A_{\{0\}})_{\field P},S^\ell \field L^\ast) $ is given by
\begin{equation}
\Homology{\blup(A,A_{\{0\}})_{\field P}}{k}{,S^\ell \field L^\ast}=\begin{cases}
\field R &\text{ if }k=0,3\text{ and }\ell\text{ even}\\
0 &\text{ otherwise}.
\end{cases}
\end{equation}
In particular, 
\begin{equation}
p_A^\ast\colon\Homology{A}{\bullet}{}\to\Homology{\blup(A,A_{\{0\}})}{\bullet}{}
\end{equation}
is an isomorphism.
\begin{proof}

    We again make use of the double cover of $ \field P $ given by
    \begin{equation*}
    \pr\colon \field S^2\ni x\mapsto [x]\in \field P.
    \end{equation*}
    First, consider $\ell=1$. The constant section of the trivial bundle $ \pr^\sharp \field L=\field S^2\times \field R $ trivialises the representation. Note that these sections are anti-invariant under the $\field Z_2$-action. Now we can proceed as in the proof of Proposition~\ref{so3_prop:cohomology_of_blowup} using a spectral sequence argument to compute the cohomology of $ \Omega^\bullet(\pr^!\blup(A,A_{\{0\}})_{\field P},\pr^\sharp S^\ell \field L^\ast) $. 
    By Theorem~\ref{theorem:abelian_ss_second_page} the second page of the spectral sequence is given by
    \begin{equation*}
    \begin{tikzpicture}[baseline=(current bounding box.center)]
    \node (1) at (0,1) {$ \field R $};
    \node (2) at (1.5,1) {$ 0 $};
    \node (3) at (3,1) {$ \field R $};
    \node (q) at (0,0) {$ \field R $};
    \node (w) at (1.5,0) {$ 0 $};
    \node (e) at (3,0) {$ \field R $};
    
    \path[->, dashed]
    (1) edge node[fill=white]{$ \D_2 $} (e)
    
    ;
    \end{tikzpicture}
    \end{equation*}
    as $ \Homology{\field S^2}{\bullet}{}=\field R\oplus 0\oplus\field R $, using integration in degree $ 2 $, and the differential is again an isomorphism following a similar reasoning as in Proposition~\ref{so3_prop:cohomology_of_blowup}. Thus
    \begin{equation*}
    \Homology{\pr^!\blup(A,A_{\{0\}})_{\field P}}{\bullet}{,\pr^\sharp S^\ell \field L^\ast}=\field R\oplus 0\oplus 0\oplus \field R.
    \end{equation*}
    Since all cohomology groups are one-dimensional, $ \Homology{\blup(A,A_{\{0\}})_{\field P}}{k}{, S^\ell \field L^\ast} $ will either be trivial or $ \field R $. We argue that it is nontrivial if and only if $ \ell $ is even and $ k=0,3 $. Indeed, only in this case the trivialising section of $ \pr^\sharp S^\ell\field L^\ast  $ is invariant, thus the cohomological degree $ 0 $ part is invariant. In degree $ 3 $ the coefficients are tensored with one more copy of $ \field L^\ast $, but since we used integration which is anti-invariant under the $ \field Z_2 $-action, again $ \field R=\Homology{\blup(A,A_{\{0\}})_{\field P}}{3}{, S^\ell \field L^\ast}  $. If $ \ell $ is odd, similar reasoning shows that $ \Homology{\blup(A,A_{\{0\}})_{\field P}}{k}{, S^\ell \field L^\ast}=0 $. Hence, we have
    \begin{equation*}
    \Homology{\blup(A,A_{\{0\}})_{\field P}}{k}{,S^\ell \field L^\ast}\simeq \Homology{A_{\{0\}}}{k}{,S^\ell (\field R^3)^\ast}.
    \end{equation*}
The rest follows from Lemma \ref{lemma:p_A_in_spectral_sequence}.
\end{proof}
\end{proposition}

	\subsection{A non-example: $ \liealg{sl}_2(\field R)\ltimes \field R^3 $}\label{sec:sl2r}
	
	Consider the action Lie algebroid $ A=\liealg{sl}_2(\field R)\ltimes \field R^3 $ coming from the coadjoint action of $\liealg{sl_2(\field R)}$, which is a trivial vector bundle with global frame $\{e_1,e_2,e_3\}$ and anchor
 \begin{equation}
     \begin{aligned}
         \rho(e_1)&=-x_3\partial_2-x_2\partial_3  ,\\
         \rho(e_2)&= x_3\partial_1+x_1\partial_3 ,\\
         \rho(e_3)&= x_2\partial_1-x_1\partial_2.
     \end{aligned}
 \end{equation}
 The bracket can be expressed in terms of the global frame, but will not be used in this section.

 The orbits of this Lie algebroid are given by the connected components of the level sets of the function $ f=x_1^2+x_2^2-x_3^2 $, where $ f^{-1}(\{0\}) $ splits into three orbits $ x_3>0 $, $ x_3<0 $, and the origin.
 Hence, the origin is the only leaf of this foliation that is not two dimensional. In this section we show that by means of repeatedly blowing up restrictions of $ A $ (or its blowups) to orbits, one cannot construct a regular Lie algebroid.
	
	\begin{proposition}
		The restriction of the Lie algebroid $ \tilde{A}=\blup(A,A_{\{0\}}) $ to $ \field P $ consists of three orbits:
		\begin{itemize}
			\item A one-dimensional orbit
			\begin{equation}
			Z=\{ [x_1:x_2:x_3]\in \field P\colon x_1^2+x_2^2=x_3^2 \};
			\end{equation}
			\item Two two-dimensional orbits given by the connected components of $ \field P\setminus Z $, explicitly given by
			\begin{equation}
			\{ [x_1:x_2:x_3]\in \field P\colon x_1^2+x_2^2<x_3^2 \} \quad\text{ and }\quad \{ [x_1:x_2:x_3]\in \field P\colon x_1^2+x_2^2>x_3^2 \}.
			\end{equation}.
		\end{itemize}
		\begin{proof}
			In the charts of $ \blup(\field R^3,\{0\}) $ from Remark~\ref{remark:charts_for_blup} we can e.g.\ compute over $(U_1, \tilde{x})$
			\begin{equation*}
			\begin{aligned}
			\rho_{\blup}(p^\sharp e_1)\at{U_1}&=-(\tilde{x}_3\tilde{\partial}_2+\tilde{x}_2\tilde{\partial}_3),\\
			\rho_{\blup}(p^\sharp e_2)\at{U_1}&=\tilde{x}_1\tilde{x}_3\tilde{\partial}_1-\tilde{x}_2\tilde{x}_3\tilde{\partial}_2+(1-\tilde{x}_3^2)\tilde{\partial}_3,\\
			\rho_{\blup}(p^\sharp e_3)\at{U_1}&=\tilde{x}_1\tilde{x}_2\tilde{\partial}_1-(1+\tilde{x}_2^2)\tilde{\partial}_2-\tilde{x}_2\tilde{x}_3\tilde{\partial}_3.
			\end{aligned}
			\end{equation*}
			Thus, on the invariant submanifold $ \field P\cap U_1 $, the orbit foliation is spanned by
			\begin{equation*}
			\{ \underbrace{\tilde{x}_3\tilde{\partial}_2+\tilde{x}_2\tilde{\partial}_3}_{(I)}, \underbrace{(1+\tilde{x}_2^2)\tilde{\partial}_2+\tilde{x}_2\tilde{x}_3\tilde{\partial}_3}_{(II)}, \underbrace{\tilde{x}_2\tilde{x}_3\tilde{\partial}_2-(1-\tilde{x}_3^2)\tilde{\partial}_3}_{(III)} \}.
			\end{equation*}
			If  $ \tilde{x}_3=0 $ then this equals the span of $ \{\tilde{\partial}_2,\tilde{\partial}_3\} $, which is two-dimensional. If $ \tilde{x}_3\neq 0 $ then $ (I) $ and $ (II) $ span a two-dimensional subspace as long as $ 1+\tilde{x}_2^2-\tilde{x}_3^2\neq 0 $ as
			\begin{equation*}
			\det \begin{pmatrix}
			\tilde{x}_2 & \tilde{x}_3\\
			(1-\tilde{x}_3^2) & -\tilde{x}_2 \tilde{x}_3
			\end{pmatrix}=-\tilde{x}_3(1+\tilde{x}_2^2-\tilde{x}_3^2).
			\end{equation*}
			If $ \tilde{x}_2=0 $ and $ \tilde{x}_3= 1 $, i.e.\ it is a point in $ Z $, the span clearly is one-dimensional. Thus let $ \tilde{x}_2\neq0$ and $ 1+\tilde{x}_2^2=\tilde{x}_3^2 $. Then 
			\begin{equation}
			\begin{aligned}
			\tilde{x}_3(II)+\tilde{x}_2(I)&=\tilde{x}_3\tilde{\partial}_2+\tilde{x}_2\tilde{\partial}_3,\\
			\tilde{x}_3(II)-\tilde{x}_2(I)&=\tilde{x}_3(\tilde{x}_2^2+\tilde{x}_3^2)\tilde{\partial}_2+\tilde{x}_2(\tilde{x}_2^2+\tilde{x}_3^2)\tilde{\partial}_3,
			\end{aligned}
			\end{equation}
			both of which are multiples of $ \tilde{x}_3\tilde{\partial}_2+\tilde{x}_2\tilde{\partial}_3 $, which shows that the image of $ \rho $ over points in $ Z\cap U_1 $ is one-dimensional. Similar computations for the other charts show that the orbit foliation is indeed as stated.
		\end{proof}
	\end{proposition}
	
	Thus by blowing up the foliation has become less singular. Since $ Z $ is completely contained in $ U=U_3\setminus\{ [0:0:1] \} $, in showing that the remaining singularity cannot be resolved in the sense that always at least one orbit will be one-dimensional, we can restrict the discussion to this open subset. Firstly, in this chart let us introduce polar coordinates for $ (\tilde{x}_1,\tilde{x}_2) $ by 
	\begin{equation}
	(\tilde{x}_1,\tilde{x}_2,\tilde{x}_3)=((r+1)\cos\phi,(r+1)\sin\phi,\tilde{x}_3),
	\end{equation}
	where $ r\in (-1,\infty)  $.
	
	\begin{lemma}\label{sl2R:lemma:foliation_in_polar}
		Over $ U $ the image of the anchor map is generated by
		\begin{equation}\label{sl2R:eq:foliation_in_polar}
		\{ \partial_\phi, r(r+2)\partial_r-(r-1)\tilde{x}_3\tilde{\partial}_3 \}.
		\end{equation}
		\begin{proof}
			Over $ U_3 $, the anchor of the blowup maps
			\begin{equation*}
			\begin{aligned}
			\rho_{\blup}(p^\sharp e_1)\at{U_3}&=\tilde{x}_1\tilde{x}_2 \tilde{\partial}_1 + (\tilde{x}_2^2-1)\tilde{\partial}_2-\tilde{x}_2\tilde{x}_3\tilde{\partial}_3,\\
			\rho_{\blup}(p^\sharp e_2)\at{U_3}&=(1-\tilde{x}_1^2)\tilde{\partial}_1 -\tilde{x}_1\tilde{x}_2\tilde{\partial}_2 +\tilde{x}_1\tilde{x}_3 \tilde{\partial}_3,\\
			\rho_{\blup}(p^\sharp e_3)\at{U_3}&=\tilde{x}_2\tilde{\partial}_1-\tilde{x}_1\tilde{\partial}_2.
			\end{aligned}
			\end{equation*}
			The change of coordinates leads to
			\begin{equation*}
			\begin{aligned}
			\rho_{\blup}(p^\sharp e_1)\at{U}&=r(r+2)\sin\phi \partial_r -\tfrac{1}{r+1}\cos\phi \partial_\phi-(r+1)\tilde{x}_3\sin \phi \tilde{\partial}_3,\\
			\rho_{\blup}(p^\sharp e_2)\at{U} &=-r(r+2)\cos\phi\partial_r-\tfrac{1}{r+1}\sin\phi \partial_\phi +(r+1)\tilde{x}_3\cos\phi \tilde{\partial}_3,\\
			\rho_{\blup}(p^\sharp e_3)\at{U}&=-\partial_\phi,
			\end{aligned}
			\end{equation*}
			which implies the statement. Indeed, the collection $ \{ f_1,f_2,f_3 \}  $ with $ f_1=\sin \phi p^\sharp e_1-\cos\phi p^\sharp e_2 $, $ f_2=\cos\phi p^\sharp e_1+\sin\phi p^\sharp e_2 $ and $ f_3=p^\sharp e_3 $ constitutes a frame over $ U $, where $ f_3 $ and $ f_1 $ are mapped to~\eqref{sl2R:eq:foliation_in_polar}, while $ \rho_{\blup}(f_2) $ is a multiple of $ \partial_\phi $.
		\end{proof}
	\end{lemma}
	
	Setting $ s=\frac{r}{(r+2)^3} $ one can simplify \eqref{sl2R:eq:foliation_in_polar} in the sense that the image of the anchor is then spanned by
	\begin{equation}\label{sl2R:eq:better_foliation_in_polar}
	\{ \partial_\phi, 2s\partial_s+\tilde{x}_3\tilde{\partial}_3 \}.
	\end{equation}
	
	\begin{remark}
		From Lemma~\ref{sl2R:lemma:foliation_in_polar} we also find that $ \tilde{A}_{\field P} $ is an abelian extension
		\begin{equation}
		\begin{tikzpicture}[baseline=(current bounding box.center)]
		\node (1) at (0,0) {$ 0 $};
		\node (2) at (1.4,0) {$ L $};
		\node (3) at (2.8,0) {$ \phantom{\tilde{A}} $};
		\node (a) at (2.8,0.04) {$ \tilde{A}_{\field P} $};
		\node (4) at (5.2,0) {$ \blup(T\field P,TZ) $};
		\node (5) at (7.4,0) {$ 0 $};
		
		\path[->]
		(1) edge node[]{$  $} (2)
		(2) edge node[above]{$  $} (3)
		(3) edge node[above]{$  $} (4)
		(4) edge node[]{$  $} (5);
		;
		\end{tikzpicture}
		\end{equation}
		where $ (L\subseteq \ker\rho_{\blup})_{\field P} $ is an abelian Lie algebroid of rank $ 1 $. 
	\end{remark}
	
However, the generating sections~\eqref{sl2R:eq:better_foliation_in_polar} of $ \rho_{\blup}(\tilde{A}_U) $ show that we cannot remove the singularity by blowing up further. There will always remain at least one leaf of dimension $ 1 $.
	
	\begin{lemma}\label{lemma:when_blowups_no_help}
		Let $ U\subseteq \field R^n $ be open with $ 0\in U $ and $ a_1,\dots,a_n\in \Cinfty(U) $. Then the  section 
  \begin{equation}\label{eq:when_blowups_no_help}
      \sum_i x_ia_i\partial_i\in \vfield(U)
  \end{equation}
induces a section $ s\in \Gamma^\infty(\blup(TU,\{0\})) $ of the Lie algebroid blowup $ \blup(TU,\{0\}) $, and $ \rho_\blup(s) $ vanishes on the subset $ Z=\{ [1:0:\dots :0],\dots,[0:\dots :0:1] \} $. 

Moreover, around points in $Z$ the vector field $\rho_\blup(s)$ is again of the form \eqref{eq:when_blowups_no_help}.
		\begin{proof}
			Fix $ i\in \{1,\dots,n\} $. Writing $ p\colon \blup(U,\{0\})\to U $ for the blow-down map, we have inside the chart $(U_i,\tilde{x})$
			\begin{equation*}
			\rho_\blup(s)\at{U_i}=\tilde{x}_ip^\ast(a_i)\tilde{\partial}_i+\sum_{j\neq i}\tilde{x}_j(p^\ast(a_j)-p^\ast(a_i))\tilde{\partial}_j.
			\end{equation*}
			This section vanishes in the origin of $ U_i $, i.e.\ the point $ [0:,\dots,:\underbrace{1}_{i}:\dots:0] $, and is of the form \eqref{eq:when_blowups_no_help}.
		\end{proof}
	\end{lemma}
	
	\begin{corollary}
		Repeatedly blowing up the one-dimensional orbits will not result in a Lie algebroid that has a $ 2 $-dimensional orbit foliation. Every blowup along a single $ 1 $-dimensional orbit will increase the number of $ 1 $-dimensional orbits by $ 1 $.
		\begin{proof}
			After the first blowup the statement follows inductively from Lemma \ref{lemma:when_blowups_no_help}. Indeed, considering the sections \eqref{sl2R:eq:better_foliation_in_polar} that generate the image of the anchor map around the one-dimensional orbit, the section $2s\partial_s+\tilde{x}_3\tilde{\partial}_3$ satisfies the requirements of Lemma \ref{lemma:when_blowups_no_help}.
		\end{proof}
	\end{corollary}
	
	In conclusion, blowing up further will only increase the number of singularities.

	\appendix
\section{A Gysin sequence for Lie algebroids}
We develope a Gysin sequence for Lie algebroids, which we made use of in Section \ref{section:blow_up_of_transversals}. To be able to formulate the result we first introduce a notion of fibre integration for Lie algebroids.

	\subsection{Integration along fibres}\label{sec:integration_along_fibres}
	Throughout this section let $ \pr\colon F\to N $ denote a locally trivial fibre bundle with orientable and connected fibres and $ B\laover N $ a Lie algebroid. We aim to define a notion of integrating along fibres 
	\begin{equation}
	(\pi^!)_\ast\colon \homology^\bullet_\mathrm{cv}(\pi^! B)\to \homology^{\bullet-\rank F}(B,o(F)),
	\end{equation}
	where $ \homology^\bullet_\mathrm{cv}(\pi^! B) $ denotes compact vertical cohomology, see Definition~\ref{def:compact_vertical_cohomology}.
	Here, $ o(F)\to N $ denotes the orientation bundle of $ F $. This is standard for vector bundles, but can also be made sense of in our more general situation. Since the fibres are orientable we can construct a double cover $ \tilde{N} $ of $ N $ by
	\begin{equation}
	\tilde{N}_p=\{ \lambda\colon \lambda\text{ is an orientation on the manifold }F_p \}
	\end{equation}
	with the obvious smooth structure, which we call the \textit{orientation double cover}. On the trivial bundle $ \tilde{N}\times \field R\to \tilde{N} $ we have a $ \field Z_2 $ action given by
	\begin{equation}
	(-1).(\lambda,t)=(\bar{\lambda},-t),
	\end{equation}
	where $ \bar{\lambda} $ denotes the opposite orientation. Then
	\begin{equation}\label{eq:definition_orientation_bundle}
	o(F)=\tilde{N}\times \field R/\field Z_2 \to N.
	\end{equation}
	The bundle \eqref{eq:definition_orientation_bundle} is trivial if and only if there is a globally consistent way of choosing orientations for the fibres of $ F $. In this case, we call the fibre bundle $ F\to N $ \textit{orientable}.
	\begin{remark}
		If $ F\to N $ is a fibre bundle such that $ o(F) $ is not the trivial bundle, pulling back $ F $ to $ \tilde{N} $ will result in an orientable fibre bundle.
	\end{remark}
	
	The orientation bundle carries a canonical flat connection of $ TN  $ (induced by the trivial $ T\tilde{N} $-connection on $ \tilde{N}\times\field R $) analogous to \cite[Chapter 7]{stokesfordensities} in the case of vector bundles, and the representation of $ B $ on $ o(F) $ is induced by this connection using the anchor map. 
	
	To define fibre integration we first consider the case $ B=0 $, i.e.\ $ \pi^! B=\foliation{F}(\pi) $ is the Lie algebroid corresponding to the foliation of $ F $ into the fibres of $ \pi $. 
	
	\begin{definition}
		Let $ \pi\colon F\to N $ be a locally trivial fibre bundle of rank $ k $ with orientable and connected fibres. Then fibre integration on $ \Omega^\bullet_\mathrm{cv}(\foliation{F}(\pi)) $ is defined by
		\begin{equation}\label{eq:def_fibre_integration}
		\int_{\foliation{F}(\pi)} \omega =\begin{cases}
		\Big(N\ni p\mapsto \int_{F_p}\iota_p^\ast \omega\Big)\in \Gamma^\infty(o(F)) &\text{ if }\omega\in \Omega^k_\mathrm{cv}(\foliation{F}(\pi))\\
		0 &\text{ otherwise}.
		\end{cases}
		\end{equation}
		Here, $ \iota_p\colon F_p\to F $ denotes the inclusion of the fibre at $ p\in N $.
	\end{definition}
	
	\begin{remark}
		Fibre integration is well defined and coincides with the standard notion of fibre integration in the following sense: Since $ \foliation{F}(\pi)\hookrightarrow TF $ we can choose a vector bundle complement (a horizontal subbundle) $ C $, i.e.\ $ TF=\foliation{F}\oplus C $, to extend foliated forms to forms on $ TF $. The result of the ordinary fibre integration will not depend on the chosen complement (since it is a top-degree form of $ \foliation{F}(\pi) $) and coincides with~\eqref{eq:def_fibre_integration}.
	\end{remark}
	
	\begin{proposition}\label{prop:integration_foliation_cohomology}
		Fibre integration descends to a map
		\begin{equation}\label{eq:integration_foliation_cohomology}
		\int_{\foliation{F}(\pi)} \colon \homology^k_\mathrm{cv}(\foliation{F}(\pi))\to \Gamma^\infty(o(F)).
		\end{equation}
		\begin{proof}
			We only need to check that integration vanishes on exact forms. But, since we have for foliated forms $ \omega\in \Omega^\bullet(\foliation{F}(\pi)) $ that
			\begin{equation*}
			\D_{F_p}\iota^\ast_p \omega=\iota^\ast_p \D_{\foliation{F}} \omega,
			\end{equation*} 
			this follows from Stoke's theorem.
		\end{proof}
	\end{proposition}
	
	Recall that in case of an orientable vector bundle $ F\to N $, ordinary fibre integration of differential forms yields an isomorphism
	\begin{equation}
	\homology_\mathrm{cv}^\bullet(F)\simeq \homology^{\bullet-k}(N),
	\end{equation}
	where the pre-image of $ 1\in \homology^0(N) $ is the \textit{Thom class} $ \theta\in \homology^k_\mathrm{cv}(F) $ (see~\cite[(12.2.1)]{stokesfordensities}). 
 In our setting we obtain a similar statement in Lemma~\ref{lemma:Thom_for_Liealgebroids}, which in the current situation is the following. 
	
	\begin{lemma}\label{lemma:invintegr_foliation}
		Let $ \pi\colon F\to N $ be a vector bundle of rank $ k $. 
Then \begin{equation}\label{eq:cv_cohomology}
    \homology^n_\mathrm{cv}(\foliation{F} (\pi))=\begin{cases}
        \Gamma^\infty(o(F)) &\text{ if }n=k\\
        0 &\text{ otherwise},
    \end{cases}
\end{equation}
where we identify the $\Cinfty(N)$-modules
		\begin{equation}\label{eq:integration_is_iso}
		\int_{\foliation{F}(\pi)} \colon \homology^k_\mathrm{cv}(\foliation{F}(\pi))\xrightarrow{\sim} \Gamma^\infty(o(F))
		\end{equation}
If $ F $ is oriented then for $ 1\in \Cinfty(N)=\Gamma^\infty(o(F)) $ we have
		\begin{equation}
		\left(\int_{\foliation{F}(\pi)}\right)^{-1}(1) = \rho_{\foliation{F}}^\ast \theta,
		\end{equation}
		where $ \theta\in \homology^k_\mathrm{cv}(F) $ denotes the Thom class and $ \rho_{\foliation{F}} $ denotes the anchor of $ \foliation{F}(\pi) $.
		\begin{proof}
			We first show \eqref{eq:integration_is_iso}. Note that integrating is indeed a map of $ \Cinfty(N) $-modules, where the module structure on $ \homology^\bullet_{\mathrm{cv}}(\foliation{F}(\pi)) $ is given by multiplying with pullbacks. 
            The statement in the non-orientable case follows from the orientable one by pulling everything back to a double cover $ \tilde{N}\to N $ which trivialises $ o(F) $. Thus, suppose $ F $ is orientable and let $ \theta\in \homology^k_\mathrm{cv}(F) $ denote the Thom class of the orientable vector bundle. We can pick any representative of the Thom class to compute its integral, and for the computation we only need the contributions that are $ k $-tangent to the foliation since the rest is mapped to zero anyway. But those are given by $ \rho_{\foliation{F}}^\ast\theta $, so we get 
			\begin{equation*}
			\int_{\foliation{F}(\pi)} \rho_{\foliation{F}}^\ast\theta=1\in \Cinfty(N)
			\end{equation*}
			immediately. Since integrating is a module morphism, this implies surjectivity. For injectivity it is enough to argue locally, since we can exploit the $\Cinfty(N)$-module structure and the existence of a partition of unity on $N$. Thus, let $F=N\times \field R^k$ with coordinates $(x,y)$ and 
   \begin{equation*}
    \omega=f \D y_1\wedge \cdots \wedge\D y_k\in \Omega^k_\mathrm{cv}(\foliation F(\pi))
   \end{equation*}
   be given, where $f\in \Cinfty_\mathrm{cv}(F)$. Suppose that $\int_{\foliation{F}(\pi)} \omega=0$, i.e.\ 
   \begin{equation*}
       \int f(x,y)dy=0
   \end{equation*}
   for every $x\in N$. One can adapt the proof of \cite[Lemma 2.4]{gutt.rawnsley:2002a} to show that there exist functions $g_1,\ldots,g_k\in \Cinfty_\mathrm{cv}(F)$ such that 
   \begin{equation*}
       f=\sum_{i=1}^k \frac{\partial g_i}{\partial y_i}.
   \end{equation*}
   Then $\eta\in \Omega^{k-1}_\mathrm{cv}(\foliation F(\pi))$ defined by
   \begin{equation*}
       \eta=\sum_{i=1}^k (-1)^{i+1}g_i \D y_1\wedge\cdots \D y_{i-1}\wedge \D y_{i+1}\wedge\cdots\wedge \D y_k
   \end{equation*}
   is a primitive for $\omega$.

To show that $\homology^n_\mathrm{cv}(\foliation F(\pi))=0$ if $n\neq k$, first note that for $n=0$ the statement is clear. 
For $1\leq n<k$ note that again $\homology^n_\mathrm{cv}(\foliation F(\pi))$ is a $\Cinfty(N)$-module. 
Hence it is enough to show that the cohomology vanishes over a small enough open subsets of $N$, where we utilise the proof of the well-known Poincaré-Lemma for compact support. 
Let $U\subset N$ be relatively compact such that $F$ is trivial over an open neighbourhood of the closure of $U$. 
Let $\omega\in \Omega^n_\mathrm{cv}(\foliation F(\pi))$ be closed. 
By definition of $\Omega^n_\mathrm{cv}(\foliation F(\pi))$ and by compactness of $K$ there exists a compact subset $A\subseteq \field R^k$ such that
   \begin{equation*}
       \supp \omega\at{\pi^{-1}(U)}\subseteq A\times U.
   \end{equation*}
    Denote by $i\colon \field R^k\hookrightarrow \field S^k$ the embedding of $\field R^k$ into the $k$-sphere via the stereographic projection of the north pole $\mathcal{N}$, and by
   \begin{equation*}
       i_\ast\colon \Omega^n_\mathrm{cv}(\foliation (\pi)_{\pi^{-1}(U)})\to \Omega_{\foliation F}^n(\field S^k\times U)
   \end{equation*}
   the extension by $0$ as a foliated form on $S^k\times U\to U$. 
   Then, since $\homology^n_{\foliation F}(\field S^k\times U)=0$ by \cite[Lemma 4.8]{marcut.schuessler:2024a}, and $\D i_\ast \omega=0$, there exists $\eta\in \Omega^{n-1}_{\foliation F}(\field S^k\times U)$ with $\D_{\foliation F} \eta=i_\ast\omega$. 
   Let $O\subset \field S^k\setminus i(A)$ be a contractible neighbourhood of $\mathcal{N}$ and $\chi$ a bump function which is $1$ in a neighbourhood of $\mathcal{N}$ and supported on $O$. On $O\times U$, we have 
   \begin{equation*}\tag{$\ast$}
       \D_{\foliation F} \eta\at{O\times U}=i_\ast\omega\at{O\times U}=0.
   \end{equation*}
   Thus, if $n>1$, there exists $\tilde{\eta}\in \Omega^{n-2}(O\times U)$ with $\D_{\foliation F} \tilde{\eta}=\eta\at{O\times U}$, again by \cite[Lemma 4.8]{marcut.schuessler:2024a}. Then 
   \begin{equation*}
       i^\ast(\eta-\D_{\foliation F} (\chi\tilde{\eta}))\in \Omega_\mathrm{cv}^{n-1}(\field R^k\times U)
   \end{equation*}
   is a well-defined primitive of $\omega$.

   If $n=1$, then $(\ast)$ implies that $\eta\at{O\times U}$ is the pullback of a function $f\in \Cinfty(U)$. Hence, in this case $i^\ast(\eta-\pi_{\field S^k\times U}^\ast f)\in \Cinfty_{\mathrm{cv}}(\field R^k\times U)$ is a primitive of $\omega$.
		\end{proof}
	\end{lemma} 

If the fibres of $\pi\colon F\to N$ are compact, orientable, and connected the situation, we obtain a similar statement for the top degree foliated cohomology.
\begin{lemma}
    Let $\pi\colon F\to N$ be a locally trivial fibre bundle with typical fibre a compact, orientable, and connected manifold of dimension $k$. Then fibre integration yields an isomorphism
\begin{equation}
    \int_{\foliation{F}(\pi)} \colon \homology^k(\foliation{F}(\pi))\xrightarrow{\sim} \Gamma^\infty(o(F))
\end{equation}
    \begin{proof}
        By \cite[Lemma 4.8]{marcut.schuessler:2024a}, the foliated cohomology $\homology^k(\foliation{F} (\pi))$ are sections of a line bundle over $N$ with fibres $\homology^k(\pi^{-1}(x))$ for all $x\in N$, which can be readily identified with $o(F)$ via fibre integration.
    \end{proof}
\end{lemma}

	This concludes the discussion for $ \pi^! B $, where $ B\laover N $ is the zero Lie algebroid. For a general Lie algebroid we can first decompose the forms on $ \pi^!B $ according to their vertical part.
	
	\begin{lemma}\label{lemma:decomp_pi^!A}
		Let $ F\to N $ be a locally trivial fibre bundle and $ B\laover N $ a Lie algebroid. Picking a connection on $ F $ leads to a decomposition
		\begin{equation}\label{eq:decomp_pi^!A}
		\pi^! B = \Ver(F)\oplus \pi^\sharp B
		\end{equation}
		with the property that the anchor is given by the identity on vertical vectors and maps $ \pi^\sharp a\mapsto \rho(a)^\hor $ for $ a\in \Gamma^\infty(B) $, and $ \pi^!\colon \pi^! B\to B $ is given by $ \pi^! =\pi^\sharp \circ (\mathrm{pr}_{\pi^! B \to \pi^\sharp B})$.
		\begin{proof}
			Given a connection on $ E $ and local frames $ \{ s_\alpha \}_\alpha $ of $ F_U $ and $ \{ a_i \}_i $ of $ B_U $ for some open $ U\subseteq N $, the collection
			\begin{equation*}
			\{ (0,s_\alpha^\ver) \}_\alpha \cup \{ (\pi^\sharp a,\rho(a_i)^\hor) \}_i
			\end{equation*}
			yields a local frame for $ \pi^! B_{F_U} $. Then the statements follow immediately.
		\end{proof}
	\end{lemma}
	
	Thus, by choosing a connection on $ F $ we obtain a decomposition
	\begin{equation}
	\Omega_\mathrm{cv}^\bullet(\pi^! B)=\bigoplus_{p+q=\bullet} \Omega^p_\mathrm{cv}(\foliation{F}(\pi),\pi^\sharp\Lambda^q  B^\ast).
	\end{equation}
	This allows to define fibre integration for forms on $ \pi^! B $ by just integrating out the fibre components. More precisely, for $ \omega=\sum_{i,j} \eta_{i,j}\tensor \pi^\sharp \alpha_{i,j} $, where $ \eta_{i,j}\in \Omega^j_\mathrm{cv} (\foliation{F}(\pi)) $ and $ \alpha_{i,j}\in \Omega^\bullet(B) $, we define
	\begin{equation}\label{eq:def_integral_with_tensors}
	(\pi^!)_\ast \Big(\sum_{i,j} \eta_{i,j}\tensor \pi^\sharp \alpha_{i,j}\Big)= \sum_i\Big( \int_{\foliation{F}(\pi)} \eta_{i,k} \Big)\alpha_{i,k}.
	\end{equation}
	It is clear that \eqref{eq:def_integral_with_tensors} yields a map
	\begin{equation}
	(\pi^!)_\ast \colon \Omega_\mathrm{cv}^\bullet(\pi^! B)\to \Omega^{\bullet-k}(B,o(F)).
	\end{equation}
	Since only contributions of $ \Omega^k(\foliation{F}(\pi)) $ (i.e.\ of top degree) matter in computing the integral, it does not depend on the chosen connection, which is also emphasised by the following description. Consider $ \omega\in \Omega^\bullet_\mathrm{cv} (\pi^! B) $ and fix $ p\in N $. Then we can define the $ k $\textit{-fold restriction} $ \omega\at{F_p}^k\in \Omega^k_c( F_p, \Lambda^{\bullet-k}B_p^\ast ) $ of $ \omega $ to $ F_p $ in the following way: For $ X_1,\dots,X_k\in \Gamma^\infty(TF_p) $ we define
	\begin{equation}
	\omega\at{F_p}^k (X_1,\dots,X_k)\colon F_p\to \Lambda^{\bullet-k}(\pi^\sharp B^\ast)\at{F_p}=\Lambda^{\bullet-k}(B^\ast_p)
	\end{equation}
	in the obvious way, as $ \pi^! B / \foliation{F}(\pi) = \pi^\sharp B $. Then one integrates this form along $ F_p $, which yields the same result as~\eqref{eq:def_integral_with_tensors} (note that the integral will vanish if $ \omega $ is not of foliated degree $ k $).
	
	\begin{lemma}
		Let $ F\to N $ be a locally trivial fibre bundle of rank $ k $ with orientable and connected fibres with trivial orientation bundle, and $ B\laover  N $ a Lie algebroid. Then $ (\pi^!)_\ast\colon \Omega^\bullet_\mathrm{cv}(\pi^! B)\to \Omega^{\bullet-k}(B) $ is, up to a sign depending on the degree, a morphism of cochain complexes and induces a map
		\begin{equation}
		(\pi^!)_\ast\colon \homology^\bullet_\mathrm{cv}(\pi^! B)\to \homology^{\bullet-k}(B).
		\end{equation}
	\end{lemma}
	\begin{proof}
Since $ (\pi^!)_\ast $, $ \D_{\pi^! B} $ and $ \D_B $ are local in the sense that to calculate them in a point $p\in N$ we only need information about the form on $ F_U $ and $ U $, respectively, where $U$ is a neighbourhood of $p$. 
Thus, suppose that $\pi\colon F=\tilde{F}\times N\to N \tilde{F}$ is a product bundle. 
Then the canonical flat connection induces a decomposition
\begin{equation*}
    \Omega^n_\mathrm{cv}(\pi^! B)\simeq \bigoplus_{i+j=n} \Omega^i_\mathrm{cv}(\foliation{F}(\pi))\otimes_{\Cinfty(B)}\Omega^j(B).
\end{equation*}
Under this decomposition, we obtain the following.
\begin{enumerate}
    \item By flatness of the connection the differential on $\Omega^n_\mathrm{cv}(\pi^! B)$ splits into
    \begin{equation*}
        \D_{\pi^! B}= \D_{\foliation F(\pi)}\otimes \id +(-1)^j \id\otimes \D_B.
    \end{equation*}
    \item By definition of $(\pi^!)_\ast$ we have for $\omega\in\Omega^\bullet_\mathrm{cv}(\foliation{F}(\pi)) $ and $\alpha\in \Omega^\bullet(B)$ that
    \begin{equation*}
        (\pi^!)_\ast(\omega\tensor (\pi^!)^\alpha)=((\pi^!)_\ast\omega)\otimes (\pi^!)^\ast\alpha.
    \end{equation*}
    \item By Lemma \ref{lemma:invintegr_foliation} we have
    \begin{equation*}
        (\pi^!)_\ast \D_{\foliation F(\pi)}\Omega^\bullet_\mathrm{cv}(\foliation F(\pi))=0.
    \end{equation*}
\end{enumerate}
Together, this implies the statement.
\end{proof}

	With the fibre integration on cohomology defined, we get a Thom isomorphism for Lie algebroids as a consequence of Lemma~\ref{lemma:invintegr_foliation}.
	
	\begin{lemma}[Thom isomorphism for Lie algebroids]\label{lemma:Thom_for_Liealgebroids}
		Let $ \pi\colon E\to N $ be a vector bundle of rank $ k $ with orientation bundle $ o(E) $ and $ B\laover N $ a Lie algebroid. Then fibre integration
		\begin{equation}
		(\pi^!)_\ast\colon \homology_\mathrm{cv}^\bullet(\pi^!B)\to \homology^{\bullet-k}(B,o(E))
		\end{equation}
		is an isomorphism. If $ \theta\in \homology^k_\mathrm{cv}(E,\pi^\sharp o(E)^\ast) $ denotes the Thom class of the vector bundle, the inverse of $ (\pi^!)_\ast $ is given by
		\begin{equation}
		\homology^{\bullet-k}(B,o(E))\ni \omega\mapsto \rho_{\pi^! B}^\ast\theta\wedge (\pi^!)^\ast\omega\in \homology_\mathrm{cv}^\bullet(\pi^! B).
		\end{equation}
	\end{lemma}
 \begin{proof}
     The statement follows from a spectral sequence argument. Consider the filtration on $\Omega^\bullet_\mathrm{cv}(\pi^!B)$ given by
     \begin{equation*}
         \mathcal{F}^p\Omega^\bullet_\mathrm{cv}(\pi^!B)=(\Omega_\mathrm{cv}(\foliation{F} (\pi))^{\wedge p}\wedge \Omega_\mathrm{cv}(\pi^!B))^\bullet
     \end{equation*}
     And the induced spectral sequence $\{ E_r^{p,q}\}_{r\geq 0}$. By standard arguments we obtain
\begin{equation*}
    E_1^{p,q}\simeq \homology^q(B,\homology^p_\mathrm{cv}(\filtration{F} (\pi))).
\end{equation*}
     Then by Lemma \ref{lemma:invintegr_foliation} the statement follows.
 \end{proof}

	\subsection{The Gysin sequence}
	
	 We prove a Gysin-like sequence, which we have used in Section \ref{section:blow_up_of_transversals}, for the cohomology of $ \pi^! B $, where $ B\laover N $ is a Lie algebroid and $ \pi\colon \field S\to N $ is a sphere bundle of rank $ k $, i.e.\ a locally trivial fibre bundle with a $ k $-dimensional sphere as typical fibre. In this case, by the Serre spectral sequence for fibre bundles \cite[Theorem 5.7]{marcut.schuessler:2024a} we obtain a spectral sequence converging to the Lie algebroid cohomology of $\pi^! B$ with second page given by
	\begin{equation}
	E_2^{p,q}=\Homology{B}{p}{,\mathcal{H}^q(\field S)}.
	\end{equation}
 Here, $\mathcal{H}^q(\field S)\to N$ denotes a smooth vector bundle with fibres given by
 $\mathcal{H}^q(\field S)_x=\mathrm{H}^q(\pi^{-1}(x))$ for $x\in N$. 
 Thus, $E_2^{p,q}$ has nontrivial entries only if $ q=0 $ or $ q=k $. 
 Therefore the next (and last) nontrivial differential is $ \D_{k+1} $. Recall that in case of $ B=TN $ and trivial orientation bundle it is given by $ \D_{k+1}=\wedge e $, where $ e\in \Homology{N}{k+1}{} $ is the Euler class of the sphere bundle~\cite[Chapter 11]{stokesfordensities}. If $ o(\field S) $ is nontrivial, we can consider the pullback to a trivialising double cover $ \tilde{N}\to N $ and find that $ e\in \Homology{\tilde{N}}{k+1}{}_-=\Homology{N}{k+1}{,o(\field S)} $ instead.

	\begin{theorem}\label{theorem:gysin_sequence_for_LA}
		Let $ B\laover N $ be a Lie algebroid with anchor $ \rho $ and $ \pi\colon \field S\to N $ a sphere bundle of rank $ k $. Then there is a long exact sequence
		\begin{equation}
		\begin{tikzpicture}[baseline=(current bounding box.center)]
		\node (1) at (0,0) {$ \dots $};
		\node (2) at (1.7,0) {$ \homology^\bullet(B) $};
		\node (3) at (4,0) {$ \homology^\bullet( \pi^! B ) $};
		\node (4) at (7,0) {$ \homology^{\bullet-k}(B,o(\field S)) $};
		\node (5) at (10.1,0) {$ \homology^{\bullet+1}(B) $};
		\node (6) at (12,0) {$ \dots $};
		
		\path[->]
		(1) edge node[above]{$  $} (2)
		(2) edge node[above]{$ (\pi^!)^\ast $} (3)
		(3) edge node[above]{$ (\pi^!)_\ast $} (4)
		(4) edge node[above]{$ \wedge \rho^\ast e $} (5)
		(5) edge node[above]{$  $} (6)
		
		;
		\end{tikzpicture}
		\end{equation} 
		Here, $ (\pi^!)_\ast  $ denotes fibre integration and $ e\in \homology^{k+1}(N,o(\field S)) $ is the Euler class of the sphere bundle.
		\begin{proof}
			If $ o(\field S) $ is nontrivial we can pull everything back to a trivialising double cover. Using Lemma~\ref{lemma:lapullback_cohomology} and noting that integration and $ \wedge e $ map $ \field Z_2 $-invariant classes to anti-invariant ones and vice versa, one then obtains the result in general. Thus, let $ o(\field S) $ be trivial. The differential $\D_{k+1}\colon E_2^{0,k}\to E_2^{k+1,0} $ can be computed by evaluating on a generator $ s $. We can write $ s=\sum_i \chi_i [\omega_i] $, where $ \omega_i\in \Omega^k(\field S_{U_i}) $ are closed with respect to the de Rham differential and $ \chi_i\in \Cinfty(N) $. Let $ \gamma\colon \pi^\sharp B\to \pi^! B $ be a splitting, $\omega=\sum_i \chi_i\omega_i$, and $ a_0,\dots,a_k\in \Gamma^\infty(B) $. Using that for $ B=TN $ we have $ \D_{\mathrm{dR}}\omega=\pi^\ast e$, by a standard calculation we obtain 
			\begin{equation*}
			\begin{aligned}
			\D_{k+1} \omega( \gamma&\pi^\sharp a_0,\dots,\gamma\pi^\sharp a_k )= 
			\sum_{i=0}^k (-1)^i \rho_{\pi^! A}(\gamma \pi^\sharp a_i)\omega(  \gamma\pi^\sharp a_0,\dots,\stackrel{i}{\wedge},\dots,\gamma\pi^\sharp a_k )
			\\&\phantom{\pi^\sharp a_0,\dots,\gamma\pi^\sharp a_k )=}+\sum_{\mathclap{0\leq i< j\leq k}}(-1)^{i+j}\omega( [\gamma\pi^\sharp a_i,\gamma\pi^\sharp a_j],  \gamma\pi^\sharp a_0,\dots,\stackrel{i}{\wedge},\dots,\stackrel{j}{\wedge},\dots,\gamma\pi^\sharp a_k ) \\
			=& \sum_{i=0}^k (-1)^i \rho_{\pi^! A}(\gamma \pi^\sharp a_i)\omega(  \rho_{\pi^! A}(\gamma\pi^\sharp a_0),\dots,\stackrel{i}{\wedge},\dots,\rho_{\pi^! A}(\gamma\pi^\sharp a_k) )
			\\&+\sum_{\mathclap{0\leq i< j\leq k}}(-1)^{i+j}\omega( \rho_{\pi^! A}([\gamma\pi^\sharp a_i,\gamma\pi^\sharp a_j]),  \rho_{\pi^! A}(\gamma\pi^\sharp a_0),\dots,\stackrel{i}{\wedge},\dots,\stackrel{j}{\wedge},\dots,\rho_{\pi^! A}(\gamma\pi^\sharp a_k) )\\
			=& \D_\mathrm{dR}\omega ( \rho_{\pi^! A}(\gamma\pi^\sharp a_0),\dots, \rho_{\pi^! A}(\gamma\pi^\sharp a_k) )\\
			=& \pi^\ast e (\rho_{\pi^! A}(\gamma\pi^\sharp a_0),\dots, \rho_{\pi^! A}(\gamma\pi^\sharp a_k))\\
			=&\pi^\ast ( \rho_A^\ast e ( a_0,\dots,a_k ) ).
			\end{aligned}
			\end{equation*}
			The rest then follows analogously to the proof of \cite[Proposition 14.33]{stokesfordensities}.
		\end{proof}
	\end{theorem}
	
	If $ \field S\to N $ actually comes from a vector bundle $ \pr\colon V\to N $, we can rewrite the Gysin sequence using the Thom isomorphism on $ \homology^{\bullet-(k-1)}(B,o(\field S)) $.
	
	\begin{lemma}\label{lemma:thom_euler_diagram}
		Let $ B\laover N $ be a Lie algebroid with anchor $ \rho $ and $ \pr\colon V\to N $ a vector bundle of rank $ k $. Then, under the isomorphism $ \homology^\bullet(B)=\homology^\bullet(\pr^! B) $ induced by $ (\iota^!)^\ast $, the diagram
		\begin{equation*}
		\begin{tikzpicture}[baseline=(current bounding box.center)]
		\node (o1) at (0,2) {$ \homology^{\bullet-k}(B,o(\field S)) $};
		\node (o2) at (3,2) {$\homology^\bullet(B)  $};
		\node (u1) at (0,0) {$ \homology_\mathrm{cv}^\bullet(\pr^! B) $};
		\node (u2) at (3,0) {$ \homology^\bullet(\pr^! B) $};
		
		\path[->]
		(o1) edge node[above]{$  \rho^\ast e\wedge $} (o2)
		(o1) edge node[left]{$ \Phi $} node[right]{$ \cong $} (u1)
		(u1) edge node[above]{$ i $} (u2)
		(u2) edge node[right]{$ (\iota^!)^\ast $} node[left]{$ \cong $} (o2)
		
		;
		\end{tikzpicture}
		\end{equation*}
		commutes, where $ \Phi\colon \homology^{\bullet-k}(B,o(\field S))\to \homology^\bullet_\mathrm{cv}(\pr^! B) $ denotes the Thom isomorphism from Lemma \ref{lemma:Thom_for_Liealgebroids}, $ \iota^!\colon B\to \pr^! B $ is the inclusion of Lie algebroids over the zero section $ \iota\colon N\to V $, and $ i $ denotes the natural map of considering a compact vertical form as just a form on $ \pr^! B $.
		\begin{proof}
			First note that $ o(V)=o(\field S) $ since the sphere bundle is associated to $ V $ \cite[Proposition 11.2]{stokesfordensities}. Again, we will only proof the statement for orientable vector bundles. Writing $ \theta\in \homology_\mathrm{cv}^k(V) $ for the Thom class, for any $ \omega\in \homology^\bullet(B) $ we have
			\begin{equation*}
			(\iota^!)^\ast \Phi(\omega)= (\iota^!)^\ast \rho_{\pr^! B}^\ast \theta \wedge \omega = (\rho_{\pr^! B}\circ \iota ^!)^\ast \theta\wedge \omega = (T\iota\circ \rho)^\ast \theta\wedge \omega = \rho^\ast e\wedge \omega
			\end{equation*}
			as the pullback of $ \theta $ by the zero section is the Euler class.
		\end{proof}
	\end{lemma}
	
	\begin{corollary}\label{Cor:gysin_for_LA_withcompactsupport}
		Let $ B\laover N $ be a Lie algebroid, and $ \pr\colon V\to N $ a vector bundle with induced sphere bundle $ \pi\colon \field S\to N $. Then there is a long exact sequence
		\begin{equation}
		\begin{tikzpicture}[baseline=(current bounding box.center)]
		\node (1) at (0,0) {$ \dots $};
		\node (2) at (2,0) {$ \homology^\bullet(\pr^!B) $};
		\node (3) at (4.4,0) {$ \homology^\bullet( \pi^! B ) $};
		\node (4) at (7,0) {$ \homology^{\bullet+1}_\mathrm{cv}(\pr^!B) $};
		\node (5) at (9.8,0) {$ \homology^{\bullet+1}(\pr^!B) $};
		\node (6) at (11.9,0) {$ \dots $};
		
		\path[->]
		(1) edge node[above]{$  $} (2)
		(2) edge node[above]{$ (\pi^!)^\ast $} (3)
		(3) edge node[above]{$ (\pi^!)_\ast $} (4)
		(4) edge node[above]{$ i $} (5)
		(5) edge node[above]{$  $} (6)

		;
		\end{tikzpicture}
		\end{equation}
	\end{corollary}

	
	
	%
	%

\bibliography{literature}
	\bibliographystyle{alpha}
	
\Addresses
\end{document}